\documentclass{article}

\usepackage{arxiv}

\RequirePackage{amsthm,amsmath,amsfonts,amssymb,dsfont, amstext}
\RequirePackage[numbers]{natbib}
\usepackage{xr-hyper}
\usepackage{xcolor}
\RequirePackage[colorlinks,citecolor=blue,urlcolor=blue]{hyperref}

\RequirePackage[utf8]{inputenc}
\RequirePackage[OT1]{fontenc}

\RequirePackage{graphicx}

\RequirePackage{color}
\RequirePackage{verbatim}
\RequirePackage{enumitem}    
\usepackage{filecontents}

\usepackage{mathrsfs} 
\usepackage{bm}
\usepackage{changepage}

\usepackage[sectionbib]{bibunits}
\usepackage{nameref}
\usepackage[nameinlink,noabbrev,capitalise]{cleveref}
\usepackage[normalem]{ulem}

\theoremstyle{plain}

\newtheorem{theorem}{Theorem}[section]
\newtheorem{lemma}[theorem]{Lemma}
\newtheorem{corollary}[theorem]{Corollary}
\newtheorem{proposition}[theorem]{Proposition}
\theoremstyle{remark}

\newtheorem{example}{Example}

\newtheorem{remark}[theorem]{Remark}
\newtheorem{assumption}[theorem]{Assumption}
\newcommand{\ds}[1]{\textcolor{black}{#1}}

\usepackage{color}
\usepackage{verbatim}
\usepackage{algorithmic}
\usepackage{enumitem}
\usepackage[varg]{txfonts}    

\renewcommand{\mathbb}{\varmathbb} 

\renewcommand{\leq}{\leqslant}

\renewcommand{\geq}{\geqslant}

\usepackage{bm}

\usepackage{xspace}

\usepackage{boxedminipage}

\newcommand{\brac}[1]{[#1 ]}
\newcommand{\Brac}[1]{\left[#1\right]}

\newcommand{\norm}[1]{\left\lVert#1\right\rVert}

\newcommand{\Z}{{\mathbb Z}}
\newcommand{\N}{{\mathbb N}}
\newcommand{\R}{\mathbb R}

\newcommand{\Esymb}{\mathbb{E}}
\newcommand{\Psymb}{\mathbb{P}}

\DeclareMathOperator*{\E}{\Esymb}

\DeclareMathOperator*{\ProbOp}{\Psymb}

\newcommand{\Prob}[1]{\ProbOp\Brac{#1}}
\newcommand{\probz}[1]{\mathbb{P}_0\brac{#1}}
\newcommand{\Probz}[1]{\mathbb{P}_0\Brac{#1}}

\newcommand{\Probf}[1]{\mathbb{P}_f\Brac{#1}}

\newcommand{\Ex}[1]{\E\Brac{#1}}

\newcommand{\Exz}[1]{\mathbb{E}_0\Brac{#1}}

\newcommand{\Exf}[1]{\mathbb{E}_f\Brac{#1}}

\newcommand{\indep}{\perp \!\!\! \perp}

\newcommand{\e}{\epsilon}

\let\e\epsilon

\newcommand{\eve}{\Tilde{\Omega}_T}

\newcommand{\cste}{C}

\PassOptionsToPackage{unicode}{hyperref}
\PassOptionsToPackage{naturalnames}{hyperref}

\makeatletter

\newcommand*{\addFileDependency}[1]{
\typeout{(#1)}
\@addtofilelist{#1}
\IfFileExists{#1}{}{\typeout{No file #1.}}
}\makeatother

\newcommand*{\myexternaldocument}[1]{%
\externaldocument{#1}%
\addFileDependency{#1.tex}%
\addFileDependency{#1.aux}%
}

\myexternaldocument{supplement}

\title{Bayesian estimation of nonlinear Hawkes processes}

\author{
 D\'eborah Sulem \\ 
 Department of Statistics\\
 University of Oxford\\
 \texttt{deborah.sulem@stats.ox.ac.uk} \\
  \And
  Vincent Rivoirard \\
  Ceremade, CMRS, UMR 7534 \\
  Universit\'e Paris-Dauphine, PSL University \\
 \texttt{vincent.rivoirard@dauphine.fr}
   \And
  Judith Rousseau \\ 
 Department of Statistics\\
 University of Oxford\\
 \texttt{judith.rousseau@stats.ox.ac.uk} \\
}

\begin{document}

\maketitle

\begin{abstract}

Multivariate point processes (MPPs) are widely applied to model the occurrences of events, e.g., natural disasters, online message exchanges, financial transactions or neuronal spike trains. In the Hawkes process model, the probability of occurrences of future events depend on the past of the process. This model is particularly popular for modelling interactive phenomena such as disease expansion.
In  this  work  we  consider  the nonlinear multivariate Hawkes model, which allows to account for \emph{excitation} and \emph{inhibition} between interacting entities. We provide theoretical guarantees for applying nonparametric Bayesian estimation methods in this context. In particular, we obtain concentration rates of the posterior distribution on the parameters, under mild assumptions on the prior distribution and the model. These results also lead to convergence rates of Bayesian estimators. Another object of interest in event-data modelling is to infer the \textit{graph of interaction} - or Granger causal graph. In this case, we provide consistency guarantees; in particular, we prove that the posterior distribution is consistent on the graph adjacency matrix of the process, as well as a Bayesian estimator based on an adequate loss function.\
\end{abstract}





\section{Introduction}

\subsection{Nonlinear Hawkes processes}\label{section:def_hawkes}

The Hawkes model is a popular temporal point process (PP)  for modelling the occurrences of event-type phenomena. Extending the Poisson cluster process \cite{moller2005perfect}, this model allows the probability of occurrence of a new event to depend on the history of the process. The first construction by Hawkes \cite{hawkes71} \ds{aimed at modelling} the \textit{self-excitatory} behaviour of earthquakes' strikes with aftershocks, and is called the 
  \emph{linear Hawkes process}. Since then, it has been extensively used, partly due to its interpretable parameters and branching structure representation \cite{reynaud06}. This notably leads to tractable inference and simulation methods \cite{bacry2015sparse, Chen2017,  Hansen:Reynaud:Rivoirard}.

Hawkes processes have been largely and successfully applied in various contexts of correlated event-data, including online social popularity \cite{farajtabar2016coevolve}, stock prices moves \cite{embrechts2011multivariate}, topic modelling \cite{du15}, DNA motifs occurrences \cite{carstensen2010multivariate, gusto:hal-02682109,  Reynaud_Bouret_2010}, and neuronal activity modelling \cite{MR961117, JMM, MR3197017}. They are used to infer both diffusion phenomena on networks and the structure of time-dependent networks \cite{Miscouridou2018}. Related and extended models include the mutually-regressive PP \cite{apostolopoulou2019mutually}, the age-dependent \cite{raad2020stability} and marked \cite{karabash2015limit} Hawkes processes, the dynamic contagion process 
\cite{dassios_zhao_2011}, the reactive PP \cite{ertekin2015}, the self-correcting PP \cite{Isham1979} and the Dirichlet-Hawkes process \cite{du15}. More recently, neural point processes inspired by the Hawkes model have also been proposed \cite{du2016recurrent, mei2017neural}.

In a multivariate temporal PP,  each dimension represents an entity, a location or a type of event - it is equivalent to a \emph{marked} point process with finite mark space. For $K \in \mathbb{N} \backslash \{0\}$, the PP can be described as a counting process $N = (N_t)_t = (N_t^1, \dots, N_t^K)_{t \geq 0}$, where $N_t^k$ denotes the number of events that have occurred until time $t$ at location $k$. Its dynamics are characterised by a conditional intensity function $(\lambda_t)_t = (\lambda_t^1, \dots, \lambda_t^K)_{t\geq 0}$, which is informally the infinitesimal rate of event conditionally on the past of the process, i.e, for $ k=1,\dots,K$,
$
\lambda_t^k dt = \Prob{N_t^k \: \text{has a jump in } \: [t,t+dt]|\mathcal{G}_{t}}, 
$
where $\mathcal{G}_{t}$ is the history of the process up to time $t$. In the nonlinear Hawkes model, only one dimension $N^k$ of the process can jump at each time $t$ and the intensity process has the following form
\begin{equation}\label{def:NLintensity}
    \lambda_t^k = \phi_k \left( \nu_k + \sum_{l=1}^{K} \int_{-\infty}^{t^-} h_{lk}(t-s)dN_s^l \right), \quad k=1,\dots,K.
\end{equation}
In \eqref{def:NLintensity}, the parameter $\nu_k > 0$ denotes the \textit{background} - or \textit{spontaneous} - rate of events, and models exogeneous influences. The endogenous effects on the process are parametrised by \textit{interaction functions} $(h_{lk})_{l,k=1}^K$ - or \emph{triggering kernels}. More precisely, for $(l,k) \in [K]^2$, the function $h_{lk}: \R \to \R$ models the influence of component $N^l$ onto component $N^k$. It can be decomposed into an \emph{excitating} contribution 
($h_{lk}^+ = \max(h_{lk},0)$) and an \emph{inhibiting} contribution
($h_{lk}^- = \max(-h_{lk},0)$). Finally, the \textit{link} or \emph{activation function} $\phi_k: \mathbb{R}\to \mathbb{R}^+$ ensures that the intensity is a non-negative process, and is generally chosen to be monotone non-decreasing. If all the interaction functions $h_{lk}$ are non-negative and all the link functions equal the identity functions, \eqref{def:NLintensity} corresponds to the linear Hawkes model. 

The dependence on past events in the intensity \eqref{def:NLintensity} leads to a notion of \emph{causality}. For Hawkes processes,  a Granger-causal relationship between two components of the process corresponds to a non-null interaction function \cite{eichler2016graphical}. We can define the \emph{connectivity graph} parameter $\delta \in \{0,1\}^{K^2}$ such that for each $(l,k)$, $\delta_{lk}=1$ if the function $h_{lk}$ in \eqref{def:NLintensity} is non null and $\delta_{lk}=0$ otherwise. \ds{We note that this parameter is redundant with $(h_{lk})_{l,k=1}^K$.}



To the best of our knowledge, the estimation of the parameters of nonlinear Hawkes processes $\nu =(\nu_k)_k, \: h=(h_{lk})_{l,k=1}^K, \: \delta = (\delta_{lk})_{l,k=1}^K$ - as well as additional parameters of the link functions $(\phi_k)_k$ has not been theoretically analysed, neither in the frequentist nor in the Bayesian frameworks. In the nonparametric setting, the existing results apply to linear Hawkes processes for the estimation of $(\nu, h)$ \cite{donnet18} and for the estimation of the connectivity graph $\delta$ \cite{Hansen:Reynaud:Rivoirard, Chen2017}. In the nonlinear model, \cite{chen17b}  study the estimation of the cross-covariances of the process, and \cite{wang2016isotonic} estimate a piecewise-constant link function assuming a parametric form on the interaction functions.

In this work, we analyse the theoretical properties of Bayesian methods for estimating $\nu$, $h$, $\delta$ and additional parameters of the nonlinear functions $(\phi_k)_k$. We consider a prior distribution on the parameters, say $\Pi$, and our aim is to study  posterior concentration rates  in such models. More precisely, we wish to determine $\epsilon_T = o(1)$ and conditions on the model and on $\Pi$ such that 
$$ \mathbb E_{f_0} [\Pi( d( f, f_0 ) > \epsilon_T | N )] \xrightarrow[T \to \infty]{} 1,$$
where $f = (\nu, h)$, $d(.,.)$ is some loss function on the parameter space,  and $\Pi(.| N)$ denotes the posterior distribution given an observation of the process on $[0,T]$. 
 \ds{In the last equation, we assume that the data $N$ is generated by a Hawkes process with \emph{true parameter}$f_0$, and we denote ${\mathbb P}_{f_0}$ its generating distribution and $\E_{f_0}$  the associated expectation.}
In particular, a consequence of such result is the construction of estimators on $\nu, h$ which converge in the frequentist sense at the rate $\epsilon_T$. \ds{We also obtain posterior consistency results on the graph parameter  $\delta$, and construct a consistent risk-minimising estimator.} 

\subsection{Related works}


There is  a rich literature on Hawkes processes in probability, statistics, and more recently in machine learning and deep learning. The stability properties of the nonlinear Hawkes model have been studied under several assumptions \cite{bremaud96, karabash2013stability}, together with the rate of convergence to the stationary solution \cite{bremaud2002rate} and the Bartlett spectrum \cite{massoulie98}. Regenerative properties of Hawkes processes were investigated for the models with finite \cite{costa18} and infinite \cite{Graham2019, Raad2019} memory. Recently \cite{BacrySPA2013, gao2017asymptotic, Gao_2018} derived functional central limit theorems and large deviations principles for ergodic processes. Malliavin-Stein calculus was applied by \cite{torrisi2016, torrisi2017} to establish Gaussian and Poisson approximations of functionals of the linear Hawkes process, and later by \cite{hillairet2021} to obtain Berry-Ess\'een bounds. Stationary distributions of high dimensional Hawkes processes were also studied, notably in the mean-field limit \cite{Delattre2016, delattre2016b, raad2020stability}.

Many statistical works have been dedicated to designing robust and efficient estimation procedures in the linear Hawkes model. In the seminal work of \cite{ogata1988statistical}, the interaction functions are given in a parametric form and estimated by maximising the likelihood function. In parametric models, an Expectation-Maximisation algorithm was proposed in \cite{veen2008estimation} to compute the maximum likelihood estimator while MCMC methods were designed for sampling from the posterior distribution \cite{Rasmussen2013}. The EM algorithm was extended by \cite{lewis2011nonparametric} to nonparametric Hawkes models using a penalised likelihood objective. Another nonparametric approach was introduced by \cite{Reynaud_Bouret_2010} for the linear univariate model by using a model selection strategy. In the multivariate Hawkes model, Lasso-type estimates were designed by \cite{Hansen:Reynaud:Rivoirard}. Still for linear models, Bayesian approaches have also been implemented for nonparametric Hawkes models, see for instance \cite{du15}.
In \cite{donnet18} the authors study asymptotic properties of the posterior distribution in the linear model.

Causality graphs for discrete-time events were introduced by  \cite{granger1969investigating} and extended to marked point processes by  \cite{Didelez2007}, with an explicit definition in the case of multivariate Hawkes processes by \cite{eichler2016graphical}. 
In linear parametric models, some approaches optimise a least-square objective based on the intensity process \cite{bacry2015sparse, BacrySPA2013}. For nonparametric Hawkes processes, \cite{Xu2016} 
apply an EM algorithm based on a penalized likelihood objective leading to temporal and group sparsity. Still in the linear model, Lasso-type estimates proposed by \cite{Hansen:Reynaud:Rivoirard} for nonparametric Hawkes processes naturally lead to sparse connectivity graphs. This procedure has been generalised to high-dimensional processes by \cite{Chen2017} by adding an edge screening step.


\subsection{Our contributions}\label{sec:contributions}
This paper considers a general multivariate Hawkes model with a nonlinear and nonparametric form of the intensity function,
 and provides theoretical guarantees on Bayesian estimation methods. We cover  a large range of  link functions $\phi_k$, 
 which covers most of the nonlinear Hawkes models considered in the literature \cite{costa18, Hansen:Reynaud:Rivoirard, gerhard17, carstensen2010multivariate, chen17b, Menon2018ProperLF, mei2017neural, deutsch2022, truccolo2005}, 
 such as the ReLU $\phi_k(x)=(x)_+=\max(x,0)$, clipped exponentials $\phi_k(x)=\min(e^x, \Lambda)$, the sigmoid  $\phi_k(x)=(1 + e^{- x})^{-1}$, and the softplus $\phi_k(x)=\log(1 + e^x)$. 
These models have been notably introduced for neuronal spike-train data modelling, where intense-activity periods alternate with resting states called \textit{refractory periods}\footnote{A refractory period is a time interval during which a neuron is unlikely to emit a spike train.}.  The ReLU function directly extends the original linear Hawkes model to handle negative interaction functions. In  \cite{Hansen:Reynaud:Rivoirard, costa18} it is called the \emph{standard} nonlinear Hawkes model, as it is the closest to the linear Hawkes process. Exponential and sigmoidal functions appear in several applied works \cite{gerhard17, carstensen2010multivariate}, where smoothness, saturation and thresholding effects are desirable properties. The softplus function is often preferred in machine learning algorithms as a soft approximation of ReLU \cite{mei2017neural}.

The first question to answer is the identifiability of $f = (\nu, h)$,   which is treated in Section \ref{sec:identifiability}. Building on these results, we study posterior concentration rates  in terms of the $L_1$-norm on $f$ in Section \ref{sec:post_conc}. Our aim is to describe the posterior concentration rates in terms of conditions on the prior $\Pi$ and on the true parameter $f_0 = (\nu_0 = (\nu_k)_k, h_0 = (h_{lk}^0)_{l,k})$ which are  simple to verify and under rather weak assumptions on the link functions.  Interestingly, we eventually reduce the problem to conditions on the prior and the $f_0$ similar to those found in the literature on density and nonlinear regression estimation (see Theorem \ref{thm:conc_g}), which makes them easy to verify in a wide range of prior models. From this we derive convergence rates of Bayesian estimators of $f_0$  (Corollary \ref{cor:post_mean})  and posterior consistency on $\delta_0$ (Theorem \ref{thm:post_graph}), with $\delta_0$ the true graph parameter associated to $h_0$.

We also extend our results to the case where the link functions are partially unknown, in the special case of shifted ReLU  link functions. More precisely we consider models in the form 
$\phi_k(x) = \theta_k + (x)_+$, with $\theta_k>0$ unknown. For such models we show identifiability of the parameters $(f, \theta), \: \theta = (\theta_k)_{1\leq k\leq K}$ and derive a general posterior concentration rate result similar to Theorem \ref{thm:conc_g} on both $f$ and $\theta$.


\ds{To the best of  our knowledge, these results are the first theoretical properties on the nonparametric estimation of both $f_0$ and $\delta_0$ in the frequentist and Bayesian literature of nonlinear Hawkes processes. Besides, for partially known link functions, in the particular setting of the shifted ReLU model, we also provide the first result on the estimation of the additional parameter $\theta$. We note that recently, computational methods for a related setting have been developed in  \cite{zhou2021nonlinear, zhou2021efficient, malemshinitski2021nonlinear, zhou2021jmlr}. In the latter works, a sigmoidal nonlinear Hawkes model is defined with $\phi_k(x) = \theta_k (1 + e^{-x})^{-1}$ and unknown parameter $\theta = (\theta_k)_k$. However, although the theoretical analysis of the latter model 
is beyond the scope of this paper, it is similar in spirit to our models. In fact, our techniques could potentially be applied to this multiplicative parametrisation, which we leave for future work.}

\ds{Our results are related to those of \cite{donnet18}, obtained in the case of linear Hawkes processes. However, the analysis of the process and  our proofs for estimating the parameter rely on \emph{renewal} properties, newly introduced by \cite{costa18} in the univariate ReLU nonlinear Hawkes model. One key novelty of our work is to leverage the concept of \emph{excursions} in the context of statistical analysis. This concept allows to decompose the trajectory of the process into independent, observable subintervals, and also to analyse the process on specific events where the parameter estimation is simplified. Developing these tools for nonlinear processes is fundamental since classical technical arguments used for linear Hawkes processes and based on Poisson branching structures cannot be applied in this case. We believe that these new proof techniques have an interest in themselves, in addition to  weakening some of the assumptions on the prior distribution considered in  \cite{donnet18}.}


The rest of the paper is organised as follows. In Section \ref{sec:problem}, we define the multivariate  stationary nonlinear Hawkes process, present the identifiability results and describe the Bayesian framework. Section \ref{sec:main_results} presents the posterior concentration results on $f$ and $\theta$ and consistency on $\delta$ results. Section~\ref{sec:priors} is dedicated to the construction of prior distributions that satisfy the assumptions of the theorems. The most novel aspects of the proofs are reported in Section \ref{sec:proof}. Appendix \ref{app:main_lemmas} contains some technical lemmas. Finally, supplementary proofs and results can be found in the supplementary material \cite{supplementary}.

\textbf{Notations. } For a function $h$, we denote $\norm{h}_1 = \int_{\mathbb R} |h(x)|dx$ the $L_1$-norm, $\norm{h}_2 = \sqrt{\int_{\mathbb R}h^2(x)dx}$ the $L_2$-norm, $\norm{h}_\infty = \sup \limits_{x\in \mathbb R} |h(x)|$ the supremum norm, and $h^+ = max(h,0), \: h^- = max(-h,0)$ its positive and negative parts. For a $K \times K$ matrix $A$, we denote $r(A)$ its spectral radius and $\norm{A}$ its spectral norm. 
For a vector $u \in \R^K, \norm{u}_1 = \sum_{k=1}^K |u_k|$. The notation $k \in [K]$ is used for $k \in \{ 1, \cdots, K\}$. For a set $B$ and $k \in [K]$, we denote $N^k(B)$ the number of events of $N^k$ in $B$ and $N^k|_B$ the point process measure restricted to the set $B$. For random processes, the notation $ \overset{\mathcal{L}}{=}$ corresponds to equality in distribution. 
We also denote $\mathcal{N}(u, \mathcal{H}_0,d)$ the covering number of a set $\mathcal{H}_0$ by balls of radius $u$ w.r.t. a metric $d$. For any $k \in [K]$, let $\mu_k^0 = \mathbb{E}_0[\lambda_t^k(f_0)]$ be the mean of $\lambda_t^k(f_0)$ under the stationary distribution $\mathbb{P}_0$. For a set $\Omega$, its complement is denoted $\Omega^c$. We also use the notations $u_T \lesssim v_T$ if $|u_T/v_T|$ is bounded when $T \to \infty$, $u_T \gtrsim v_T$ if $|v_T/u_T|$ is bounded and $u_T \asymp v_T$ if $|u_T/v_T|$ and $|v_T/u_T|$ are bounded.




\section{Problem setup}\label{sec:problem}

\subsection{Definition and stationary distribution}

In this section, we first recall the  formal definition of  a multivariate Hawkes process. We consider a probability space $(\mathcal{X}, \mathcal{G}, \mathbb{P})$ and a  MPP $N = (N_t)_{t \in \R} = (N_t^1, \dots, N_t^K)_{t \in \R}$. Let $\{\mathcal{G}_t\}_{t \in \R}$ be the filtration such that $\mathcal{G}_t = \sigma(N_s, s \leq t)$ and for $T>0$, we assume that $\mathcal{G}_T \subset \mathcal{G}$. We say that $(N_t)_t$ is a multivariate Hawkes process with parameter $f = ((\nu_k)_{k=1}^K, (h_{lk})_{l,k=1}^K, (\theta_k)_{k=1}^K)$ adapted to $\mathcal{G}$ if
\begin{enumerate}[label={\roman*)}]
    
    \item almost surely, $\forall k,l \in [K]$, $(N^k_t)_t$ and $(N^l_t)_t$ never jump simultaneously;
    
    \item for all $k \in [K]$, the $\mathcal{G}_t$-predictable intensity process of $N^k$ at $t \in \R$ is 
     given by
    \begin{equation*}
    \lambda_t^k(f) = \phi_k \left( \nu_k + \sum_{l=1}^{K} \int_{-\infty}^{t^-} h_{lk}(t-s)dN_s^l \right), \quad k=1,\dots,K.
\end{equation*}

\end{enumerate}
We consider finite-memory Hawkes processes for which interaction functions have a bounded support included in $[0,A]$ with $A > 0$ known - chosen arbitrarily large in practice. We recall that in \eqref{def:NLintensity}, if for all $k$, $\phi_k$ is the identity function and for all $l$, $h_{lk}$ is non-negative, this PP model corresponds to the classical linear Hawkes process with parameter $\nu = (\nu_k)_{k=1}^K$ and $h = (h_{lk})_{k,l=1}^K$ and intensity process:
    \begin{equation}\label{eq:lin_intensity}
        \Tilde{\lambda}_t^k(\nu, h) := \nu_k + \sum_{l=1}^{K} \int_{t-A}^{t^-} h_{lk}(t-s)dN_s^l.
    \end{equation}
With this notation, the nonlinear intensity can be written as $ \lambda_t^k(f) = \phi_k (\Tilde{\lambda}_t^k(\nu, h))$. For a nonlinear Hawkes process, the existence and uniqueness of a stationary distribution is proved under some assumptions on the parameters $f$ and the link functions $\phi=(\phi_k)_k$. In the following lemma, we provide two sufficient conditions, which are variants of existing work. We recall that a function $ \phi$ is $L$-Lipschitz, if for any $(x,x')\in\R^2$,
   $|\phi(x) - \phi(x')| \leq L |x - x'|$.
\begin{lemma}\label{lem:existence}
Let $N$ be a Hawkes process with parameter $f$ and link functions $(\phi_k)_k$ such that for any $k \in [K]$, $\phi_k: \mathbb{R} \to  \mathbb{R}^+ $ is monotone non-decreasing and $L$-Lipschitz, with $L > 0$. If one of the following conditions is satisfied:
\begin{itemize}
%
    
     \item [\textbf{(C1)}]  The matrix $S^+$ with entries $S_{lk}^+ = L \norm{h_{lk}^+}_1$ satisfies $r(S^+) < 1$;
  
     \item [\textbf{(C2)}] For any $k \in [K]$, $\phi_k$ is bounded, i.e., $\exists \Lambda_k > 0, \forall x \in \R$, $\phi_k(x) \leq \Lambda_k$. 
\end{itemize}
then there exists a unique stationary version of the process $N$ with finite average.  
\end{lemma}

In the previous lemma, the second stationarity condition \textbf{(C2)} directly comes from Theorem 7 by \cite{bremaud96} and is applied to our (less general) context of Lipschitz and non-decreasing link functions. The first condition \textbf{(C1)} is obtained in Theorem 1 of \cite{deutsch2022}, in a more restricted Hawkes model where $\phi_k(x) = (x)_+$ and the interaction functions are of the form $h_{lk} = K_{lk} g(t)$ with $g\geq0$ and $K_{lk} \in \mathbb{R}$, but the same arguments can be applied to prove the stationarity of the process in our more general nonlinear model. However, in the context of inference, we will consider a slightly stronger condition:

\vspace{1mm}

\textbf{(C1bis)} The matrix $S^+$ with entries $S_{lk}^+ = L \norm{h_{lk}^+}_1$ satisfies $\norm{S^+} < 1$.

\vspace{1mm}


From now on, we will assume that we observe on a window $[-A,T]$ a stationary Hawkes process with link functions $(\phi_k)_k$ and true parameters $f_0 = ((\nu^0_k)_{k=1}^K, (h^0_{lk})_{l,k=1}^K)$. We denote $\mathbb{P}_0$ the stationary distribution of $N$ and $\mathbb{P}_0(.|\mathcal{G}_0)$ its conditional distribution given $\mathcal{G}_0$. We note that $\mathbb{P}_0$ is a well-defined transformation of the probability distribution $\mathbb{P}$ (through its alternative definition in Lemma \ref{lem:stoc_domination} \cite{supplementary}). For $f = ((\nu_k)_{k=1}^K, (h_{lk})_{l,k=1}^K)$  satisfying the assumptions of Lemma \ref{lem:existence}, the log-likelihood of the processs on $[0,T]$ conditionally on $\mathcal{G}_0$ (i.e., conditionally on $N|_{[-A,0\textcolor[rgb]{0.5,0.2,0.4}{)}}$) is given by
    \begin{equation*}
        L_T(f)  := \sum_{k=1}^K \left[\int_0^T \log (\lambda_t^k(f)) dN_t^k - \int_0^T \lambda_t^k(f) dt\right].
    \end{equation*}
Then, for any parameter $f$, we define the conditional distribution $\mathbb{P}_{f}$ from the likelihood function
    \begin{equation}\label{eq:def_pf}
        d\mathbb{P}_{f}(.|\mathcal{G}_{0}) = e^{L_T(f) - L_T(f_0)}d\mathbb{P}_0(.|\mathcal{G}_{0}).
    \end{equation}
We denote $\mathbb{E}_0$ (resp. $\mathbb{E}_{f}$) the expectation associated with $\mathbb{P}_0$ (resp. $\mathbb{P}_{f}$). 
\subsection{Identifiability of the parameters}\label{sec:identifiability}
In this section, we provide some conditions on the model which ensure that the  parameters of the nonlinear Hawkes model defined in \eqref{def:NLintensity}  are identifiable. To do so we consider the following weak assumption.  

\begin{assumption}{}\label{ass:identif_f}
For $f = (\nu, h)$,   there exists $\varepsilon > 0$ such that for any $k \in [K]$, $\phi_k$ restricted to  $I_k = (\nu_k - \max \limits_{l \in [K]} \norm{h_{lk}^{-}}_\infty - \varepsilon, \nu_k + \max \limits_{l \in [K]} \norm{h_{lk}^{+}}_\infty + \varepsilon)$ is injective.

\end{assumption}

\begin{proposition}\label{lem:identif_param}
 Let $N$ be a  nonlinear Hawkes process  as defined in \eqref{def:NLintensity} with link functions $(\phi_k)_k$ and parameter $f = (\nu, h)$ satisfying the conditions of Lemma~\ref{lem:existence} and Assumption~\ref{ass:identif_f}. If $N'$ is a Hawkes processes with the same link functions $(\phi_k)_k$ and parameter $f' = (\nu', h')$,  then if $N$ and $N'$ have the same distribution, i.e., $ N \overset{\mathcal{L}}{=} N' $, then $ \nu = \nu'$ and $h = h'$.
\end{proposition}

 Note that if the $\phi_k$'s are injective on $\mathbb R$,  which holds in particular for the sigmoid and the softplus functions, then Assumption \ref{ass:identif_f}  is verified for all $f$. However our result is more general and also covers link functions which are only injective on a sub-interval of $\mathbb R$ such as ReLU or shifted ReLU  ($\phi_k(x) = \theta_k + \max(x,0))$ and clipped exponentials ($\phi_k(x) = \min(e^x, \Lambda_k)$ ). In this case,  Assumption \ref{ass:identif_f} holds over a restricted parameter space for $f$. More precisely, $\phi_k$ needs to be injective over an interval which includes  all the possible values of $\nu_k + h_{lk}(s)$, for any $l \in [K]$ and $s\in [0,A]$.


%

\begin{remark}
One consequence of  Assumption~\ref{ass:identif_f} is that for any $t > 0$ such that $N[t-A,t) \leq 1$, then  $\lambda^k_t(f) > 0$ (since $\phi_k$ is non-negative and monotone non-decreasing) for all  $k \in [K]$. However, Assumption \ref{ass:identif_f} still allows to model the \emph{refractory periods} of biological neurons, i.e., when the neurons cannot or are very unlikely to fire again during a period after firing. Indeed, one can have $\lambda_t^k(f) $ very small for $t$ such that $N^k[t-A,t) = 1$, depending on $f$ and $\phi_k$.
\end{remark}

Proposition \ref{lem:identif_param} supports the feasibility of the parameter estimation when the  nonlinear functions $\phi_k$'s are fully known. 
It can however be extended to the setup where the link functions are partially known. In the next proposition, we consider the case of $\phi_k(x) = \theta_k + \psi_k(x)$ where $\psi_k$ is a function such that $\lim \limits_{x \to -\infty} \psi_k(x) = 0$ and $\theta_k\geq 0$ is an unknown parameter, for each $k \in [K]$. In this case, we denote $\lambda_t(f, \theta)$ the intensity process.



\begin{proposition}\label{lem:identif_param2}
Let $N$ be a Hawkes process with parameter $f = (\nu, h)$ and link function $\phi_k(x; \theta_k) = \theta_k + \psi_k(x)$ with $\theta_k\geq 0$ for any $k\in [K]$ satisfying the conditions of Lemma~\ref{lem:existence} and Assumption~\ref{ass:identif_f}. We also assume that for all $k\in [K]$, $\lim \limits_{x \to -\infty} \psi_k(x) = 0$ and
 \begin{equation}\label{ass:identif_theta}
  \exists l\in [K], \, x_1<x_2, \quad \text{such that } \, h_{lk}^{-}(x)>0, \quad \forall x \in [x_1, x_2].
  \end{equation} 
 Then if $N'$ is a Hawkes processes with link functions  $\phi_k(x; \theta_k') = \theta_k' + \psi_k(x)$, $\theta_k'\geq 0$ and parameter $f' = (\nu', h')$,
\vspace{-3mm}\begin{align*}
 N \overset{\mathcal{L}}{=} N' \implies \nu = \nu', \quad  h = h', \quad \text{and} \quad \theta = \theta', \quad \theta = (\theta_k)_{k=1}^K, \quad \theta'=(\theta_k')_{k=1}^K.
\end{align*}
Besides, in this case we have $\mathbb{P}_f[\inf \limits_{t \geq 0}  \lambda^k_t(f, \theta) = \theta_k] = 1.$
\end{proposition}

The proofs of Propositions \ref{lem:identif_param} and \ref{lem:identif_param2} are reported in Section \ref{sec:proof:lem:existence} in the supplementary material \cite{supplementary}. In Proposition \ref{lem:identif_param2}, the condition \eqref{ass:identif_theta} implies that each component $k$ receives some inhibition (i.e., $\exists l, \: h_{lk}^{-} \neq 0$).  In particular, we will use this condition in the shifted ReLU model where $\psi_k(x) = (x)_+$. We note that $\theta_k$ is not identifiable when no inhibition is received by $N^k$ (i.e., when $\forall l, \: h_{lk}^- =  0$). More precisely, the following lemma - proved in Section \ref{sec:proof:lem:existence} in the supplementary material \cite{supplementary} - states that in a mutually-exciting ReLU model, 
 the parametrisation of the process is not unique. 
Informally, our models present a singularity at the parameter ``$h^- = 0$".

\begin{lemma}\label{lem:hawkes_exciting}
Let $N$ be a Hawkes process with parameter $f = (\nu, h)$ and link functions $\phi_k(x; \theta_k) = \theta_k + (x)_+, \: \theta_k \geq 0, \: k\in [K]$  satisfying Assumption \ref{ass:identif_f}, and let $k \in [K]$. If $\forall l \in [K], h_{l k} \geq 0$, then for any $\theta_k' \geq 0$ such that $\theta_k + \nu_k - \theta_k' > 0$, let $N'$ be the Hawkes process driven by the same underlying Poisson process $Q$ as $N$ (see Lemma \ref{lem:stoc_domination} \cite{supplementary}) with parameter $f' = (\nu', h')$ and link functions $\phi_k(x; \theta_k') = \theta_k' + (x)_+, k\in [K]$ with $\nu' = (\nu_1, \dots, \nu_k+ \theta_k - \theta_k', \dots, \nu_K) \neq \nu, \: h' = h$, and $\theta' = (\theta_1, \dots, \theta_k', \dots, \theta_K) \neq \theta$. Then for any $t \geq 0$,$\lambda^k_t(f, \theta) = \lambda^k_t(f', \theta')$, and therefore $N \overset{\mathcal{L}}{=} N'$.
\end{lemma}

\subsection{Bayesian inference}\label{sec:bayesian_inference}
We can now present our Bayesian estimation framework. We assume that the observed Hawkes process $N$ satisfies the conditions of Lemma~\ref{lem:existence}, i.e., the link functions $\phi_k$'s are monotone non-decreasing, $L$-Lipschitz with $L>0$ and either we consider a bounded model $\phi_k(x) \leq \Lambda, \forall k, \Lambda > 0$ (condition \textbf{(C2)}) or we assume $\norm{S_0^+} < 1$ (condition \textbf{(C1bis))} with $S_0^+ = (L \norm{h_{lk}^{0+}}_1)_{l,k \in [K]^2}$. We define the parameter space for $f = ((\nu_k)_{k=1}^K, (h_{lk})_{l,k=1}^K)$ and the functional space as follows. Let 
\begin{align*}
&\mathcal{H}' = \left \{ h: [0,A] \to \R; \: \|h \|_\infty < \infty \right \}, \quad \mathcal{H} = \left \{ h =  (h_{lk})_{l,k=1}^K \in \mathcal{H}'^{K^2}; \:  (h, \phi) \text{ satisfy  \textbf{(C1bis)} or  \textbf{(C2)} }  \right \}, \\
&\mathcal{F} = \left \{  f = (\nu, h) \in (\R_+\backslash\{0\})^K \times \mathcal{H}; \: f \text{ satisfies Assumption ~\ref{ass:identif_f} } \right \}.
\end{align*}
We recall that for an unbounded link function, condition \textbf{(C1bis)} corresponds to $\norm{S^+} < 1$ with $ S^+ = (L \norm{h^+_{lk}}_1)_{l,k \in [K]^2}$. 
We also recall that  $A > 0$ is fixed. In the graph estimation problem (see Section \ref{sec:post_consistency}), the  parameter of interest is $\delta_0 \in \{0,1\}^{K^2}$ where $h^0_{lk} = 0 \iff \delta^0_{lk}=0$. With a slight abuse of notations, we sometimes write $f=((\nu_k)_{k}, (h_{lk})_{l,k})_{k}, (\delta_{lk})_{l,k})$ with $\delta \in \{0,1\}^{K^2}$.

\begin{remark}
With ReLU-type link functions, we have $\mathcal{H} = \{h = (h_{l,k}) \in (\mathcal{H}')^{K^2}; \: \norm{S^+} < 1\}$ and $\mathcal{F} = \{
 f \in (\R_+\backslash\{0\})^K \times \mathcal{H}; \: \norm{h^-_{lk}}_\infty < \nu_k, \:  (l,k) \in [K]^2  \}$. With clipped exponential links $\phi_k(x) = \min(e^x, \Lambda_k)$, we have $\mathcal{H} =\mathcal{H}'^{K^2}$ and $\mathcal{F} = \left \{f  \in \R_+\backslash\{0\}^K \times \mathcal{H}'^{K^2};   \: \nu_k + \norm{h^+_{lk}}_\infty < \log \Lambda_k,  \:    (l,k) \in [K]^2 \right \}$.
\end{remark}


We now define our metric on the parameter space $\mathcal{F}$. For any $f = (\nu, h), \: f'=(\nu', h') \in \mathcal{F}$, we define the following $L_1$-distances:
    \begin{align*}
    &\norm{\nu - \nu'}_1 =  \sum_{k=1}^K |\nu_k - \nu_k'|, \quad \norm{h -h'}_1 =  \sum_{l=1}^K \sum_{k=1}^K \|h_{lk} - h_{lk}'\|_1, \quad \|f-f'\|_1 = \norm{\nu - \nu'}_1 + \norm{h -h'}_1.
    \end{align*}

Finally, we consider a prior distribution $\Pi$ on $\mathcal{F}$ and define the posterior distribution on $B \subset \mathcal{F}$ as
    \begin{equation}\label{def:pposterior_dist}
        \Pi(B|N) = \frac{\int_{B} \exp(L_T(f)) d\Pi(f)}{\int_{\mathcal{F}} \exp(L_T(f)) d\Pi(f)} = \frac{\int_{B} \exp(L_T(f) - L_T(f_0)) d\Pi(f)}{\int_{\mathcal{F}} \exp(L_T(f) - L_T(f_0)) d\Pi(f)} =: \frac{N_T(B)}{D_T},
    \end{equation}
denoting $N_T(B)$ and $D_T$ our numerator and denominator of the posterior with the form above.

\



\section{Main results}\label{sec:main_results}

In this section, we state our most important results on the posterior distribution on the parameter $f$ and the restriction on the connectivity graph $\delta$, leading respectively to convergence rates and consistency of some Bayesian nonparametric estimators.

\subsection{Posterior concentration rates}\label{sec:post_conc}

We first prove that under mild assumptions on the  link functions and the true parameter, we can describe the posterior concentration rate $\epsilon_T$ with respect to the $L_1$-distance on $\mathcal{F}$ defined in Section \ref{sec:bayesian_inference}, in terms of standard conditions on the prior. We then consider the case where the link functions $\phi_k$ depend on an unknown parameter, in the special case of shifted  ReLU: $\phi_k(x; \theta_k^0) = \theta_k^0 + (x)_+$, for which we prove posterior concentration on both $f_0$ and $\theta_0$. To do so, we use the following assumption on the true parameter, which is a strengthening of the identifiability condition in Assumption \ref{ass:identif_f}. 


\begin{assumption}\label{ass-psi} 
For $f_0 = (\nu_0, h_0)$, we assume that there exists $\varepsilon > 0$ such that for any $k \in [K]$, $\phi_k$ restricted to $I_k = (\nu^0_k - \max \limits_{l \in [K]} \norm{h_{lk}^{0-}}_\infty - \varepsilon, \nu^0_k + \max \limits_{l \in [K]} \norm{h_{lk}^{0+}}_\infty + \varepsilon)$ is bijective from $I_k$ to $J_k = \phi_k(I_k)$ and its inverse is $ L'$- Lipschitz on $J_k$, with $L' > 0$. We also assume that one of the two following conditions is satisfied:
\begin{enumerate}[label={\roman*)}]
	\item  For any $k \in [K]$, $ \inf \limits_{x \in \R}  \phi_k(x) >0$.
	\item  For any $k \in [K]$, $\phi_k > 0$ and $\sqrt{\phi_k}$   and  $\log \phi_k$ are $L_1$-Lipschitz  with $L_1 > 0$ . 
\end{enumerate}
\end{assumption}
The first part of Assumption \ref{ass-psi}, which is a slight strengthening of Assumption \ref{ass:identif_f}, holds in all cases described previously. The second part considers two cases: (i) the $\phi_k$'s are lower bounded by a positive constant, which holds for instance when $\phi_k(x; \theta_k) = \theta_k + \psi_k(x) $ with $\theta_k >0$ and $\psi_k \geq 0$ and (ii) the $\phi_k $'s can approach 0 but satisfy an additional smoothness condition which holds in particular if the derivatives $\phi_k'$ are bounded and the functions $\log \phi_k$'s are Lipschitz. It is notably the case for the commonly used Hawkes models \cite{costa18, Hansen:Reynaud:Rivoirard, gerhard17, carstensen2010multivariate, chen17b, Menon2018ProperLF, mei2017neural}, see Example \ref{rem:link_functions} below. Note that this assumption excludes the standard ReLU function $\phi_k(x) = (x)_+$, which we will treat separately in Proposition \ref{prop:relu}.


\begin{example}\label{rem:link_functions}
The following nonlinear models verify Assumption \ref{ass-psi}. Let  $s, t, \Lambda > 0$. 
\begin{itemize}
    \item \textbf{Positive or shifted ReLU}-type functions:  $\phi_1(x) = \max(sx,t) \geq t > 0$,  which is injective on $[t/s,+\infty)$, $s$-Lipschitz and its inverse  on $[t,+\infty)$, $\phi_1^{-1}(x) = s^{-1}x$ is $s^{-1}$-Lipschitz.
    \item \textbf{Clipped exponential} functions:  $\phi_2(x) = \min(e^{sx}, \Lambda)$, which is injective on $(-\infty, s^{-1} \log \Lambda]$ and 
    $s\Lambda$-Lispchitz. Its inverse on $(0,\Lambda]$, $\phi_2^{-1}(x) = s^{-1} \log x $ is Lipschitz on any compact of $(0,\Lambda]$ and $\sqrt{\phi_2}(x) = \sqrt{\min (e^{sx}, \Lambda})=\min (e^{sx/2}, \sqrt{\Lambda})$ and $\log \phi_2 = \min(sx, \log \Lambda)$ are respectively 
    $s\Lambda$--Lispchitz and $s$-Lipchitz;
    \item \textbf{Sigmoid}  functions: $\phi_3(x) = (1 + e^{- s(x - t)})^{-1}$, 
   which is injective on $\R$ and $s$-Lipschitz. Its inverse $\phi_3^{-1}(x) = t+\frac{1}{s}\log \frac{x}{1-x}$ is Lipschitz on any compact of $(0,1)$, $\sqrt{\phi_3}$ is $s$-Lipschitz and $\frac{\phi_3'(x)}{\phi_3(x)} \leq s$ thus $\log \phi_3$ is $s$-Lipschitz;
    \item \textbf{Softplus} functions: $\phi_4(x) = \log (1 + e^{s(x-t)})$, which is injective on $\R$, $s$-Lipschitz and its inverse $\phi_4^{-1}(x) = \frac{1}{s}\log (e^{x} - 1) + t$ is Lipschitz on any compact of $\R_+^*$ . Finally $\sqrt{\phi_4}$ and $\log \phi_4$ are $s$-Lipschitz.
\end{itemize}
\end{example}

To state our first result, we also define the following neighbourhoods in $f_0$ in  supremum and $L_2$-norms respectively, for $B > 0$:
\begin{align*}
&B_\infty(\epsilon_T) = \{f \in \mathcal{F}; \:  \nu_k^0 \leq \nu_k \leq  \nu_k^0 +  \epsilon_T, \, h_{lk}^0 \leq h_{lk} \leq  h_{lk}^0 + \epsilon_T, \:  (l,k) \in [K]^2\}. \\
&B_2(\epsilon_T, B) = \{f \in \mathcal{F}; \: \max_k |\nu_k - \nu_k^0| \leq \epsilon_T,  \: \max_{l,k} \|h_{lk} - h_{lk}^0\|_2 \leq \epsilon_T,  \: \max_{l,k} \|h_{lk}\|_\infty < B \}.
\end{align*}
In particular, $B_{\infty}(\e_T)$ is chosen so that for any $f \in B_{\infty}(\e_T)$, $k \in [K]$ and $t \geq 0$, the intensities verify $\lambda_t^k(\nu, h) \geq \lambda_t^k(\nu_0, h_0)$. 
Finally we define \begin{align} \label{kappaT}
\kappa_T = 10 (\log T)^r 
\end{align}
with $r=0$ if $(\phi_k)_k$ satisfies Assumption~\ref{ass-psi} (i), $r=1$ if $(\phi_k)_k$ satisfies Assumption~\ref{ass-psi} (ii). 

\begin{theorem}{}\label{thm:conc_g}
Let $N$ be a Hawkes process with known link functions $\phi = (\phi_k)_k$ and   parameter $f_0 = (\nu_0, h_0)$ such that $(\phi, f_0)$  satisfy Assumption~\ref{ass-psi}. Let $\epsilon_T = o(1/\sqrt{\kappa_T})$ be a positive sequence verifying $\log^3 T=O(T \epsilon_T^2)$ and $\Pi$ be a prior distribution on $\mathcal{F}$. We assume that the following conditions are satisfied for $T$ large enough.

\textbf{(A0)} There exists $c_1 > 0$ such that $\Pi(B_\infty(\epsilon_T)) \geq e^{-c_1T\epsilon_T^2}.$

\textbf{(A1)} There exist subsets $\mathcal{H}_T \subset \mathcal{H}$ and $c_2 > 0$ such that, with $\Upsilon_T = \{\nu = (\nu_k)_k, \: 0 < \nu_k \leq e^{c_2 T\epsilon_T^2}, \: \forall k\}$,
$ \Pi(\mathcal{H}_T^{c}) + \Pi(\Upsilon_T^c) = o(e^{- (\kappa_T +c_1) T\epsilon_T^2}).$

\textbf{(A2)} There exist $\zeta_0 > 0$ and $x_0 > 0$ such that $\log \mathcal{N}\left(\zeta_0 \epsilon_T, \mathcal{H}_T, ||.||_1\right) \leq x_0 T \epsilon_T^2.$

Then,  for $M>0 $ large enough, we have
\begin{equation}\label{conc:L1:1}
\mathbb{E}_0 \left[\Pi\big(\|f -  f_0\|_1 > M \sqrt{\kappa_T} \epsilon_T \big| N \big) \right]  = o(1).
\end{equation}

\end{theorem}

The proof of Theorem~\ref{thm:conc_g} is provided in Section \ref{sec:proof_conc_g}. 

\begin{remark}  \label{rk:B2} 
In Theorem~\ref{thm:conc_g}, if we replace $B_\infty( \epsilon_T)$ by $B_2(\epsilon_T,B)$ for some $B>0$ in \textbf{(A0)}, then the concentration rate in \eqref{conc:L1:1} is $\sqrt{\log \log T\kappa_T} \epsilon_T $ instead of $\sqrt{\kappa_T} \epsilon_T $. Replacing $B_\infty( \epsilon_T) $ by $B_2( \epsilon_T,B) $ can be useful for some families of priors, as seen in the case of mixtures of Beta distributions in Section \ref{sec:mixture} in the supplementary material \cite{supplementary}. 
\end{remark}

\begin{remark}
Our concentration rate in  \eqref{conc:L1:1} holds under the stationary distribution $\mathbb{P}_0$, implying in this case that the ``initial condition" $ N|_{[-A,0]} \subset \mathcal{G}_0 $ also comes from the stationary law. However, in practice, one might observe a process on $[-A,T]$ with an arbitrary distribution on $[-A,0]$. Under the conditions of Lemma \ref{lem:existence}, the dynamics of the resulting process are \emph{stable} (in the sense of Definition 1 of \cite{bremaud96}), using the results in \cite{bremaud96}. In particular, its distribution $\mathbb{P}_0(.|\mathcal{G}_0)$  converges exponentially fast to the stationary distribution $\mathbb{P}_0$. Therefore, we expect that  \eqref{conc:L1:1} would still hold under $\mathbb{P}_0(.|\mathcal{G}_0)$, i.e., under a more general initial distribution on $[-A,0]$.
\end{remark}

An interesting aspect of Theorem \ref{thm:conc_g} is that the assumptions on the prior \textbf{(A0)}, \textbf{(A1)} and \textbf{(A2)}, are similar to those of simpler estimation problems like density estimation, regression or linear Hawkes processes. This allows to directly derive explicit forms of the posterior concentration rates in the nonlinear Hawkes model under common families of priors, such as Gaussian processes, hierarchical Gaussian processes, basis expansions or mixture models (see \cite{vzanten:vdv:09, vvvz08, arbel, rousseau:09}  or Section 2.3.2 of \cite{donnet18} ).
 In Section~\ref{sec:priors}, we illustrate this using splines and mixture models. 
Additionally, our Theorem \ref{thm:conc_g} avoids the unpleasant assumption on the prior in Theorem 3 of \cite{donnet18} which requires that for some  $u_0 > 0$, 
$
     \Pi(\norm{S} > 1 - u_0 (\log T)^{1/6} \e_T^{1/3}) \leq e^{-2 c_1 T \e_T^2}.
$
This is thanks to our novel proof techniques using regeneration times under the true model $\mathbb{P}_0$ (see Section \ref{sec:lem_excursions}).

Theorem~\ref{thm:conc_g} provides posterior concentration rates for a large class of link functions, as discussed earlier. \ds{In particular, it covers the case of shifted ReLU link functions, i.e, $\phi_k(x) = \theta_k + (x)_+$ where $\theta_k > 0$ is a \emph{baseline} rate, which can be arbitrarily small. This link function can be seen as an alternative to the exponential function with positive baseline rate in \cite{gerhard17} 
In \cite{gerhard17}, neurons firing rates are modelled using nonlinear Hawkes processes with a positive link function, which still allows to account for the refractory periods of neurons (for which the firing rate is small). Moreover, while in the case of the shifted ReLU model, Theorem~\ref{thm:conc_g} assumes that the baseline rates $\theta = (\theta_k)_k$ are known, we show in the next proposition that we can also estimate $\theta$. Besides, we additionally provide a posterior concentration result when using the standard ReLU function $\phi_k(x) = (x)_+$, under a stronger assumption on the model.  We note that for this latter choice of link function, the intensity function is only \emph{non-negative} and the likelihood function is equal to 0 in parts of neighbourhoods of $h_0$, which causes several issues in the control of the Kullback-Leibler divergence.} 

Before stating our results, we define neighbourhoods in $\theta_0$, also in  supremum and $L_2$-norms, respectively $B^{\Theta}_\infty(\epsilon_T) = \{\theta \in \Theta; \:  \norm{\theta - \theta_0}_\infty  \leq  \epsilon_T \}$ and $B^\Theta_2(\epsilon_T, B) = \{\theta \in \Theta; \:  \norm{\theta - \theta_0}_2  \leq  \epsilon_T \}$, and in this case we define
$
\kappa_T = 10 (\log T)^r 
$
with $r=0$ in the shifted ReLU model (Case 2 of the following proposition) and  $r=2$ in the standard ReLU model (Case 1). 

\begin{proposition}\label{prop:relu}
Let $N$ be a nonlinear Hawkes process with link functions $(\phi_k)_k$ and parameter $f_0 = (\nu_0, h_0)$ satisfying  Assumption~\ref{ass:identif_f}.  Let $\epsilon_T = o(1/\sqrt{\kappa_T})$  be a positive sequence verifying $\log^3 T=O(T \epsilon_T^2)$ and $\Pi$ be a prior distribution on $\mathcal{F}$. 
\begin{itemize}
\item \textbf{Case 1 (Standard ReLU):  $\phi_k(x) = (x)_+$, for all $k \in [K]$.} Under the Assumptions \textbf{(A0)},  \textbf{(A1)} and  \textbf{(A2)} of Theorem~\ref{thm:conc_g}, if $f_0 $ verifies the following additional assumption 
\begin{equation} \label{condA3}
    \lim \sup_{T \to \infty}\frac{1}{T} \mathbb{E}_0\left(\int_0^T \frac{\mathds{1}_{\lambda^k_t(f_0) > 0}}{\lambda^k_t(f_0)}dt \right) < + \infty, \quad   k \in [K],
\end{equation}
then for $M>0 $ large enough,  \eqref{conc:L1:1} holds. 
\item \textbf{Case 2 (Shifted ReLU with $\theta_0$ unknown): $\phi_k(x; \theta_k^0) = \theta_k^0 + (x)_+, \: \theta_k^0>0,$ for all $ k \in [K]$.}  Let $\Pi_\theta$ be a prior distribution on $\Theta = \{\theta = (\theta_k)_k; \: \theta_k > 0\}$. 
If the Assumptions \textbf{(A0)},  \textbf{(A1)} and  \textbf{(A2)} of Theorem~\ref{thm:conc_g} are satisfied when replacing $B_\infty(\epsilon_T)$ by  $B_\infty(\epsilon_T) \cap B^\Theta_\infty(\epsilon_T) $ for $T$ large enough, and if \eqref{ass:identif_theta} is verified, then for  $M>0 $ large enough, 
\begin{align*}
&\mathbb{E}_0 \left[\Pi\big(\norm{ f -  f_0}_1 + \norm{\theta - \theta_0 }_1  > M \sqrt{\kappa_T'} \epsilon_T \big| N \big) \right]  = o(1).
\end{align*} 
\end{itemize}
\end{proposition}

\begin{remark}
In Case 1 of  Proposition~\ref{prop:relu} only,  $B_\infty(\epsilon_T)$ cannot be replaced by $B_2(\epsilon_T)$ in assumption \textbf{(A0)}. This is due to the fact that we need to consider parameters $f$ such that the likelihood at $f$ is positive (i.e., $\exp(L_T(f))>0$)  in order to control the Kullback-Leibler divergence (see Lemmas \ref{lem:KLdecomp}, \ref{lem:t_j} and \ref{lem:KL}).
In this argument, we also need the additional assumption \eqref{condA3}.
The latter is a non trivial condition on the intensity of the true model, which we do not expect to hold in many situations. For instance it does not hold if $\tilde \lambda_t(f_0)$ is Lipschitz in a neighbourhood of $t$ such that $\tilde \lambda_t(f_0)=0$. We expect that this can happen with significant probability as soon as one interaction functions $h_{lk}^0$ is Lipschitz and $h_{lk}^{0-}$ is non-null. It is however not clear if this condition is sharp, i.e., if Bayesian or other likelihood-based methods would be suboptimal without this assumption (from our construction of tests, it is easy to construct frequentist estimates of $f $ which converge at the rate \ds{$\sqrt{\epsilon_T}$} defined by the testing condition in \cite{ghosal:vdv:07}). \ds{This also motivates the study of the shifted ReLU  model, as an alternative of interest for modelling positive intensity functions}.
Nonetheless, in Lemma \ref{lem:histo}, we provide sufficient conditions in a finite-histogram model so that \eqref{condA3} holds. Finally, we note that using Theorem 1.2 of \cite{costa18} and notation $\tau_1, \tau_2$ for the regeneration times defined in Lemma \ref{lem:excursions}, \eqref{condA3} is equivalent to
$
    \mathbb{E}_0\left(\int_{\tau_1}^{\tau_2} \frac{\mathds{1}_{\lambda^k_t(f_0) > 0}}{\lambda^k_t(f_0)}dt \right) < + \infty.
$
\end{remark}

\begin{remark}\label{rem:unknown_link}
In Theorem \ref{thm:conc_g},  we in fact obtain the posterior concentration rate on $((\phi_k(\nu_k))_k,h)$, i.e.,
$$\mathbb{E}_0 \left[\Pi\left( \sum_k |\phi_k(\nu_k) - \phi_k(\nu_k^0) | + \norm{h - h_0}_1 > M \sqrt{\kappa_T} \epsilon_T\big| N \right) \right]  = o(1),$$ for $M$ a large enough constant. Moreover, if the $\phi_k$'s are partially known of the form $\phi_k(x;\theta_k) = \theta_k + \psi(x)$ where $\theta_k \geq 0$ and $\psi$ is given, then we obtain $$\mathbb{E}_0 \left[\Pi\big(\norm{ h -  h_0}_1 + \sum_k |\theta_k + \psi(\nu_k) - \theta_k^0 - \psi(\nu_k^0)  | > M \sqrt{\kappa_T} \epsilon_T \big| N \big) \right]  = o(1).$$
\end{remark}
\vspace{-5mm}
In the next corollary, we deduce from the previous results the convergence rate of the posterior means
\begin{align*}
&(\hat{\nu}, \hat{h}) = \mathbb{E}^\Pi[f|N] = \int_{\mathcal{F}} f d\Pi(f| N), \quad 
\text{and} \quad \hat{\theta} = \mathbb{E}^\Pi[(\theta)|N] \quad \text{when $\theta_0$ is unknown (in the shifted ReLU model)}.
\end{align*}
\begin{corollary}\label{cor:post_mean}
Under the assumptions of Theorem \ref{thm:conc_g} or Case 1 of Proposition \ref{prop:relu}, if $\int_{\mathcal{F}} \norm{f}_1 d\Pi(f) < + \infty$, 
then for $M > 0$  large enough, it holds that
\begin{align*}
&\mathbb{P}_0 \left[\|\hat{\nu} -  \nu_0\|_1 + \|\hat{h} -  h_0\|_1 > M \sqrt{\kappa_T} \epsilon_T \right]  = o(1).
\end{align*}
Under the assumptions of Case 2 of Proposition \ref{prop:relu}, we have
\begin{align*}
&\mathbb{P}_0 \left[\|\hat{\nu} -  \nu_0\|_1 + \|\hat{h} -  h_0\|_1 + \|\hat{\theta} -  \theta_0\|_1 > M\sqrt{\kappa_T} \epsilon_T \right]  = o(1).
\end{align*}

\end{corollary}

The proofs of Theorem \ref{thm:conc_g} and Proposition \ref{prop:relu} are given in Sections \ref{sec:proof_conc_g} and  \ref{sec:proof_conc_f}, and the proof of Corollary \ref{cor:post_mean} is reported in Section \ref{sec:proof_cor_post_mean}  in the supplementary material \cite{supplementary}.

\subsection{Consistency on the connectivity graph}\label{sec:post_consistency}

In this section, we state our consistency results on the connectivity or Granger causality graph $\delta \in \{0,1\}^{K^2}$, which characterises the fact that interaction functions between pairs of dimensions are null or not, i.e., $\delta_{lk} = 0 \iff h_{lk} = 0, \: (l,k) \in [K]^2$. We note that the definition of Granger causality graph for the linear Hawkes model (see for instance Definition 3.3 in \cite{eichler2016graphical}) also holds for the nonlinear model. 
This leads us to consider the following hierarchical spike-and-slab prior structure.  Writing $h_{lk} = \delta_{lk} h_{lk} = \delta_{lk}S_{lk} \bar h_{lk}$, with $S_{lk} = \norm{h_{lk}}_1$ and $\bar h_{lk}$ such that $\norm{\bar h_{lk}}_1 = 1$, we define a family of priors: 
\begin{align} \label{structureprior}
 \delta &\sim  \pi_\delta, \quad \mathcal I(\delta) = \{ (l,k) \in [K]^2 ; \, \delta_{lk}=1\}, \nonumber\\
 (   h_{lk}, \, (l,k) \in \mathcal I(\delta) )|\delta &\sim \Pi_{h|\delta}( \cdot | \delta) \quad \text{and} \quad \forall (l,k)\notin \mathcal I(\delta), \: h_{lk} =0,
 \end{align}
with $\pi_\delta$ a probability distribution on $\{0,1\}^{K^2}$. We can either determine $\Pi_{h|\delta}$ as a distribution on the set of $(   h_{lk}, \, (l,k) \in \mathcal I(\delta) )$ and obtain the marginal distribution of $S = (S_{lk})_{lk}$, or construct it as in \cite{donnet18} - see also the prior construction in Section \ref{sec:priors}. Adapting \textbf{(A0)} to the above structure, we recall that $\delta_0 $ corresponds to the true connectivity parameter and we consider the following assumption 
$$ \textbf{(A0')} \quad  \Pi(  B_\infty(\epsilon_T) | \delta= \delta_0)   \geq e^{ -c_1 T\epsilon_T^2/2} , \quad \pi_\delta ( \delta = \delta_0) \geq e^{ -c_1 T\epsilon_T^2/2}.$$

For instance, one can choose $\pi_\delta = \mathcal B(p)^{K^2}$ with $0 < p < 1$, implying that the $\delta_{lk} $'s are i.i.d. Bernoulli random variables. Then for any fixed $p$, \textbf{(A0')}  is verified as soon as $ \Pi_{h|\delta}( B_\infty(\epsilon_T)|\delta = \delta_0)  \geq e^{ -c_1 T\epsilon_T^2/2} $ holds. This formalism allows us to consider the posterior  distribution of $\delta$ which is a key object to infer the connectivity graph. The next theorem is our posterior consistency result, which is a consequence of Theorem \ref{thm:conc_g} and Proposition \ref{prop:relu} and holds for all previously considered link functions $\phi$.

\begin{theorem}\label{thm:post_graph}
Let $N$ be a Hawkes process with function $\phi = (\phi_k)_k$ and parameter $f_0 = (\nu_0, h_0)$,  $\epsilon_T = o(1/\sqrt{\kappa_T})$  be a positive sequence and $\Pi$ be a prior distribution on $\mathcal{F}$ satisfying the conditions of Theorem \ref{thm:conc_g} or Proposition \ref{prop:relu} (replacing \textbf{(A0)} by \textbf{(A0')}). Then, 
\begin{equation} \label{H0}
    \mathbb{E}_0 \left[\Pi(\delta_{lk} \neq \delta_{lk}^0, \, \forall (l,k) \in \mathcal I(\delta_0) |N) \right] = o(1), \quad I(\delta_0) = \{(l,k) \in [K]^2; \: \delta_{lk}^0 = 1\}.
\end{equation}
 If in addition the following holds
\begin{align}\label{eq:cond_prior_graph}
\forall \delta  \in \{0,1\}^{K^2}, \: \forall C>0,\,  \forall (l,k) \in  \mathcal I(\delta)\cap \mathcal I(\delta_0)^c , \: \Pi_{h|\delta} \left(S_{lk} \leq C \epsilon_T \left| \delta \right.\right) = o\left(e^{- (\kappa_T + c_1) T \epsilon_T^2}\right),
\end{align}
with $c_1 > 0$ defined in \textbf{(A0')}, 
then $\mathbb{E}_0 \left[\Pi(\delta \neq \delta_0 |N) \right] = o(1).$
\end{theorem}

The first part of Theorem \ref{thm:post_graph} in \eqref{H0} is directly obtained from Theorem \ref{thm:conc_g} or Proposition \ref{prop:relu} (Cases 1 and 2) and says that the posterior probability of $\delta_{lk} =1$ converges to 1, if the edge $l \to k$ is in $I(\delta_0)$, i.e., $\delta_{lk}^0=1$. The second and more difficult part of Theorem \ref{thm:post_graph} is to infer a non-edge $\delta_{lk}^0=0$. The condition \eqref{eq:cond_prior_graph} forces the conditional prior distribution $\Pi_{h|\delta}$ to be exponentially small around 0 for all $h_{lk}$ such that $\delta_{lk}=1$.  We note that it also implies that if $h^0_{lk}\neq 0$ and is small, then it may not be detected nor estimated properly. In Section \ref{sec:priors}, we present two common families of  priors on the $S_{lk}$'s that verify \eqref{eq:cond_prior_graph}.


Interestingly, if the model is more constrained, a much  weaker condition on the prior distribution on $S_{lk}$ is required which avoids this issue on the estimation of small ``signals" $h_{lk}^0$. We now consider two restricted Hawkes models, where the interaction functions are either all the same, or only depend on the ``receiver" node. For simplicity of exposition, we consider the case of fully known link functions satisfying the assumptions of Theorem \ref{thm:conc_g}, however our next proposition remains valid for the ReLU and shifted ReLU models under the assumptions of Proposition \ref{prop:relu}.
\begin{itemize}
    \item \textbf{All-equal model:}  we assume that $\forall (l,k) \in [K]^2, \: h_{lk} = \delta_{lk}  \Tilde h$, with $\Tilde h \in \mathcal{H}'$
so that $\mathcal{F}= \{f=(\nu,  \delta, \Tilde h) \in \R_+\backslash\{0\}^{K} \times \{0,1\}^{K^2} \times \mathcal{H}'; \: (f, \phi) \: \text{satisfy \textbf{(C1bis)} or \textbf{(C2)} and Assumption \ref{ass-psi}} \}$. 
When $\delta \neq 0$, then $\tilde h \sim \Pi_{\tilde h}$ is a probability distribution on $\mathcal H' \cap \{ \Tilde h\neq 0\}$.
\item \textbf{Receiver node dependent model:} we assume that $\forall (l,k) \in [K]^2, h_{lk} = \delta_{lk} h_k$ with $h_k \in \mathcal{H}'$, so that $\mathcal{F}= \{f=(\nu, \delta, (h_k)_{k}); \: h_k \in \mathcal{H}', \: \forall k, \: (f, \phi) \: \text{ satisfy \textbf{(C1bis)} or \textbf{(C2)} and Assumption \ref{ass-psi}} \}$. We also assume that the prior distribution $\Pi$ can be written as a product of priors $(\Pi_k)_{k}$ where for each $k$, $\Pi_k$ is a distribution on $(\nu_k, h_k, \delta_{lk}, \,  l \in [K])$, restricted to $ \mathcal F$. We denote $\delta_{\cdot k} = (\delta_{lk}, \: l \in [K])$.

\end{itemize}


\begin{proposition}\label{prop:graph_restr}
We consider a restricted Hawkes model either defined above as the \textbf{All-equal model} or as the \textbf{Receiver node dependent model}. Let $N$ be a Hawkes process with function $\phi = (\phi_k)_k$ and parameter $f_0 = (\nu_0, h_0)$ and 
let $\Pi$ be a prior distribution on $\mathcal{F}$ such that the prior on $\nu$ has positive and continuous density wrt the Lebesgue measure. We also assume that there exists $0<p_1<1/2$ such that for any $(l,k) \in [K]^2$,  $p_1\leq \Pi(\delta_{lk}=1)\leq 1-p_1$.

\begin{itemize} 
\item In the \textbf{All-equal model}:

\begin{enumerate}
\item  If there exists $(l,k) \in [K]^2$ such that 
$\delta_{lk}^0\neq 0$, then if $\Pi_{\tilde h}( h_0\leq \tilde h \leq h_0+ \epsilon_T ) \geq e^{-c_1 T \epsilon_T^2/2} $  and if \textbf{(A1)}, \textbf{(A2)} hold, then
$\Exz{\Pi ( \delta \neq \delta_0 | N)} = o(1).$
\item If $\delta_0=0$, then if there exists  $\mathcal H_T \subset \mathcal H$ such that for all $\delta\neq 0$, $\Pi_{h|\delta}( \mathcal H_T^c|\delta)  = o ( T^{-K/2} ) $, if \textbf{(A2)} holds with $\epsilon_T = \sqrt{\log T/T}$, and if
\begin{align}\label{ass:a4_1}
\forall C > 0, \: \Pi_{\tilde h} \left(0 < \|\tilde h\|_1 \leq C  \sqrt{ \log T /T}  \right) =o((\log T)^{-K/2}),
\end{align}
then $\Exz{ \Pi ( \delta \neq 0 | N)} = o(1).$
\end{enumerate}

\item In the \textbf{Receiver node dependent model}: under \textbf{(A0')}, \textbf{(A1)}, \textbf{(A2)}, for any $k \in [K]$, 

\begin{enumerate}
\item  If  there exists $l \in [K]$ such that  $\delta_{lk}^0\neq 0$, then 
$\Exz{\Pi ( \delta_{k_1k}  \neq \delta^0_{k_1k}  | N)} = o(1), \: \forall  k_1\in [K].$
\item If  $\delta_{\cdot k}^0=0$, if  there exists  $\tilde{\mathcal H_T} \subset \mathcal H_1$ such that $\Pi_k( \tilde{\mathcal H_T}^c)  = o( T^{-K/2} ) $, and if for $M>0$ large enough and $x_0 > 0, \: \zeta_0 > 0$,
$$\mathcal{N}\left(\zeta_0 M \sqrt{\log T/T},\tilde{\mathcal H_T}, ||.||_1\right) \leq  T^{x_0M}, $$
and if \eqref{ass:a4_1} holds with $h_k$ instead of $\Tilde{h}$,
then 
$\Exz{ \Pi (\delta_{\cdot k} \neq \delta_{\cdot k}^0 | N)} = o(1).$
\end{enumerate}
\end{itemize}

\end{proposition}

Consequently,  in those restricted Hawkes models, the above proposition states that the posterior distribution is consistent at $\delta_0$ under the much weaker assumption\eqref{ass:a4_1}  on the prior  compared to \eqref{eq:cond_prior_graph} of Theorem \ref{thm:post_graph}. In fact, in the \textbf{All-equal model} (resp. the \textbf{Receiver node dependent model}), if the true graph has no edge (resp. no edge arriving on node $k$), then the posterior distribution on $h$ (resp. $h_k$) concentrates at the paranetric rate $\sqrt{\log T/T}$. This gives a sharp lower bound on the marginal density of $N$, i.e., on the denominator $D_T$ in \eqref{def:pposterior_dist}. We note that  \eqref{ass:a4_1} is a mild condition which is verified in particular when the prior distribution on $\Tilde{S} = \|\Tilde{h}\|_1$ (resp. $S_k= \|h_k\|_1$) conditionally on $\Tilde{S} \neq 0$ (resp. $S_k\neq 0$) has a  density wrt the Lebesgue measure bounded by $\Tilde{S}^{-a}$  (resp. $S_k^{-a}$) with $a > 0$ near 0.

We now study the consistency of Bayesian estimators of the connectivity graph. From Theorem \ref{thm:post_graph} or Proposition~\ref{prop:graph_restr}, we can directly obtain that the graph estimator based on the $0-1$ loss function defined as $\hat{\delta}^{\Pi}_{lk}(N) = 1 \iff \Pi(\delta_{lk}=1 |N) > \Pi(\delta_{lk}=0 |N)$,
is consistent, i.e., $\mathbb{P}_0 \left[\hat{\delta}^{\Pi}(N)  \neq \delta_0 \right] = o(1)$.  
 This result is obtained with the prior condition \eqref{eq:cond_prior_graph} in the non-restricted model, which as previously explained can deteriorate the inference of small and non-null interaction functions. We thus propose an alternative graph estimator based on a loss function penalising small signals, which therefore allows us  to use  prior distributions which do not verify \eqref{eq:cond_prior_graph}. For any graph estimator $\hat{\delta}  = (\hat{\delta}_{lk})_{l,k} \ \in \{0,1\}^{K^2}$ and parameter $f = (\nu, h, \delta) \in \mathcal{F}$, we define
\begin{equation*}
    L(\hat{\delta}, f) = \sum_{l,k = 1}^K \mathds{1}_{\hat{\delta}_{lk}=0} \mathds{1}_{\delta_{lk}=1} + \mathds{1}_{\hat \delta_{lk}=1} \big( \mathds{1}_{\delta_{lk}=0} + \mathds{1}_{\delta_{lk} = 1} F(\norm{h_{lk}}_1)\big),
\end{equation*}
with $F: \R^+ \to [0,1]$ a monotone  non-increasing function, with $F(0)=1$. For a prior $\Pi$, the risk of the estimator $\hat{\delta}$ is defined as
\begin{align*}
        r(\hat{\delta},\Pi|N) &= \int_{\mathcal{F}} L(\hat{\delta},f) d\Pi(f|N) = \sum_{l,k} \mathds{1}_{\hat{\delta}_{lk}=0}  \Pi(\delta_{lk}=1|N) + \mathds{1}_{\hat{\delta}_{lk}=1} \left[\Pi(\delta_{lk}=0|N) + \mathbb{E}^\Pi(\mathds{1}_{\delta_{lk}=1} F(\norm{h_{lk}}_1)|N)\right].
\end{align*}
Then the associated risk-minimising estimator, $ \hat{\delta}^{\Pi,L} (N) = \arg\min \limits_{\delta \in \{0,1\}^{K^2}} r(\delta,\Pi|N)  $, verifies
\begin{align}\label{eq:def_estim}
    \hat \delta_{lk}^{\Pi,L} (N) = 1 & 
    \iff \mathbb{E}^\Pi[(1 - F(\norm{h_{lk}}_1)) \mathds{1}_{\delta_{lk}=1} |N] \geq \Pi(\delta_{lk}=0|N).
\end{align}
In the next theorem, we prove that our estimator $\hat{\delta}^{\Pi,L}(N)$ is consistent under the true model $\mathbb{P}_0$ if the penalisation function $F$ satisfies an exponential condition.

\begin{theorem}\label{cor:graph_estim}
Let $N$ be a Hawkes process with function $\phi = (\phi_k)_k$ and parameter $f_0 = (\nu_0, h_0)$,  $\epsilon_T = o(1/\sqrt{\kappa_T})$ be a positive sequence and $\Pi$ be a prior distribution on $\mathcal{F}$ satisfying the conditions of Theorem \ref{thm:conc_g} or Proposition \ref{prop:relu} (replacing \textbf{(A0)} by \textbf{(A0')}). Then, if there exists $a>0$ such that
\begin{align}\label{eq:cond_penalisation}
 0 \leq 1 - F( M \sqrt{\kappa_T} \epsilon_T) \leq e^{-(c_1+a+\kappa_T) T\epsilon_T^2},
\end{align}
for $T$ large enough and with $M > 0$  defined in Theorem \ref{thm:conc_g},  then  for any $(l,k) \in \mathcal I(\delta_0)$ such that
$ 1 - F(\norm{h_{lk}^0}_1)  \geq 2e^{- (\kappa_T +c_1) T\epsilon_T^2}$,  we have
 $\mathbb{P}_0 \left[\hat{\delta}^{\Pi,L} (N) \neq \delta_0 \right] = o(1).$
\end{theorem}

\begin{remark}
The assumption on the penalisation function \eqref{eq:cond_penalisation} is verified in particular if (i) $F$ is truncated, i.e., $F(x) = \mathds{1}_{[0,\epsilon]}(x)$ for some (arbitrarily small) $\epsilon > 0$, or if (ii) $F$ is exponentially decreasing around 0, i.e., $F(x) = 1 - \exp \{ - \frac{1}{x^p}\}$ with $p >1/\beta$ if  $\epsilon_T = T^{-\beta/2\beta+1}(\log T)^q $ for some $q\geq 0$ (see Corollary \ref{cor:histograms} for instance). We note that the choice of penalisation function $F$ determines the detection level of our risk-minimising graph estimator for ``small signals". With (i), we will detect ``signals" $\|h_{lk}^0\|_1 > \epsilon$ and with (ii), we can detect $\|h_{lk}^0\|_1 > T^{-(p(2\beta+1))^{-1}}$.  We also note that this assumption is related to \eqref{eq:cond_prior_graph}, however, since it applies on the penalisation function $F$ and not on the prior distribution, it does not alter the posterior distribution, thus the estimation of $\nu_0$ and $h_0$. 
\end{remark}

The proofs of Theorem \ref{thm:post_graph}, Proposition \ref{prop:graph_restr} and Theorem \ref{cor:graph_estim} can be found respectively in Section \ref{sec:proof_post_graph}, Section  \ref{sec:proof_prop_restr} in the supplementary material \cite{supplementary} and Section \ref{sec:proof_cor_graph}.




\section{Prior models}\label{sec:priors}

In this section, we construct prior distributions $\Pi$ that satisfy the assumptions of our main results stated in Section \ref{sec:main_results} and obtain explicit posterior concentration rates for H\"{o}lder-smooth classes of interaction functions. 
For ease of exposition, we consider link functions $\phi_k$'s injective on $(m_k,M_k)$, with $m_k, M_k \in \R \cup \{-\infty, +\infty\}$.

First, we consider a prior on $\nu = (\nu_k)_k$
of the form:
$\nu_k  \stackrel{i.i.d.}{\sim} \pi_\nu(\nu_k | (h_{lk})_{l \in [K]}) )\propto \pi_\nu(\nu_k) \mathds{1}_{(m_k,M_k)}(\nu_k)$ with 
$\pi_\nu $ a positive and continuous probability density on $(0, +\infty)$. \ds{To verify \textbf{(A1)}, we can for instance choose $\pi_\nu$ such that $\pi_\nu(\nu_k > x) \leq x^{-a}$ with $a > 1$. Then it is enough to choose $c_2$ such that $c_2 > (\kappa_T + c_1)/a$.}
Moreover in Case 2 of Proposition \ref{prop:relu} (i.e., shifted ReLU with unknown shift $\theta_0$), we consider a prior on $\theta$ such that $\theta_k \stackrel{i.i.d.}{\sim} \pi_\theta$ with $\pi_\theta$  a density wrt the Lebesgue measure on $(0,+\infty)$.

For the prior on $h$, we consider the hierarchical structure \eqref{structureprior} introduced in Section \ref{sec:post_consistency} and for the sake of simplicity we assume that 
$ \delta_{lk} \stackrel{i.i.d.}{\sim} \mathcal{B}(p), \: \forall (l,k) \in [K]^2, \: p \in (0,1), $
although as previously mentioned, more  general priors on $\delta$ could be considered. We recall that $I(\delta) = \{(l,k) \in [K]^2; \: \delta_{lk} =1 \}$. We then consider two parametrisation setups. In the first one, $h =(h_{lk}, \: (l,k) \in I(\delta) )$ is drawn from a truncated distribution of the form 
\begin{equation}\label{parametrisation1}
    d\Pi_h( h | \delta)  \propto d\Pi_h^{\otimes |I(\delta)|}(h) \mathds{1}_{ \norm{S^+}<1}(h), 
\end{equation} 
or simply $d\Pi_h( h | \delta)  \propto d\Pi_h^{\otimes |I(\delta)|}(h)$ in the case of a bounded link function (condition \textbf{(C2)}), where $\Pi_h$ is a prior distribution on one function. In the second parametrisation setup,
\begin{equation}\label{parametrisation2}
h_{lk} = S_{lk} \bar{h}_{lk}, \quad \norm{\bar{h}_{lk}}_1 = 1, \quad   [\bar{h}_{lk} | (l,k) \in I(\delta)] \stackrel{i.i.d.}{\sim} \Pi_{\bar h}, \quad   S|\delta \sim \Pi_{S|\delta},
\end{equation}
with $\Pi_{\bar h}$ is a prior distribution on one $L_1$-normalised function and $\Pi_{S|\delta}$ is a prior distribution on matrices with non-zero entries $\delta$ and, under \textbf{(C1bis)}, 
satisfying $\norm{S^+}<1$.

Examples of the parametrisation setup \eqref{parametrisation1}  are Gaussian processes (or hierarchical Gaussian processes) priors, and prior distributions based on an expansion on some basis, such as Legendre, Fourier, wavelets,  splines, etc. As mentioned earlier, the prior assumptions \textbf{(A0)-(A2)} are very common in the literature, which allows to directly apply existing results, as we illustrate on  spline priors in Section \ref{sec:spline}. In \cite{donnet18}, a similar construction is provided using a mixture of Betas distributions in the linear Hawkes model, which leads to the minimax rate of assumption up to a logarithmic factor. We report this construction in the nonlinear model in Section \ref{sec:mixture} in the supplementary material \cite{supplementary} and obtain the same estimation rate up to logarithmic terms. The difficulty in this parametrisation might be to prove condition \eqref{eq:cond_prior_graph} in Theorem \ref{thm:post_graph} for estimating the connectivity graph. 
In Section \ref{sec:random:histo}, we illustrate the second parametrisation setup \eqref{parametrisation2} with random histogram priors, which is a setup where condition \eqref{eq:cond_prior_graph} can be more easily verified. We also consider a prior based on mixtures of Beta distributions in the supplementary material \cite{supplementary}. We denote $\mathcal{H}(\beta, L_0)$ the class of $\beta$-smooth functions with radius $L_0$.

\subsection{Spline priors for $\Pi_h$}\label{sec:spline}

A nonparametric prior $\Pi_{h}$ satisfying the assumptions of Theorem \ref{thm:conc_g} can be constructed using  the family of splines or free knot splines. Without loss of generality, we assume that $A = 1$. For $J \geq 1$, let $t_0 = 0 < t_1 < \dots < t_J = 1$ define a partition of $[0,1]$ and $I_j = (t_{j-1}, t_j)$, $j\in [J]$. We consider splines of order $q \geq 0$, i.e., piecewise polynomial functions (on the partition) of degree $q$ and for $q\geq 2$, $q-2$ times continuously differentiable. For a given partition, this defines a vector space of dimension $V = q+J-1$ (see for instance \cite{stone, ghosal00}).

For the sake of simplicity, we present the construction of regular partitions, where $t_j  = j/J$, however random partitions can be dealt with following the computations of Section 2.3.1 of \cite{donnet18}. Let $B = (B_1, \cdots, B_V)$ be the $B$-spline basis of order $q$, as defined in \cite{ghosal00}. Recall that for any $j \in [V]$, $B_j$ has support included in an interval of length $q/J$, $B_j\geq 0$ and $\sum_j B_j(x)=1$ for all $x \in [0,1]$.  We can then define
\begin{align*}
    h_{w,J}(x) = w^T B(x), \quad w \in \mathbb R^V,  \quad J \sim \mathcal P(\lambda), 
\end{align*}
where $\mathcal P(\lambda)$ is the Poisson distribution with mean $\lambda$, and consider the following hierarchical construction of $\Pi_{h}$
\begin{align}\label{eq:prior_J}
    w_j \stackrel{i.i.d.}{\sim} \pi_w, \quad 1 \leq j \leq V = q + J-1,
\end{align}
with $\pi_w$ a positive and continuous density on $\mathbb R$ satisfying $ \pi_w(x) \lesssim e^{-a_1 |x|^{a_2}} $ for some $a_1,a_2, \lambda > 0$.

Using Lemma 4.1 of \cite{ghosal00}, if $h_0$ is $\mathcal H(\beta, L_0)$ for some $\beta \leq q$ and $L_0>0$, then setting $J_T =  J_0(T/\log T)^{1/(2\beta+1)}$, $\epsilon_T = (T/\log T)^{-\beta/(2\beta+1)}$, there exist $w_0\in \mathbb R^{V_T}$, $V_T = q+J_T-1$ and $C>0$ such that $\|h_0 -h_{w_0,J_T} \|_\infty \leq C \epsilon_T $. Moreover using Lemma 4.2 and Lemma 4.3 of \cite{ghosal00}, we have  $\|w_0\|_\infty \leq C_0$, for some $C_0$,  and obtain that 
$ \{w\in \mathbb R^{V_T} , \: \|w-w_0\|_\infty \leq \epsilon_T\} \subset  B_\infty(\epsilon_T)$, which leads to \textbf{(A0)}. Similarly, from Lemma 4.2 of \cite{ghosal00},  
$\|h_{w,J} - h_{w',J}\|_1 \lesssim \|w-w'\|_\infty $
and with
$\mathcal H_T = \{ h_{w,J}; \: \|w\|_\infty \leq T^{B_0}, \: J\leq J_1 J_T\}$ for some $B_0>0 $ and $J_1>0$, \textbf{(A1)} and \textbf{(A2)} are also verified. We finally obtain the following result.

\begin{corollary}{}\label{cor:spline}
Let $N$ be a Hawkes process with link functions $\phi = (\phi_k)_k$ and parameter $f_0 = (\nu_0, h_0)$ such that $(\phi, f_0)$ verify the conditions of Lemma~\ref{lem:existence}, and Assumption \ref{ass-psi}. Under the above spline  prior, 
if for any $(l,k) \in [K]^2, h_{lk}^0 \in \mathcal{H}(\beta,L_0)$ with $\beta \in (0, q+1]$ and $L_0 > 0$, then for $M > 0$ large enough, we have
\begin{align*}
    \mathbb{E}_0 \left[\Pi\big(\|f - f_0 \|_1 > M  \left(T/\log T\right)^{-\beta/(2\beta+1)}(\log T)^{q_0} \big| N \big) \right]  = o(1),
\end{align*}
where $q_0=0 $ if $\phi$ verifies Assumption \ref{ass-psi}(i) and $q_0=1/2$ if $\phi$ verifies Assumption \ref{ass-psi}(ii).
\end{corollary}

To estimate the connectivity graph $\delta_0$, one can either use the penalised estimator \eqref{eq:def_estim}, which from the above computations and Corollary \ref{cor:graph_estim} is consistent, or use the estimator based on the 0-1 loss function if \eqref{eq:cond_prior_graph} can be verified. 
In the next section, we consider a prior based on random histograms and illustrate how the latter condition \eqref{eq:cond_prior_graph} can be satisfied.

\subsection{Random histograms prior} \label{sec:random:histo}

Random histograms are a special case of splines with $q=0$. These piecewise constant functions are of particular interest in the modelling of spike trains emitted by biological neurons, which only interact  on certain time periods. We use a similar construction as in Section 2.3.1. of \cite{donnet18}, however here the interaction functions are no longer restricted to be non-negative. Using parametrisation \eqref{parametrisation2}, the interaction function $h_{lk}$ for $(l,k) \in I(\delta)$ has the form  $h_{lk} = S_{lk} \bar h_{lk}$ and the $\bar h_{lk} $'s are independent and distributed as a random histogram  $\bar h_{w,\mathbf t}$ defined as follows. Given a partition $\mathbf t:$ $0=t_0< t_1 <\cdots< t_J=1$, we define 
$$ \bar h_{w,\mathbf t}(x) = \sum_{j=0}^{J-1} \frac{ w_j }{ t_{j+1}-t_j}\mathds{1}_{(t_{j-1}, t_j]}, \quad \sum_{j = 0}^{J-1} |w_j| =1, \quad J \sim \mathcal P(\lambda), \: \lambda > 0 .$$
Similarly to \cite{donnet18},  the prior on $(|w_1|, \cdots, |w_J|)$ is constructed by first selecting the non-zero coefficients $w_j$'s, then defining a Dirichlet prior on the vector of non-zero $|w_j|$'s, and finally sampling the sign of the $w_j$'s. Hence, 
$$\forall  j \in [J], \: w_j = Z_j u_j, \quad Z_j \in \{ -1, 0 , 1\}, \quad u_j \geq 0, \quad \sum_{j=1}^J u_j=1,$$
and $u_j=0$ if $Z_j=0$. We can consider 
$Z_j\stackrel{i.i.d.}{\sim}\text{Multinomial}(p_{-1}, p_0, p_1)$, with $p_{-1} + p_0 + p_1 = 1$, and given $(Z_1, \cdots, Z_J)$,
$ (u_{i_1}, \cdots, u_{i_{s_z}}) \sim \mathcal D( a_{s_z}, \cdots, a_{s_z}), \quad s_z= \sum_j |Z_j|,$
where $i_1, \cdots , i_{s_z}$ are the indices of the non zero $Z_j$'s and $\alpha_{-1}, \alpha_{0}, \alpha_{1}, a_{s_z} > 0$.
Finally if the partition $\mathbf{t}$ is random, we consider a Dirichlet prior $\mathcal D(\alpha, \cdots, \alpha)$ on $(t_1, t_2-t_1, \cdots, 1-t_{J-1})$. We note that this construction is very similar to Section 2.3.1 of \cite{donnet18}, and we therefore obtain the same results as in Corollaries 2 and 3 of \cite{donnet18}.

 Besides, to estimate the connectivity graph using the 0-1 loss (and to establish our posterior consistency result), we can now verify \eqref{eq:cond_prior_graph}. This condition  holds if, with $d\Pi_{S|\delta} = \prod_{(l,k) \in I(\delta)} d\Pi_S(S_{lk}) \mathds{1}_{\|S^+\|<1}$ (under \textbf{(C1bis)}),  $\Pi_S$ has a positive and continuous density $\pi_S$ on either $[\epsilon, 1]$ if $S_{lk}^0 >\epsilon$, or if the density near 0 verifies 
 \begin{equation*}
    \pi_S(s^p) \propto s^{-p(\alpha-1)} \exp (-a/s^p) \mathds{1}_{[0,1]}(s),\quad p > \beta, \quad a>0.
\end{equation*}

We now present a corollary of Theorem \ref{thm:conc_g} in the case of random histograms with random partitions, which is proved as in \cite{donnet18}. 
\begin{corollary}{}\label{cor:histograms}
Let $N$ be a Hawkes process with link functions $\phi = (\phi_k)_k$ and parameter $f_0 = (\nu_0, h_0)$ such that $(\phi, f_0)$ verify  Assumption \ref{ass-psi}. Under the above random histogram prior, 
if  for any $ (l,k) \in [K]^2, h_{lk}^0 \in \mathcal{H}(\beta,L_0)$ with $\beta \in (0,1]$ and $L_0 > 0$, then for $M$ large enough, we have
\begin{align*}
    \mathbb{E}_0 \left[\Pi\big(\|f - f_0 \|_1 > M  (T/\log T)^{-\beta/(2\beta+1)}(\log T)^q \big| N \big) \right]  = o(1),
\end{align*}
where $q= 0$ if $\phi$ verifies Assumption \ref{ass-psi}(i), and $q=1/2$ if $\phi$ verifies Assumption \ref{ass-psi}(ii).
\end{corollary}

Finally, in the case of the ReLU model (Proposition \ref{prop:relu}), we can also verify \eqref{condA3}, in special case of the true parameter $f_0 = (\nu_0, h_0)$ where each $h_{lk}^0$ lie in the space of finite histograms.

\begin{lemma}\label{lem:histo}
Let $N$ be a nonlinear Hawkes process with  parameter $f_0 = (\nu_0, h_0)$ and ReLU link functions $\phi_k(x) = (x)_+, \forall k$, satisfying  Assumption~\ref{ass:identif_f} (and condition \textbf{(C1bis)}). If for all $(l,k) \in [K]^2$, there exists $J_0 \in \mathbb{N}^*$ such that
$
    h_{lk}^{0}(t) = \sum_{j=1}^{J_0} \omega^{lk}_{j0} \mathds{1}_{I_j}(t),
$
with $\{I_j\}_{j=1}^{J_0}$ a partition of $[0,1]$ and $\forall j \in [J_0], \: \omega_{j0}^{lk} \in \mathbb Q $, then \eqref{condA3} holds.
\end{lemma}

\begin{remark}
In the previous lemma, the condition that the weights $w_{j0}^{lk}, (l, k) \in [K]^2, j \in [J]$ are rational numbers  is a technical argument that allows to find a lower bound 
 on $\Tilde{\lambda}_t^k(f_0)$ when  $\lambda_t^k(f_0) > 0$. 
 This results from a density argument of the linear combinations of the weights, which, under these conditions, constrains $\lambda_t^k(f_0)$ to take values on a lattice. Besides, we note that our result is in fact more general and applies to any model with Lipschitz link functions such that $\min_{x \in \R} \phi_k(x) = 0$.
\end{remark}

Lemma \ref{lem:histo} is proved in Section \ref{sec:proof_lem_histo}  in the supplementary material \cite{supplementary}.

%
%
%
%




\section{Proofs}\label{sec:proof}
In this section, we report the proofs of our main theorems on the posterior concentration properties (Theorems \ref{thm:conc_g} and Proposition \ref{prop:relu}), and on the estimation of the  connectivity graph (Theorems \ref{thm:post_graph} and \ref{cor:graph_estim}). Instead of using the clustering structure of linear Hawkes processes like in \cite{donnet18} or a coupling technique like in \cite{chen17b}, these proofs leverage the renewal properties of nonlinear Hawkes processes notably studied by Costa et al. in \cite{costa18}. The novelty of our proofs lies in the selection of parts or special ``excursions", that allow us to estimate the parameter at a rate equivalent to the one for a linear Hawkes process.  In the following section, we first recall the definitions of the concept of excursions and some properties of the process' renewal times.



\subsection{Renewal times and excursions}\label{sec:lem_excursions}

In the following lemma, we introduce the concept of \emph{excursions} for stationary nonlinear Hawkes processes verifying the conditions of Lemma~\ref{lem:existence}. This result extends the ones of Costa et al. in \cite{costa18} to the multivariate case under condition \textbf{(C1bis)} of Lemma ~\ref{lem:existence} and to bounded models (condition \textbf{(C2)}).

\begin{lemma}\label{lem:excursions}
Let $N$ be a Hawkes process with monotone non-decreasing and Lipschitz link functions $\phi = (\phi_k)_k$ and parameter $f = (\nu, h)$ such that $(\phi, f)$ verify \textbf{(C1bis)} or \textbf{(C2)}.
Then the point process measure $X_t(.)$ defined as
\begin{equation}\label{eq:pp_measure_x}
    X_t(.) = N|_{(t-A,t]},
\end{equation}
is a strong Markov process with positive recurrent state $\emptyset$. Let $\{\tau_j\}_{j\geq 0}$ be the sequence of random times defined as
\begin{align*}
    \tau_j = \begin{cases}
    0 & \text{ if } j=0; \\ 
    \inf \left \{t > \tau_{j-1}; \: X_{t^-} \neq \emptyset, \: X_{t} = \emptyset \right \}  = \inf \left \{t > \tau_{j-1}; \: N|_{[t-A,t)} \neq \emptyset, \: N|_{(t-A,t]} = \emptyset \right \} & \text{ if } j\geq 1 .
    \end{cases}
\end{align*}
Then, $\{\tau_j\}_{j\geq 0}$ are stopping times for the process $N$. For $T > 0$, we also define 
\begin{equation}\label{def:J_T}
    J_T=\max\{j\geq 0;\: \tau_j \leq T\}.
\end{equation}
The intervals $\{[\tau_j, \tau_{j+1})\}_{j=0}^{J_{T}-1} \cup [\tau_{J_T}, T]$ form a partition of $[0,T]$. The point process measures $(N|_{[\tau_j, \tau_{j+1})})_{1 \leq j \leq J_T - 1}$ are i.i.d. and independent of $N|_{[0, \tau_1)}$ and $N|_{[\tau_{J_T},T]}$; they are called \emph{excursions} and the stopping times $\{\tau_j\}_{j\geq 1}$ are called \emph{regenerative} or \emph{renewal} times. 
\end{lemma}

The proof of the previous lemma is omitted since it is a fairly direct multivariate extension of some elements of Proposition 3.1, Proposition 3.4, Theorem 3.5 and Theorem 3.6 in \cite{costa18}, recalled in Section \ref{sec:results_costa}  in the supplementary material \cite{supplementary}. For the extension to bounded models, we use a direct consequence of the results in Costa et al. \cite{costa18} that if $N$ is dominated by a homogeneous Poisson point process, then it also have the regenerative properties of Lemma~\ref{lem:excursions}. We also note that since $A$ is known, the renewal times $\tau_j$'s are observable. In the rest of this article, we denote
\begin{equation}\label{eq:def_tau}
\Delta \tau_1 = \tau_2 - \tau_1,
\end{equation}
the length of a generic excursion. For any link functions $\phi_k$'s and parameter $f = (\nu, h)$, we denote $r_f$ the value of the intensity process at the beginning of each excursion, defined as
\begin{equation}\label{eq:r_0}
r_f = (r_1^f, \dots, r_K^f),\quad r_k^f = \phi_k(\nu_k), \quad k \in [K].
\end{equation}
In the next two lemmas, we prove some useful results on the distributions of $\Delta \tau_{1}$, on the number of points in a generic excursion $N[\tau_1, \tau_2)$ and on the number of excursions in the observation window $[-A,T]$, $J_T$, defined in  \eqref{def:J_T}.

\begin{lemma}\label{lem:tau}
Under the assumptions of Lemma \ref{lem:excursions}, the random variables $\Delta \tau_1$ and $N[\tau_1, \tau_2)$ admit exponential moments. More precisely, under condition \textbf{(C1bis)}, with  $m = \norm{S^+}<1$, we have 
$$
\forall s < \min (\norm{r_f}_1,\gamma/A), \quad \Exf{e^{s \Delta \tau_1}}  \leq \frac{1+m}{2m}, \quad \text{and}\quad \Exf{e^{sN[\tau_{1},\tau_2)}} < + \infty, \quad \gamma = \frac{1- m}{2\sqrt{K}} \log \left(\frac{1 + m}{2m}\right).
$$
Under condition \textbf{(C2)}, we have
$
\forall s < \min_k \Lambda_k), \: \Exf{e^{s \Delta \tau_1}}  \leq \frac{\norm{\Lambda}_1^2 }{ (\min_k \Lambda_k - s)^2}$ and $ \Exf{e^{sN[\tau_{1},\tau_2)}} < + \infty$.
In particular, this implies that $\Exf{N[\tau_{1},\tau_2) + N[\tau_1,\tau_2)^2} < + \infty.$
\end{lemma}

\begin{remark}
The previous lemma provides exponential moments of $\Delta \tau_1$ and $N[\tau_1, \tau_2)$, under the assumption that $\norm{S^+} < 1$ \textbf{(C1bis)}, but we conjecture that results of Lemma~\ref{lem:tau} still holds under the more general conditions $r(S^+) < 1$ \textbf{(C1)}  of Lemma \ref{lem:existence}.
\end{remark}


\begin{lemma}\label{lem:conc_J}
Under the assumptions of Lemma \ref{lem:excursions}, for any $\beta > 0$, there exists a constant $c_\beta > 0$ such that
$
\Probf{J_T \notin [J_{T,\beta, 1}, J_{T,\beta, 2}]} \leq  T^{-\beta},
$
with $J_T$ defined in \eqref{def:J_T} and
\begin{align*}
&J_{T,\beta, 1} = \left \lfloor \frac{T}{\Exf{\Delta \tau_1}} \left(1 - c_\beta \sqrt{\frac{\log T}{T}}\right)\right \rfloor, \quad J_{T,\beta, 2} = \left \lfloor \frac{T}{\Exf{\Delta \tau_1}} \left(1 + c_\beta \sqrt{\frac{\log T}{T}}\right)\right \rfloor.
\end{align*}
\end{lemma}

The proofs of Lemmas  \ref{lem:tau} and \ref{lem:conc_J} are reported in Section \ref{sec:proof_lem_tau}  in the supplementary material \cite{supplementary}.

\subsection{Proof of Theorem \ref{thm:conc_g} and Case 1 of Proposition~\ref{prop:relu}}\label{sec:proof_conc_g}

In this section, we prove our main posterior concentration theorem, Theorem~ \ref{thm:conc_g}, as well as Case 1 of Proposition~\ref{prop:relu},  which deals with the specific case of the standard ReLU model. The first step of this proof borrows some ideas from the one of Theorem 3 in \cite{donnet18}, but also introduces novel elements built from the renewal properties of the process. In particular, the posterior concentration is first proved in terms of a particular distance on the intensity process (see Proposition~\ref{thm:d1t} below), which in fact corresponds to a stochastic (pseudo) distance on the parameter space $\mathcal{F}$. This stochastic distance $\Tilde{d}_{1T}$ resembles the $L_1$ stochastic distance used in \cite{donnet18}, except that it is restricted to a subset of the observation window $[-A,T]$ which only contains the beginning of each excursion. More precisely for any excursion index $j \in [J_T-1]$, we denote $(U_j^{(1)}, U_j^{(2)})$ the times of the  first two events after the $j$-th renewal time $\tau_j$ (as defined in Lemma \ref{lem:excursions}). We note that by definition, $U_j^{(1)} \in [\tau_j,\tau_{j+1})$, \: $U_j^{(2)} \in [\tau_j,\tau_{j+2}]$ and $\tau_{j+1} \geq U_j^{(1)} + A$. We then define our restricted observation window $A_2(T)$ as
\begin{equation}\label{def:A2}
    A_2(T) := \bigcup_{j=1}^{J_T - 1} [\tau_j, \xi_j],
\end{equation}
with $\xi_j :=  U_j^{(2)}$ if $U_j^{(2)} \in [\tau_j,\tau_{j+1})$ and $\xi_j := \tau_{j+1} $ 
otherwise. We note that the interval $[\tau_j, \xi_j]$ corresponds either to the beginning of the $j$-th excursion or to the whole excursion $[\tau_j, \tau_{j+1})$ when the latter contains only one event, implying that $U_j^{(2)} \geq  \tau_{j+1}$. Moreover, since the renewal times (and $J_T$) are observable, so is $A_2(T)$.


 The construction of $A_2(T)$ is a novel and essential element of  our proof. Informally, it corresponds to a set of intervals where the parameters can be inferred in a similar way as in the linear Hawkes model and which Lebesgue measure is of order $T$. 
More precisely, using the renewal properties from Section~\ref{sec:lem_excursions},  we will prove, using Lemma \ref{lem:main_event}, that with probability going to 1,  $|A_2(T)| \gtrsim T$ under $\mathbb P_0$. We can now define our auxiliary stochastic distance as
\begin{equation} \label{tilded1}
    \Tilde{d}_{1T}(f,f') = \frac{1}{T} \sum_{k=1}^K \int_0^T \mathds{1}_{A_{2}(T)}(t) |\lambda_{t}^k(f) - \lambda_{t}^k(f')| dt,
\end{equation}
and state our intermediate posterior concentration rate result, which holds for all models satisfying the conditions of Theorem ~\ref{thm:conc_g} and the ReLU-type models considered in Proposition ~\ref{prop:relu}.


\begin{proposition}{}\label{thm:d1t}
Under the assumptions of Theorem \ref{thm:conc_g} or Proposition ~\ref{prop:relu}, for  $M_T' = M' \sqrt{\kappa_T}$ with $M' > 0$ a large enough constant,
$$\mathbb{E}_0 \left[\Pi(\Tilde{d}_{1T}(f,f_0) > M_T' \epsilon_T|N) \right]  = o(1). $$
\end{proposition}
The proof of the previous proposition follows the strategy of \cite{donnet18} in Theorem 1, which is based on the now well-known argument by \cite{ghosal:vdv:07}. However, we note that in our setting, this strategy can be applied thanks to the definition of the stochastic distance which restricts the observation window to the set $A_2(T)$. We recall here its main steps. First, we restrict the space of probability events to a subset $\Tilde{\Omega}_T$ that has high probability (see below and Lemma \ref{lem:main_event}). Secondly, we prove a lower bound of the denominator $D_T$ defined in \eqref{def:pposterior_dist}, derived from the technical Lemma \ref{lem:KL}. Thirdly, we consider  a ball centered at the true parameter $f_0$ of radius $M_T' \epsilon_T$ w.r.t.  $\Tilde{d}_{1T}$, denoted by $A_{d_1}(M_T' \epsilon_T) \subset \mathcal{F}$. Finally, to find an upper bound of the numerator $N_T(A_{d_1}(M_T' \epsilon_T)^c)$ defined in \eqref{def:pposterior_dist}, we partition $A_{d_1}(M_T' \epsilon_T)^c$  into slices $\{S_i\}_i$ on which we can design tests that have exponentially decreasing type I and type II errors (see Lemma \ref{lem:tests}). We then define $\phi$ as the maximum of the tests on the individual slices $S_i$. Due to the space constraints, this proof is reported in Section \ref{sec:proof_thm_d1t} of the supplementary material \cite{supplementary}.

From Proposition \ref{thm:d1t}, we prove Theorem \ref{thm:conc_g} and Case 1 of Proposition \ref{prop:relu} using the following classical  decomposition (see for instance the proof of Theorem 1 in \cite{donnet18}).  Let $A, B \in \mathcal F_T \subset \mathcal F$, with $B$ possibly data dependent, $\phi \in [0,1]$ be a measurable test, $\kappa_T$ defined in \eqref{kappaT}, and $\eve \subset \Omega$. 
Then,
\small
\begin{align}\label{decomposition}
 \Exz{\Pi(A \cap B |N)} &\leq \Probz{\{D_T < e^{-(\kappa_T + c_1) T \epsilon_T^2}\} \cap \eve} + \Exz{\phi \mathds{1}_{\Tilde{\Omega}_T}}+ \probz{\eve^c} \nonumber\\
& + e^{(\kappa_T + c_1) T \epsilon_T^2}\Pi( \mathcal F_T^c)  + e^{(\kappa_T+c_1) T \epsilon_T^2}\int_{A\cap \mathcal F_T} \Exz{\Exf{(1 - \phi)\mathds{1}_B(f)\mathds{1}_{\eve}(N)\Big| \mathcal{G}_0} } d\Pi(f).
\end{align}
\normalsize

We first introduce the set $\Tilde{\Omega}_T$, which from Lemma~\ref{lem:main_event}, has probability $\Probz{\eve^c}$ going to 0 at any polynomial rate. For $T > 0$, we denote
\begin{equation*}
\mathcal{J}_T := \left \{ J \in \N; \: \left|\frac{J-1}{T} - \frac{1}{\mathbb{E}_0[\Delta \tau_1]} \right| \leq c_\beta \sqrt{\frac{\log T}{T}} \right \},
\end{equation*}
with $c_\beta > 0$ (and $\beta > 0$) chosen in Lemma \ref{lem:main_event}, and, with $r_0 := r_{f_0} = (r_1^0, \dots, r_K^0)$ where $r_k^0 = \phi_k(\nu_k^0)$, and $\mu_k^0 = \Exz{\lambda_t^k(f_0)}$, for any $k$,
\begin{align*}
       \Omega_N &= \left \{ \max \limits_{k \in [K]} \sup \limits_{t \in [0,T]} N^k[t-A,t) \leq C_\beta \log T  \right \} \cap \left \{ \sum_{k=1}^K \left|\frac{N^k[-A,T]}{T} - \mu_k^0\right| \leq \delta_T  \right \}, \\
        \Omega_{J} &= \left\{ J_T \in \mathcal{J}_T \right \}, \quad \Omega_{U} =  \left\{ \sum_{j=1}^{J_T-1} (U_j^{(1)} - \tau_j) \geq 
     \frac{T}{\mathbb{E}_0[\Delta \tau_1] \|r_0\|_1} \left(1 - 2c_\beta\sqrt{\frac{\log T }{T}}\right) \right \}, 
\end{align*}
with $\delta_T = \delta_0 \sqrt{\frac{\log T}{T}}, \: \delta_0 > 0$ and $C_\beta > 0$ chosen in Lemma  \ref{lem:main_event} and define
\begin{equation}\label{def:Omega}
\tilde\Omega_T = \Omega_N \cap \Omega_J \cap \Omega_U.
\end{equation}
The sets $\Omega_N,$ $\Omega_J$ and $\Omega_U$  control respectively the number of events, the number of excursions and the length of excursions. First, $\Omega_N$ corresponds to realisations of $N$ such that the number of events in any interval of length $A$ is upper bounded by $c_\beta \log T$, and the number of events on $[-A,T]$ is close to its expectation under the stationary distribution $\mathbb{P}_0$. Secondly, $\Omega_J$ corresponds to the realisations such that the number of excursions in the observation interval $[0,T]$ divided by $T$, $J_T/T$, is close to its limit $1 / \mathbb{E}_0[\Delta \tau_1]$. Thirdly, on $\Omega_U$, the measure of the subset corresponding to the collections of the beginnings of excursions (from $\tau_j$ to the first event $U_j^{(1)}$) is of order $T$. 


Next, we bound the denominator of the posterior $D_T$ from  \eqref{def:pposterior_dist}. From Lemma \ref{lem:KL}, together with the lower bound technique of \cite{ghosal:vdv:07}, we have that 
\begin{equation}\label{DT}
 \Probz{D_T< \Pi(B_\infty(\epsilon_T)) e^{-\kappa_TT\epsilon_T^2} } \leq 2\int_{B_\infty(\epsilon_T)} \frac{\probz{L_T(f) - L_T(f_0) < - \kappa_T T\epsilon_T^2/2}}{ \Pi (B_\infty(\epsilon_T))}d\Pi(f) = o(1),
 \end{equation}
which leads to $ \Probz{D_T < e^{-(\kappa_T + c_1) T \epsilon_T^2}} = o(1)$ using asssumption \textbf{(A0)}. 

Then, we find a lower bound on $|A_2(T)|$ on $\tilde \Omega_T$. 
We recall that the point process measures $(N|_{[\tau_j, \tau_{j+1})})_{1\leq j \leq J_T-1}$ are i.i.d. and \textit{a fortiori} that the random variables $\{U_j^{(1)} - \tau_j\}_j$ are i.i.d.
Moreover,  for any $j \in [J_T-1]$, $t \in [\tau_j, U_j^{(1)})$ and $k \in [K]$, the intensity process is by construction equal to $\lambda^k_t(f_0)=r_k^0 = \phi_k(\nu_k^0)$. Therefore,  
 conditionally on $\tau_j$, $U_j^{(1)}$ has the same distribution as an event from a Poisson point process beginning at $\tau_j$, with intensity $\norm{r_0}_1$, since the process is the superposition of $K$ univariate Poisson process with intensity $r_k^0,  \: k \in [K]$. Thus, under $\mathbb{P}_0$, each variable $U_j^{(1)} - \tau_j$ follows an exponential distribution with mean $1/\norm{r_0}_1$, 
and on $\Omega_U$, for $T$ large enough, we have that
\begin{align*}
    |A_2(T)| = \sum_{j=1}^{J_T-1} (\xi_j - \tau_j) \geq  \sum_{j=1}^{J_T - 1} (U_j^{(1)} - \tau_j)  \geq c_0 T , \quad c_0:= \frac{ 1}{2 \Exz{\Delta \tau_1} \norm{r_0}_1}.
\end{align*}
Finally, for $R>0$, we define the balls in $L_1$ and stochastic distances
\begin{align*}
&A_{L_1}(R) := \{ f \in \mathcal{F} ; \: \norm{f - f_0}_1 \leq R \}, \quad A_{d_1}(R) = \{ \Tilde{d}_{1T}(f, f_0) \leq R  \}.
\end{align*}



We now apply the decomposition \eqref{decomposition} with $\phi = 1$, $A := A_{L_1}(M_T \epsilon_T)^c$ and $B := A_{d_1}(M_T' \epsilon_T)$, with $M_T = M \sqrt{\kappa_T} $,  $M_T' = M' \sqrt{\kappa_T}  $, $M > M'$ and $M'$ defined in Theorem \ref{thm:d1t}. As in the proof of Theorem~3 of \cite{donnet18}, we are thus left to prove that
\begin{equation}\label{eq:sup_ball_l1}
    \sup \limits_{A_{L_1}(M_T \epsilon_T)^c \cap \mathcal{F}_T} \Probf{\Tilde{\Omega}_{T} \cap A_{d_1}(M_T '\epsilon_T) | \mathcal{G}_0} = o_{\mathbb{P}_0}(e^{-(c_1+\kappa_T) T \epsilon_T^2}),
\end{equation}
with $c_1$ defined in assumption \textbf{(A0)}. We recall that $\mathbb{P}_f$ is the process distribution associated to parameter $f$ defined in \eqref{eq:def_pf}. To prove \eqref{eq:sup_ball_l1}, we consider $f \in A_{L_1}(M_T \epsilon_T)^c$ such that $\Tilde{d}_{1T}(f,f_0) \leq M_T' \epsilon_T$ and for $l \in [K]$ and $j \in [J_T-1]$, we define
\begin{align}\label{eq:def_zjl}
    Z_{jl} := \int_{\tau_{j}}^{\xi_j} |\lambda^l_t(f) - \lambda^l_t(f_0)|dt. 
\end{align}
We note that using Lemma \ref{lem:excursions}, the random variables $\{Z_{jl}\}_{ j \in [ J_T-1]}$ are i.i.d., and from \eqref{tilded1} we also have that $T \Tilde{d}_{1T}(f,f_0)  > \max \limits_{l \in [K]} \sum \limits_{j=1}^{J_T-1} Z_{jl}$. In order to derive a Bernstein-type inequality on the sum of the $Z_j$'s, we first find an upper bound of $Z_{1l}$ and its moments. Using that the link functions $\phi_k$'s are $L$-Lipschitz, we have 
\begin{align} \label{UBZ}
     Z_{jl} &=  \int_{\tau_{j}}^{\xi_j} |\phi_k(\Tilde \lambda^l_t(\nu, h)) - \phi_k(\Tilde \lambda^l_t(\nu_0, h_0))|dt 
      \leq L \int_{\tau_{j}}^{\xi_j} | \Tilde \lambda^l_t(\nu, h) - \Tilde \lambda^l_t(\nu_0, h_0)|dt \nonumber\\
     &\leq L  (\xi_j - \tau_j)|\nu_l - \nu_l^0| + L \sum_k \int_{U_{j}^{(1)}}^{\xi_j} |h_{kl}-h_{kl}^0|(t - U_j^{(1)}) dt \nonumber\\
     &\leq L (A+U_{j}^{(1)} - \tau_j)|\nu_l - \nu_l^0|  + L\sum_k \|h_{kl}-h_{kl}^0\|_1 \leq L (A+1+U_{j}^{(1)} - \tau_j) \norm{f - f_0}_1.
 \end{align}
Moreover, under $\mathbb{P}_f$, for any $j \in [J]$, $U_{j}^{(1)} - \tau_j$ follows an exponential distribution with mean $1/\norm{r_f}_1$, therefore, for any $n \in \N$,
$
\Exf{(U_{j}^{(1)} - \tau_j)^n} = \frac{n!}{\norm{r_f}_1^n}.
$
Using the standard inequality $(x+y)^n\leq 2^{n-1}(x^n+y^n)$, we thus obtain that
\begin{align}
\Exf{Z_{1l}^n} &\leq 2^{n-1} L^n \left((A+1)^n + \Exf{(U_{j}^{(1)} - \tau_j)^n} \right) \norm{f - f_0}_1^n \nonumber \\
&\leq \frac{1}{2} 2 n! \left(2 L \max \left(A+1, \frac{1}{\norm{r_f}_1}\right) \norm{f - f_0}_1\right)^{n-2} \times L^2  \max \left(A+1, \frac{1}{\norm{r_f}_1}\right)^2\norm{f - f_0}_1^2 \leq  \frac{1}{2} n! b^{n-2} v^2, \label{eq:ef_n}
\end{align}
with $b := 2 L \max \left(A+1, \frac{2}{\norm{r_0}_1}\right) \norm{f - f_0}_1$ and $v := L \max \left(A+1, \frac{2}{\norm{r_0}_1}\right) \norm{f - f_0}_1$.
%
In the last inequality, we have used the fact that   $\|r_f - r_0 \|_1\ \lesssim \Tilde d_{1T}(f,f_0) \leq M_T' \epsilon_T$ on $\Tilde{\Omega}_{T}$. This is because $ (U_1^{(1)} - \tau_1) + \dots + (U_{J_T-1} ^{(1)} - \tau_{J_T - 1}) \geq c_0 T /2$, which leads to
\begin{align}
    T \Tilde{d}_{1T}(f,f_0) &\geq \sum_{k}  |r_k^f- r_k^0|  \left((U_1^{(1)} - \tau_1) + \dots + (U_{J_T-1} ^{(1)} - \tau_{J_T - 1})\right)   \geq \frac{ T\sum_{k}  |r_k^f- r_k^0|  }{ 2 \Exz{\Delta \tau_1} \norm{r_0}_1}.\label{eq:min_d1tilde}
\end{align}
It also implies that  $\|r_f \|_1 \geq  \norm{r_0}_1 - \norm{r _f- r_0}_1 \geq \|r_0 \|_1/2$ for $T$ large enough. 


Our final argument consists in using the lower bound on $\Exf{Z_{1l}} $ obtained in Lemma \ref{lem:ef}. In this technical lemma, we show that there exists $l \in [K]$ and $C(f_0) > 0$ such that
$
    \Exf{Z_{1l}} \geq C(f_0)  \norm{f - f_0}_1.
$
Therefore, for this $l$,
\begin{small}
\begin{align*}
    &\Probf{\Tilde{\Omega}_{T} \cap \{\Tilde{d}_{1T}(f,f_0) \leq M_T' \epsilon_T\}\Big|\mathcal{G}_0}\leq \Probf{\eve \cap \left \{\sum_{j=1}^{J_T-1} Z_{jl} \leq M_T' T \epsilon_T \right \} \Bigg|\mathcal{G}_0 } \\
     &\leq \Probf{ \eve \cap \left \{ \sum_{j=1}^{J_T-1} (Z_{jl} - \Exf{Z_{jl}})  \leq  M_T' T\epsilon_T - (J_T -1 ) \Exf{Z_{jl}} \right \} \Bigg|\mathcal{G}_0 } \\
  &\leq \Probf{\bigcup_{J \in \mathcal{J_T}} \left \{ \sum_{j=1}^{J-1} (Z_{jl} - \Exf{Z_{jl}})  \leq - \frac{C(f_0)T \norm{f - f_0}_1}{4\mathbb{E}_0[\Delta \tau_1]} \right \} \Bigg|\mathcal{G}_0} \leq   \sum_{J \in \mathcal{J}_T} \Probf{ \sum_{j=1}^{J-1} (Z_{jl} - \Exf{Z_{jl}})  \leq - \frac{C(f_0)T \norm{f - f_0}_1}{4\mathbb{E}_0[\Delta \tau_1]} \Bigg| \mathcal{G}_0 },
\end{align*}
\end{small}
where we have used, for the third inequality, that on $\Tilde{\Omega}_{T}$,  $J_T - 1 \geq \frac{T}{2\mathbb{E}_0[\Delta \tau_1]}$, $\norm{f - f_0}_1 \geq M_T \epsilon_T$ and $M_T' < M_T$. 
For each $J \in \mathcal{J}_T$, we can now apply the Bernstein's inequality:
\begin{align*}
\Probf{\sum_{j=1}^{J-1} (Z_{jl} - \Exf{Z_{jl}})\leq x} \leq \exp \left \{ - \frac{x^2}{2(J-1)(v^2 + bx)}\right\},
\end{align*}
with $x = - \frac{C(f_0)T \norm{f - f_0}_1}{4\mathbb{E}_0[\Delta \tau_1]} $. We first upper bound the term $v^2+bx$:
\begin{align*}
v^2+ b \frac{C(f_0) \norm{f - f_0}_1}{4\mathbb{E}_0[\Delta \tau_1]} &\leq L \max \left( A+1, \frac{2}{\|r_0\|_1} \right)\left(L \max \left( A+1, \frac{2}{\|r_0\|_1} \right) + \frac{C(f_0)}{2 \norm{r_0}_1 \Exz{\Delta \tau_1}} \right) \norm{f - f_0}_1^2\\
&= C_1(f_0) \norm{f - f_0}_1^2, 
\end{align*}
with $ C_1(f_0) := L \max \left( A+1, \frac{2}{\|r_0\|_1} \right)\left(L \max \left( A+1, \frac{2}{\|r_0\|_1} \right) + \frac{C(f_0)}{2 \norm{r_0}_1 \Exz{\Delta \tau_1}}\right)$.
Finally, we  obtain that
\begin{small}
\begin{align*}
    \Probf{\sum_{j=1}^{J-1} (Z_{jl} - \Exf{Z_{jl}})\leq - \frac{C(f_0)T  \norm{f - f_0}_1}{4\mathbb{E}_0[\Delta \tau_1]}\Bigg|\mathcal{G}_0} \leq \exp \left\{-\frac{C(f_0)^2T^2 \norm{f - f_0}_1^2}{8 (J-1)  C_1(f_0)  \norm{f - f_0}_1^2} \right \} \leq  \exp \left\{-\frac{C(f_0)^2 T}{16 C_1(f_0) }\right\},
\end{align*}
\end{small}


and since $\kappa_T\epsilon_T^2=o(1)$, we can conclude that
\begin{align*}
    &\Probf{\Tilde{\Omega}_{T} \cap \{\Tilde{d}_{1T}(f,f_0) \leq M_T' \epsilon_T\} \Big|\mathcal{G}_0}
    \leq \frac{2T}{\Exz{\Delta \tau_1}} \exp \left\{-\frac{C(f_0)^2 T}{16 C_1(f_0) }\right\} = o(e^{-(c_1 + \kappa_T) T \epsilon_T^2}),
\end{align*}
which corresponds to \eqref{eq:sup_ball_l1} and terminates the proof of Theorem \ref{thm:conc_g} and Case 1 of Proposition \ref{prop:relu}.

\subsection{Proof of Case 2 of Proposition \ref{prop:relu}}\label{sec:proof_conc_f}

We recall that in this case we consider a shifted ReLU model with unknown shift $\theta_0 = (\theta_1^0, \dots, \theta_K^0)$, corresponding to a particular case of partially known link functions
$\phi_k(x; \theta_k) = \theta_k + (x)_+,$
and for parameter $f \in \mathcal{F}$ and $\theta \in \Theta$, we denote $\lambda_t(f,\theta)$ the intensity process. We note that in this case, $r_0 = \theta_0 + \nu_0$ and similarly $r_f  = \theta + \nu$, with $r_f$ defined in \eqref{eq:r_0}. We then prove the posterior concentration rate on both $f_0$ and $\theta_0$.
First, we apply the same steps as in the proof of Theorem~\ref{thm:conc_g} in Section ~\ref{sec:proof_conc_g}, replacing $\norm{f - f_0}_1$ by $\norm{r_0 -r_f}_1 + \norm{h - h_0}_1 = \norm{\theta_0 + \nu_0 - \theta - \nu}_1 + \norm{h - h_0}_1$. 
In particular, we re-define the balls w.r.t. the $L_1$-distance as (for simplicity we keep the same notation)
\begin{align*}
\end{align*}
We therefore obtain (see also Remark ~\ref{rem:unknown_link})
\begin{align}\label{eq:partial_conc}
\mathbb{E}_0 \left[\Pi\big(\norm{ h -  h_0}_1 + \norm{\theta_0 + \nu_0 - \theta - \nu}_1 > M \sqrt{\kappa_T} \epsilon_T \big| N \big) \right]  = o(1).
\end{align}


Secondly, we design a test to separate $\theta_0$ and $\nu_0$. For this, we restrict again the set $\eve$ to a high probability set $\Omega_A$, where $\theta_0$ can be correctly estimated. Let
\begin{align*}
&A^k(T) = \{t \in [0,T]; \: \Tilde \lambda_t^k(\nu_0, h_0) < 0 \}, \quad \Omega_A = \{|A^k(T)| > z_0 T, \: \forall k \in \mathcal{K}\}, \quad 1 \leq k \leq K, 
\end{align*}
with $z_0 > 0$ defined in the proof of Lemma \ref{lem:main_event} (see Section \ref{app:proof_lem_event}  in the supplementary material \cite{supplementary}), and define $\eve' = \eve \cap \Omega_A$. Moreover, we define a neighborhood around $\theta_0$, $\bar A(R) := \{\theta \in \Theta ; \: \norm{\theta - \theta_0}_1 \leq R\}$ and $\tilde{M}_T = \Tilde{M} \sqrt{\kappa_T}$ with $ \Tilde{M} > M$.
Using again the decomposition \eqref{decomposition}, with $A = \bar A(\tilde{M}_T \epsilon_T)^c$, $B = A_{L_1}(M_T\epsilon_T)$, 
and the subset $\eve'$,  we thus only need to construct a test function  $\phi \in [0,1]$ verifying:
\begin{align}  \label{test:theta:1}
\Exz{\phi \mathds{1}_{\Tilde{\Omega}_T'}} = o(1) , 
\quad \sup_{\theta \in \bar A(\tilde{M}_T\epsilon_T)^c, f \in A_{L_1}(M_T\epsilon_T)\cap \mathcal F_T} \Exz{\Exf{(1 - \phi)\mathds{1}_{\Tilde{\Omega}_T'}} \Big| \mathcal{G}_0} = o( e^{-(\kappa_T + c_1) T \epsilon_T^2}).
\end{align}

To construct this test, we first consider some arbitrary parameter $ f_1 = ((\nu^1_k)_k, (h_{lk}^1)_{l,k}) \in  A_{L_1}(M_T\epsilon_T) $ and $\theta_1 = (\theta_k^1)_k \in \bar A(\tilde{M}_T\epsilon_T)^c  $, and for any $k \in [K]$, we define the following subset of the observation window
\begin{align}\label{def:I_0_model1}
I_k^0(f_1, \theta_1) = \left \{t \in [0,T]; \: \lambda_t^k(f_1, \theta_1) = \theta_k^1, \, \lambda^k_t(f_0, \theta_0) = \theta_k^0 \right  \}.
\end{align}
By construction $\theta_k^0$ and $\theta_k^1$ can be identified on the set $I_k^0(f_1, \theta_1)$, hence we need $I_k^0(f_1, \theta_1)$ to be large enough in order to test between $\theta_k^0$ and $\theta_k^1$. We can ensure this by defining a controlled set of excursions $\mathcal{E}$. Let $l \in [K]$ such that $h_{lk}^{0-} \neq 0$, $\delta'=(x_2-x_1)/3$ with $x_1, x_2$ defined in condition \eqref{ass:identif_theta}, $c_\star = \min_{x \in [x_1,x_2]} h_{lk}^{0-}(x)$ and  $n_1 = \lfloor 2 \nu_k^1/(\kappa_1 c_\star ) \rfloor +1$ 
for some $0< \kappa_1 <1$. We consider the following subset of excursions:
\begin{align} \label{E:model1} 
\mathcal{E} := \{j \in [J_T]; \: N[\tau_{j}, \tau_{j} + \delta') = N^l[\tau_{j}, \tau_{j} + \delta') = n_1, N[\tau_{j} + \delta', \: \tau_{j+1}) = 0 \},
\end{align}
where the $\tau_j$'s are the regenerative times defined in Lemma \ref{lem:excursions}. Using the intermediate result \eqref{eq:lower_bound_I} from the proof of Lemma \ref{lem:test_theta} in the supplementary material \cite{supplementary}, if $|\mathcal E|$ is large enough, then we can find a lower bound on $|I_k^0(f_1, \theta_1)|$.  We then define our generic test function:
\begin{small}
\begin{align} \label{def:phif1}
    \phi(f_1, \theta_1) := \max_{k \in [K]} \min \left(\mathds{1}_{N^k(I_k^0(f_1, \theta_1)) - \Lambda_k^0(I_k^0(f_1, \theta_1)) < - v_T} \vee \mathds{1}_{|\mathcal{E}| < \frac{p_0 T}{2 \Exz{\Delta \tau_1}}}, \mathds{1}_{N^k(I_k^0(f_1, \theta_1 )) - \Lambda_k^0(I_k^0(f_1, \theta_1)) > v_T} \vee \mathds{1}_{|\mathcal{E}| < \frac{p_0 T}{2\Exz{\Delta \tau_1}}}\right), 
\end{align}
\end{small}
where $p_0 = \Probz{j \in \mathcal E}$, $\Lambda_k^{0}(I_k^0(f_1, \theta_1)) = \int_0^T\mathds{1}_{ I_k^0(f_1, \theta_1)}\lambda^k_t(f_0, \theta_0)dt$, $v_T = w_T T \e_T$, $w_T = 2\sqrt{\max_k \theta_k^0 (\kappa_T + c_1)} + 2x_0$ and $x_0$ from assumption \textbf{(A2)}.
From Lemma \ref{lem:test_theta}, 
there exists $u_1> 2x_0$ and $\zeta \in (0,1) $ such that
\begin{small}
\begin{align}\label{test:fi2}
\Exz{\phi(f_{1}, \theta_1) \mathds{1}_{\Tilde{\Omega}_T'}} \leq e^{ -u_1 T \epsilon_T^2 }, \quad
&\sup_{\|f-f_1\| + \norm{\theta - \theta_1} \leq \zeta \epsilon_T} \Exz{\Exf{(1 - \phi(f_1, \theta_1))\mathds{1}_{\Tilde{\Omega}_T'}} \Big| \mathcal{G}_0} = o( e^{-(\kappa_T + c_1) T \epsilon_T^2}).
\end{align}
\end{small}

To define our global test $\phi$, we first cover  the space $\bar A(\tilde{M}_T\epsilon_T)^c \times A_{L_1}(M_T\epsilon_T)\cap \mathcal F_T$ with $L_1$-balls $\{B_i\}_{1 \leq i \leq \mathcal N}$ of radius $\zeta \epsilon_T$, with $\zeta > 0$ and $\mathcal N \in \mathbb{N}$ the covering number. For each ball $B_i$ centered at $(f_{i}, \theta_i)$, we define the elementary test $\phi(f_{i}, \theta_i)$ as in \eqref{def:phif1}. Then we define $\phi := \max_{i \in \mathcal{N}} \phi(f_{i}, \theta_i)$, and obtain that
\begin{align*}
\Exz{\phi \mathds{1}_{\Tilde{\Omega}_T'}} \leq \mathcal N e^{-u_1 T \epsilon_T^2}, \quad
&\sup_{\theta \in \bar A(\Tilde{M}_T \epsilon_T)^c, f \in A_{L_1}(M_T \epsilon_T)\cap \mathcal F_T}  \Exz{\Exf{(1 - \phi)\mathds{1}_{\Tilde{\Omega}_T'}} \Big| \mathcal{G}_0} = o( e^{-(\kappa_T + c_1) T \epsilon_T^2}).
\end{align*}

Next, we find an upper bound of the covering number $\mathcal{N}$ using assumption \textbf{(A2)}. We note that if $f \in  A_{L_1}(M_T\epsilon_T)$, then for any $( l,k) \in [K]^2$, $
\theta_k \leq \theta_k + \nu_k = r_k^f \leq  r_k^0 + \epsilon_T \leq 2( \theta_k^0 + \nu_k^0)  $. Consequently, using similar computations as in the proof of Proposition \ref{thm:d1t} (see Section \ref{sec:proof_thm_d1t} of the supplementary material \cite{supplementary}), one can find $x_0' > 0$ such that
\begin{align*}
\mathcal{N} &\leq \left(\frac{ 2\max_k(\theta_k^0 + \nu_k^0) }{\zeta \epsilon_T}\right)^K \left(\frac{\max_k \nu_k^0 + \epsilon_T}{\zeta \epsilon_T}\right)^K \mathcal{N}(\zeta \epsilon_T, \mathcal{H}_T, \norm{.}_1) \lesssim e^{- K\log \epsilon_T} e^{x_0' T \epsilon_T^2} \lesssim e^{K\log T} e^{x_0' T \epsilon_T^2} =o( e^{u_1 T \epsilon_T^2}),
\end{align*}
 since $\log T = o(T \epsilon_T^2)$ by assumption. Hence, reporting into \eqref{test:fi2}, this proves that \eqref{test:theta:1} holds and allows us to conclude that
$
\Exz{\Pi(\bar{A}(\Tilde{M}_T\epsilon_T)^c|N)} = \Exz{\Pi(\norm{\theta -\theta_0}_1 > \Tilde{M}\sqrt{\kappa_T} \epsilon_T|N)} = o(1).
$
Finally, since $\Tilde{M} > M$, from \eqref{eq:partial_conc}, we also  have that 
$
\Exz{\Pi(\norm{\nu+\theta - \nu_0 - \theta_0}_1 + \norm{h - h_0}_1 > \Tilde{M} \sqrt{\kappa_T}  \epsilon_T|N)} = o(1),
$.
Therefore it only remains to prove that
$
\Exz{\Pi(\norm{\nu - \nu_0}_1 > \Tilde{M} \sqrt{\kappa_T} \epsilon_T|N)} = o(1).
$
By the triangle inequality, we have
$
\norm{\nu - \nu_0}_1 \leq \norm{\nu+\theta - \nu_0 - \theta_0}_1 + \norm{\theta - \theta_0}_1,
$
and, up to a modification of the constant $\Tilde{M}$,
\small
\begin{align*}
\Exz{\Pi(\norm{\nu-\nu_0}_1 > \Tilde{M} \sqrt{\kappa_T}  \epsilon_T|N)} &\leq  \Exz{\Pi(\norm{\nu+\theta - \nu_0 - \theta_0}_1 > \Tilde{M} \sqrt{\kappa_T}  \epsilon_T|N)} + \Exz{\Pi(\norm{\theta - \theta_0}_1 > \Tilde{M} \sqrt{\kappa_T}  \epsilon_T|N)} = o(1),
\end{align*}
\normalsize
which terminates this proof.

\subsection{Proof of  Theorem \ref{thm:post_graph}} \label{sec:proof_post_graph} 

In this section, we show that in all the models  satisfying the assumptions of Theorem \ref{thm:conc_g} or Proposition ~\ref{prop:relu}, the posterior distribution is consistent on the connectivity graph parameter $\delta_0$. For ease of exposition, we here report the proof for the models considered in Theorem \ref{thm:conc_g}. We first recall the notation $M_T = M \sqrt{\kappa_T}$,
$A_{L_1}(M_T\epsilon_T) = \{ f \in \mathcal{F}; \: \norm{r_f - r_0}_1 + \norm{h - h_0}_1 \leq M_T \epsilon_T\}$, and $I(\delta_0) = \{(l,k) \in [K]^2, \: \delta_{lk}^0 = 1\}$.
We first note that
\begin{small}
\begin{align}\label{eq:Pi_delta}
    \Pi \left(\delta \neq \delta_0 |N \right) &= \Pi \left(\exists (l,k) \in [K]^2, \delta^0_{lk} \neq \delta_{lk} \Big| N \right) \leq \Pi \left(\exists (l,k) \in I(\delta_0), \delta_{lk} = 0 \Big| N \right)
    + \sum_{(l,k)\notin I(\delta_0)} \Pi \left(\delta_{lk} = 1 \Big| N \right).
\end{align}  
\end{small}  
For the first term on the RHS of \eqref{eq:Pi_delta}, using Theorem \ref{thm:conc_g}, we have that 
\begin{align*}
  \Pi \left(\exists (l,k) \in I(\delta_0),  \delta_{lk} = 0 \Big| N \right) &\leq \sum_{(l,k)\in I(\delta_0)} \Pi \left(\{\delta_{lk} = 0\} \cap  A_{L_1}( M_T \epsilon_T) \Big| N \right) + o_{\mathbb{P}_0}(1).
\end{align*}
 For large enough $T$, if $\|h_{lk}^0\|_1 > M_0 M_T \epsilon_T$ with $M_0>1$, then
 \vspace{-3mm}
\begin{align*}
\{f \in \mathcal{F}; \delta_{lk} = 0\} &\subset \{f \in \mathcal{F}; \norm{h^0_{lk} - h_{lk}}_1 = \norm{h^0_{lk}}_1 \} \subset \left\{f \in \mathcal{F}; \norm{h^0_{lk} - h_{lk}}_1 >\frac{  \|h_{lk}^0\|_1}{2} \right\} \subset A_{L_1}(M_T\epsilon_T)^c,  \vspace{-3mm} 
\end{align*}
therefore
$
   \Pi \left( \{\delta_{lk} = 0 \} \cap A_{L_1}(M_T\epsilon_T) \Big| N \right) = 0.
$
For the second term on the RHS of \eqref{eq:Pi_delta}, since $(l,k) \notin I(\delta_0) $ implies that $ \norm{h^0_{lk}}_1 = 0$ and 
$\{\delta_{lk}=1\} \cap A_{L_1}(M_T\epsilon_T) \subset \{f \in \mathcal{F}; \: 0 < \|h_{lk}\|_1 \leq M_T \epsilon_T \}, $
defining
$N_T = \int_{\{\delta_{lk}=1\} \cap A_{L_1}(M_T \epsilon_T)} e^{L_T(f) - L_T(f_0)} d\Pi(f),$ 
and using the decomposition  \eqref{decomposition} with $A = A_{L_1}(M_T\epsilon_T), B = \{\delta_{lk}=1\} $ and $\phi=1$, we obtain that
\small
\begin{align*}
     \mathbb{E}_0 \left[ \Pi(\{\delta_{lk}=1\}\cap A_{L_1}( M_T \epsilon_T) | N) \right]
    & \leq \mathbb{P}_0(D_T < e^{-(\kappa_T+c_1) T \epsilon_T^2}  \cap \eve) + \mathbb{P}_0( \eve^c)+
    e^{(\kappa_T +c_1) T \epsilon_T^2}  \Pi(\{\delta_{lk}=1\} \cap A_{L_1}(M_T\epsilon_T))\\
     &\leq o(1) + e^{(\kappa_T+ c_1) T \epsilon_T^2}\sum_{\delta \in \{0,1\}^{K^2}}  \mathds{1}_{\delta_{lk}=1}\Pi_{h|\delta}\left(\|h_{lk}\|_1 \leq M_T \epsilon_T|\delta \right) = o(1),
\end{align*}
\normalsize
where in the last inequality we have used assumptions \textbf{(A0)-(A1)}, \eqref{eq:cond_prior_graph}, and the construction of the prior from Section \ref{sec:post_consistency}. Consequently, from \eqref{eq:Pi_delta}, we finally arrive at $\Exz{\Pi \left(\delta \neq \delta_0 |N \right)} = o(1)$.

\subsection{Proof of Theorem \ref{cor:graph_estim}}\label{sec:proof_cor_graph}

We here prove the consistency of the penalised estimator defined in \eqref{eq:def_estim}. We consider the models satisfying the assumptions of Theorem ~\ref{thm:conc_g}, although our proof is also valid for the ReLU-type models of Proposition \ref{prop:relu}. Besides, for $f \in \mathcal{F}$, we use the shortened notation 
 $d_{1T} := \Tilde{d}_{1T}(f,f_0)$ and $\hat{\delta}^{\Pi,L} := \hat{\delta}^{\Pi,L} (N)$. We recall that for  $(l,k) \in [K]^2$, $S_{lk} = \norm{h_{lk}}_1$ and the notation from previous proofs,  $M_T =  M\sqrt{\kappa_T}$, $M_T' = M' \sqrt{\kappa_T}$ with $M > M' > 0 $. We first note that
$
    \Probz{\hat{\delta}^{\Pi,L} \neq \delta_0} \leq \sum_{l,k} \Probz{\hat{\delta}^{\Pi, L}_{lk} \neq \delta_{lk}^0}
$
and consider two cases for each $(l,k)$. 
\begin{itemize}
    \item \textbf{Case 1:} $(l,k) \notin I(\delta_0)$, i.e, $\delta_{lk}^0 = 0$. Using \eqref{eq:def_estim} and \eqref{eq:cond_penalisation}, there exists $a > 0$ such that with  $c_1' : = a + c_1 + \kappa_T$, for any $\gamma > 0$, we have
\begin{align}
      \Probz{\hat{\delta}_{lk}^{\Pi,L} \neq \delta_{lk}^0 } &= \Probz{\hat{\delta}_{lk}^{\Pi,L}  = 1} \nonumber \\
      &\leq \Probz{e^{-c_1' T \epsilon_T^2} \Pi(\delta_{lk} = 1, \: S_{lk} \leq M_T \epsilon_T | N) \geq \Pi(\delta_{lk}=0|N) - \Pi(S_{lk}> M_T \epsilon_T | N)}  \nonumber \\ 
   &\leq \Probz{e^{-c_1' T \epsilon_T^2} \Pi(\delta_{lk} = 1, \: S_{lk} \leq M_T \epsilon_T | N) \geq \Pi(\delta_{lk}=0|N)/2  }  \nonumber \\
   &+  \Probz{ \Pi(S_{lk}> M_T \epsilon_T | N) >  \Pi(\delta_{lk}=0|N)/2 }.  \label{eq:X}
\end{align}


To show that the second term in the previous equation is $o(1)$, it is enough to show that
\begin{align}
    &\Probz{ \Pi(d_{1T} > M_T' \epsilon_T | N) > \Pi(\delta_{lk}=0|N)/4} = o(1), \label{eq:U} \\
    &\Probz{ \Pi(d_{1T} \leq M_T' \epsilon_T, \: S_{lk} > M_T \epsilon_T | N) > \Pi(\delta_{lk}=0|N)/4} = o(1) \label{eq:V}.
\end{align}
Let $m_T(\delta_{lk} = 0) := \int_{\mathcal{F}_T} e^{L_T(f) - L_T(f_0)} d\Pi(f|\delta_{lk} = 0)$. Similarly to the computations of the lower bound of $D_T$ in Section \ref{sec:proof_thm_d1t}, we have under \textbf{(A0')} that
$\Probz{ m_T(\delta_{lk} = 0) \leq e^{-\kappa_T' T\epsilon_T^2} } = o(1)$ with $\kappa_T' :=\kappa_T + c_1.$
Using the test function from the proof of Theorem \ref{thm:d1t} in Section \ref{sec:proof_thm_d1t} in the supplementary material \cite{supplementary} $\phi = \max_{i} \phi(f_i)$ (with $\phi(f_i)$ defined in Lemma \ref{lem:tests}) we have
\begin{align*}
    &\Probz{ \Pi(d_{1T} > M_T' \epsilon_T | N) >\Pi(\delta_{lk}=0|N)/4} 
    \leq  \Exz{\phi \mathds{1}_{\eve}} +  \Probz{\eve^c} + \Pi(\mathcal{F}_T^c)\\
    & \quad +\Exz{(1 - \phi)\mathds{1}_{\eve} \mathds{1}_{\int_{\mathcal{F}_T} \mathds{1}_{d_{1T} > M_T' \epsilon_T} e^{L_T(f) - L_T(f_0)} d\Pi(f) >  \Pi(\delta_{lk}=0) m_T(\delta_{lk}= 0)/4}}  \\
    &\leq o(1) 
   +\Exz{(1 - \phi)\mathds{1}_{\eve}\mathds{1}_{\int_{\mathcal{F}_T} \mathds{1}_{d_{1T} > M_T' \epsilon_T} e^{L_T(f) - L_T(f_0)} d\Pi(f) > e^{- \kappa_T' T \epsilon_T^2}/4 }}\\
    &\leq o(1) + 4 e^{  \kappa_T' T \epsilon_T^2}\int_{\mathcal{F}_T}   \Exz{\Exf{\mathds{1}_{\eve}\mathds{1}_{d_{1T} > M_T' \epsilon_T}(1 - \phi)| \mathcal{G}_0} d\Pi(f|\delta_{lk} = 0) }. 
\end{align*}
 In the second inequality, we have notably used \eqref{eq:e_0_phi} $\Exz{\phi \mathds{1}_{\eve}} = o(1)$ from Section \ref{sec:proof_thm_d1t}.
Moreover, from \eqref{eq:e_f_phi}, there exists $\gamma_1>0$ such that 
$$\sum_{i \geq M_T'} \int_{\mathcal{F}_T} \Exf{\mathds{1}_{\Tilde{\Omega}_T} \mathds{1}_{f \in S_i} (1 - \phi)|\mathcal{G}_0} d\Pi(f|\delta_{lk}=0) \leq  4 (2K+1) e^{-\gamma_1  M_T'^2 T \epsilon_T^2},$$
where the $S_i$'s are the slices defined in \eqref{def:slices}. Therefore, 
we obtain \eqref{eq:U} using that
\begin{align*}
    \Probz{ \Pi(d_{1T} > M_T' \epsilon_T | N) >\Pi(\delta_{lk}=0|N)/4} &\leq  o(1) +  4e^{  \kappa_T' T \epsilon_T^2} 4(2K+1) e^{-(M_T')^2 T  \epsilon_T^2} = o(1).
\end{align*}
To prove \eqref{eq:V}, using Markov's inequality and Fubini's theorem, we have, 
for $M'$ large enough, that
\begin{align*}
    &\Probz{ \Pi(d_{1T} \leq M_T' \epsilon_T, \: S_{lk} > M_T\epsilon_T | N) > \Pi(\delta_{lk}=0|N)/4} \\
    & \quad \quad \quad \leq  \Probz{\{m_T(\delta_{lk} = 0) < e^{-\kappa_T' T \epsilon_T^2}\} \cap \eve}  + \Probz{\eve^c} \\
    & \quad \quad \quad +  4e^{  \kappa_T'  T \epsilon_T^2} \Exz{\int_{\mathcal{F}_T \cap \{ S_{lk} > M_T \epsilon_T \}} \mathds{1}_{\eve}\mathds{1}_{d_{1T} \leq M_T' \epsilon_T} e^{L_T(f) - L_T(f_0)} d\Pi(f|\delta_{lk}=0)} \\
    &\quad \quad \quad=  o(1) +4e^{  \kappa_T'  T \epsilon_T^2} \int_{S_{lk} > M_T \epsilon_T} \Exz{\Probf{\eve \cap \{d_{1T} \leq M_T' \epsilon_T\} |\mathcal{G}_0 }} d\Pi(f).
\end{align*}
Moreover, from \eqref{eq:sup_ball_l1}, we have that 
$
\sup_{ f \in A_{L_1}(M_T \epsilon_T)^c \cap \mathcal F_T} \Probf{\eve \cap \{d_{1T} \leq M_T' \epsilon_T \} |\mathcal{G}_0} = o(e^{- \kappa_T' T \epsilon_T^2}).
$
Finally, since $\delta_{lk}^0 = 0$, $ S_{lk} > M_T \epsilon_T $ implies that  $f \in A_{L_1}^c(M_T\epsilon_T)$, which thus leads to \eqref{eq:V}.
Reporting  into \eqref{eq:X}, we now have
\small
\begin{align*}
   \Probz{\hat{\delta}_{lk}^{\Pi,L} = 1} &\leq \Probz{e^{-c_1' T \epsilon_T^2} \Pi(\delta_{lk} = 1, \: S_{lk} \leq M_T \epsilon_T | N) \geq \Pi(\delta_{lk}=0|N)/2} + o(1) \\
   &\leq  \Probz{e^{-c_1' T \epsilon_T^2} \Pi(\delta_{lk} = 1 | N) \geq \Pi(\delta_{lk}=0|N)/2} + o(1) \\
   &=  \Probz{e^{-c_1' T \epsilon_T^2} m_T(\delta_{lk} = 1) \geq \frac{\Pi(\delta_{lk}=0)}{2\Pi(\delta_{lk}=1)} m_T(\delta_{lk} = 0)} + o(1) \\
   &\leq  \Probz{\left \{ e^{-c_1' T \epsilon_T^2} m_T(\delta_{lk} = 1) \geq \frac{\Pi(\delta_{lk}=0)}{2\Pi(\delta_{lk}=1)} e^{- \kappa_T' T \epsilon_T^2} \right \} \cap \eve} + \Probz{m_T(\delta_{lk} = 0) < e^{- \kappa_T' T \epsilon_T^2}} + o(1) \\
      &\leq  \Probz{\left \{m_T(\delta_{lk} = 1) \geq \frac{\Pi(\delta_{lk}=0)}{2\Pi(\delta_{lk}=1)} e^{(c_1'- \kappa_T') T \epsilon_T^2} \right \} \cap \eve} + o(1) \\
      &\leq  \Exz{m_T(\delta_{lk} = 1)} \frac{2\Pi(\delta=1)}{\Pi(\delta_{lk}=0)} e^{-(c_1'-\kappa_T') T \epsilon_T^2} + o(1) \leq \frac{2\Pi(\delta_{lk}=1)^2}{\Pi(\delta_{lk}=0)} e^{-(c_1'-\kappa_T') T \epsilon_T^2} + o(1) = o(1),
\end{align*}
\normalsize
since $c_1' > \kappa_T + c_1 = \kappa_T'$ and that $ \Exz{m_T(\delta = 1)} = \Pi(\delta_{lk} = 1)$ with Fubini's theorem.

\item \textbf{Case 2:} $(l,k) \in I(\delta_0)$, i.e, $\delta_{lk}^0 = 1$. In the case, the computations are slightly simpler since $ \{\delta_{lk} = 0\} \implies f \in A_{L_1}(M\sqrt{\kappa_T}\epsilon_T)^c$ and for $T$ large enough, $S_{lk}^0 - M_T\epsilon_T > 0$. Thus we can use the fact that $\Pi(\delta_{lk} = 0 | N) \leq \Pi( A_{L_1}(M\sqrt{\kappa_T}\epsilon_T)^c | N)$ (the full computations are reported in Section \ref{sec:proof_cor_graph_suite} in the supplementary material \cite{supplementary}.

\end{itemize}




\section{Conclusion}

In this paper we have established concentration and consistency properties of the posterior distribution and of Bayesian estimators of the parameter and connectivity graph, in a general class of nonlinear Hawkes processes. These results validate the common use of these models in different applied contexts. In particular, our results include the commonly used sigmoid and softplus models, as well as the more challenging ReLU model, under some additional restrictions on the parameter space. Moreover, we provide the first theoretical results for estimating an additional parameter of the link functions, in the case of shifted ReLU with unknown shift. To prove those results, we have built a new technique for obtaining model identifiability and concentration inequalities based on the decomposition of the process into excursions, recently introduced by Costa et al. \cite{costa18}. Finally, our results hold under reasonable assumptions on the prior distribution and the true model, and we provide practical examples for which those conditions are verified. 



Although rather weak assumptions have been used to prove our results, it is likely that the latter hold in more general contexts. In particular, we believe that one could relax the condition on processes with bounded memory ($A < + \infty$) since the regenerative properties of the nonlinear Hawkes processes also hold for processes with unbounded memory. One major improvement of our results would be to consider high dimensional processes ($K \to \infty$), possibly in restricted models such as sparse models \cite{bacry2015sparse} or clustering models \cite{raad2020stability}. Another perspective would be to prove the frequentist minimax rate of estimation, since it would be of great interest to evaluate the optimality of Bayesian procedures in nonlinear Hawkes processes. Some practitioners might also be interested in additional results on the estimation of the link function, through a different parametric or even nonparametric form, like in \cite{wang2016isotonic}.

\begin{appendix}

%
%

\section{Main lemmas}\label{app:main_lemmas}

In this section, we state some important lemmas to prove our main results in Section \ref{sec:proof}. The proofs of Lemmas  \ref{lem:main_event}, \ref{lem:KL}, \ref{lem:ef} and  \ref{lem:test_theta} are provided in Sections \ref{sec:proof_thm_d1t}, \ref{sec:technical_lemmas}  and \ref{app:proof_bay_lemmas} in the supplementary material  \cite{supplementary}.  The first two lemmas are controls respectively of  the complement of the main event $\eve$ under the true distribution $\mathbb{P}_0$, and of  the deviations of the  log likelihood ratio $L_T(f_0) - L_T(f)$. 

\begin{lemma}\label{lem:main_event}
Let $Q > 0$. We consider $\Tilde{\Omega}_T$ defined in \eqref{def:Omega} in Section \ref{sec:proof_conc_g}. For any $\beta > 0$, we can choose $C_\beta$ and $c_\beta$ in the definition of $\Tilde{\Omega}_T$ such that
$
	\probz{\Tilde{\Omega}_T^c} \leq T^{-\beta}.
$
Moreover, for any $1 \leq q \leq Q$,
$
\Exz{\mathds{1}_{\eve^c} \max_l \sup \limits_{t \in [0,T]} \left(N^l[t-A,t)\right)^q} \leq 2 T^{-\beta/2}. 
$
Finally, the previous results hold when replacing $\eve$ by $\eve' = \eve \cap \Omega_A$ with $\Omega_A$ defined in Section \ref{sec:proof_conc_f} for the model with shifted ReLU link and unknown shift. 
\end{lemma}
\begin{lemma}\label{lem:KL}
Under the assumptions 
of Theorem \ref{thm:conc_g} or Proposition \ref{prop:relu}, for any  $f \in B_\infty(\epsilon_T)$ and $T$ large enough, we have
\begin{equation*}
    \mathbb{P}_0 \left[L_T(f_0) - L_T(f) \geq \frac{1}{2} \kappa_T T \e_T^2 \right] = o(1).
\end{equation*}
with 
\vspace{-0.8cm}
\begin{align*}
    &\kappa_T =  \begin{cases}
   10 & \text{(under Assumption \ref{ass-psi}(i))} \\
    10 (\log T)  & \text{(under Assumption \ref{ass-psi}(ii) )}\\
   10 (\log T)^2 & \text{(under Case 1 and condition \eqref{condA3})}
    \end{cases}
\end{align*}
\end{lemma}
\begin{remark}
Contrary to the typical approach, the proof of Lemma \ref{lem:KL} is not based on the control of the variance of $L_T(f_0) - L_T(f)$, which is intractable due to the nonlinear form of the log-likelihood function, but on a decomposition of $L_T(f_0) - L_T(f)- KL(f_0,f)$ into a sum of i.i.d.  terms $T_j$ defined as:
\begin{align*}
    T_j := \sum_k\int_{\tau_j}^{\tau_{j+1}}\log \left(\frac{\lambda^k_t(f_0)}{\lambda^k_t(f)} \right) dN_t^k - \int_{\tau_j}^{\tau_{j+1}} (\lambda^k_t(f_0) - \lambda^k_t(f)) dt.
\end{align*}
\end{remark}
The next lemma is a notably used in the proof of Theorem \ref{thm:conc_g} in Section \ref{sec:proof_conc_g} and bridges the gap between the posterior concentration rate in stochastic distance (see Theorem \ref{thm:d1t}) and the rate in $L_1$-distance (Theorem \ref{thm:conc_g}).

\begin{lemma}\label{lem:ef}
For  $f \in \mathcal{F}_T$ and $l \in [K]$, let 
\begin{equation*}
    Z_{1l} = \int_{\tau_1}^{\xi_1} |\lambda^l_t(f) - \lambda^l_t(f_0)|dt, 
\end{equation*}
where $\xi_1$ is defined in \eqref{def:A2} in Section \ref{sec:proof_conc_g}. 
Under the assumptions of Theorem \ref{thm:conc_g} and Case 1 of Proposition \ref{prop:relu}, for $M_T \to \infty$ such that $M_T > M \sqrt{\kappa_T}$ with $M>0$ and for any $f \in \mathcal{F}_T$ such that $\norm{\nu-\nu_0}_1 \leq \max(\norm{\nu_0}_1, \Tilde{C})$ with $\Tilde{C}>0$,
there exists $l \in [K]$ such that on $\Tilde{\Omega}_{T}$,
    \begin{equation*}
        \Exf{Z_{1l}} \geq C(f_0)  \norm{f - f_0}_1, 
    \end{equation*}
with $C(f_0) > 0$ a constant that depends only on $f_0$ and $\phi=(\phi_k)_k$.

Similarly, under the assumptions of  Case 2 of Proposition \ref{prop:relu}, for $f \in \mathcal{F}_T$ and $\theta \in \Theta$, let $r_0 = (r_k^0)_k, \: r_f = (r_k^f)_k$ with $r_k^0 = \phi_k(\nu_k^0) = \theta_k^0 + \nu_k^0, \: r_k^f = \phi_k(\nu_k) = \theta_k + \nu_k, \:  \forall k$. If $\norm{r_f - r_0}_1 \leq \max(\norm{r_0}, \Tilde{C}')$ with $\Tilde{C}'>0$, then there exists $l \in [K]$ such that on $\Tilde{\Omega}_{T}$,
    \begin{equation}\label{eq:ef}
        \Exf{Z_{1l}} \geq C'(f_0) (\|r_f - r_0\|_1 + \norm{h - h_0}_1), \quad C'(f_0) > 0.
    \end{equation}
\end{lemma}

Finally, this last lemma provides upper bounds on type I and type II errors for the tests used in the proof of Case 2 of Proposition \ref{prop:relu} in Section \ref{sec:proof_conc_f} for estimating the parameter of the link functions $\theta_0$.
\begin{lemma}\label{lem:test_theta}
Using the notations of Section \ref{sec:proof_conc_f}, for $\theta_1 \in \bar A(\tilde{M}_T \epsilon_T)^c, \: f_1 \in A_{L_1}(M_T\epsilon_T) \cap \mathcal{F}_T$, we define
\small
\begin{align*}
    \phi(f_1, \theta_1) = \max_{k \in [K]} \min \left(\mathds{1}_{N^k(I_k^0(f_1, \theta_1)) - \Lambda_k^0(I_k^0(f_1, \theta_1)) < - v_T} \vee \mathds{1}_{|\mathcal{E}| < \frac{p_0 T}{2 \Exz{\Delta \tau_1}}}, \mathds{1}_{N^k(I_k^0(f_1 , \theta_1)) - \Lambda_k^0(I_k^0(f_1, \theta_1), f_0) > v_T} \vee \mathds{1}_{|\mathcal{E}| < \frac{p_0 T}{2\Exz{\Delta \tau_1}}}\right),
\end{align*}
\normalsize
with $I_k^0(f_1, \theta_1)$ and $\mathcal{E}$ defined in \eqref{def:I_0_model1} and \eqref{E:model1}, $p_0 = \Probz{j \in \mathcal{E}}$, $\Lambda_k^0(I_0^k(f_1, \theta_1)) = \int_0^T \mathds{1}_{I_k^0(f_1, \theta_1)} \lambda_t^k(f_0, \theta_0)dt$ and $v_T = w_T T \e_T$ with $w_T = 2\sqrt{\max_k \theta_k^0 (\kappa_T + c_1)} + 2x_0$ and $x_0 > 0$. Then  there exists $u_1>2x_0$ such that 
\begin{align*}
\Exz{\phi(f_{1}, \theta_1) \mathds{1}_{\Tilde{\Omega}_T}'} \leq e^{-u_1 T\epsilon_T^2},
\quad
&\sup_{\norm{\theta - \theta_1} + \|f-f_1\|\leq \zeta \epsilon_T } \Exz{\Exf{(1 - \phi(f_1, \theta_1))\mathds{1}_{\Tilde{\Omega}_T'}} \Big| \mathcal{G}_0} = o( e^{-(\kappa_T + c_1) T \epsilon_T^2}).
\end{align*}
\end{lemma}
\end{appendix}



\vspace{-0.5cm}

{\bf Acknowledgements:}
The project leading to this work has received funding from the European Research Council (ERC) under the European Union’s Horizon 2020 research and innovation programme (grant agreement No 834175). The project is also partially funded by the EPSRC via the CDT OxWaSP. The authors would like to thank the Editor and two anonymous referees for valuable  comments and suggestions.

\bibliographystyle{plainnat} 
\bibliography{bib}       

\makeatletter\@input{s_aux.tex}\makeatother

\end{document}


\renewcommand{\thepage}{S\arabic{page}}
\renewcommand{\thesection}{S\arabic{section}}
\renewcommand{\thetable}{S\arabic{table}}
\renewcommand{\thefigure}{S\arabic{figure}}
\renewcommand{\theequation}{S\arabic{section}.\arabic{equation}}
\newcommand{\ds}[1]{\textcolor{black}{#1}}

\hypersetup{hypertexnames=false,citecolor=blue,urlcolor=blue,linkcolor=blue,filecolor=blue,breaklinks=true}

\title{Supplementary material of Bayesian estimation of nonlinear Hawkes processes}

\maketitle

This supplementary material contains additional results and proofs that could not be included in the main paper \cite{main} due to space limitations. In Section \ref{sec:proof_thm_d1t}, we report the proofs of Theorem \ref{thm:d1t} and Lemma  \ref{lem:KL}. Section \ref{sec:supp:graph} contains the proofs of two results in the graph estimation problem (second part of Theorem \ref{cor:graph_estim} and Proposition  \ref{prop:graph_restr}). In Section \ref{sec:proof_cor_post_mean} we prove frequentist results of Corollary~\ref{cor:post_mean}. Results regarding the construction of prior distributions can be found in Section \ref{sec:prior_dist}. In Sections \ref{sec:technical_lemmas}, \ref{sec:lem:KL}  and \ref{app:proof_technical_lemmas}  we report additional technical results and their proofs, notably on the tests used in the main theorems and on the Kullback-Leibler divergence defined for the Hawkes model.  Lemmas \ref{lem:main_event} and \ref{lem:ef} are proved in Section \ref{app:proof_bay_lemmas}. Finally, we report multivariate extensions of existing results on the regenerative properties of the nonlinear Hawkes model in Section \ref{sec:results_costa}.


For the sake of simplicity,  all sections, theorems, corollaries, lemmas and equations presented in the supplement are designed with a prefix S. Regarding the others, we refer to the material of the main text  \cite{main}. This is not specified at each place.


\section{Proofs of Theorem \ref{thm:d1t} and of Lemma \ref{lem:KL} }\label{sec:proof_thm_d1t}

\subsection{Proof of Theorem \ref{thm:d1t}} \label{sec:proof_thm_d1t1}
This section contains the proof of  the posterior concentration rate w.r.t. the stochastic distance defined in \eqref{tilded1} in \cite{main}. We use the well-known strategy of \cite{ghosal:vdv:07} which has the following steps. First, the space of observations is restricted to a subset $\Tilde{\Omega}_T$ defined in \eqref{def:Omega} which has high probability (see Lemma \ref{lem:main_event}). Secondly, we use a lower bound of the denominator $D_T$ defined in \eqref{def:pposterior_dist} using Lemma \ref{lem:KL}. Thirdly, we consider $A_{d_1}(M_T' \epsilon_T) \subset \mathcal{F}$, the ball centered at $f_0$ of radius $M_T' \epsilon_T$ w.r.t the auxiliary stochastic distance $\Tilde{d}_{1T}$. To find an upper bound of the numerator $N_T(A_{d_1}(M_T' \epsilon_T)^c)$ as defined in \eqref{def:pposterior_dist}, $A_{d_1}(M_T' \epsilon_T)^c$ is partitioned into slices $S_i$ on which we can design tests that have exponentially decreasing type I and type II errors (see Lemma \ref{lem:tests}). We then define $\phi$ as the maximum of the tests on the individual slices $S_i$. Note that the following proof applies to all estimation scenarios, and for generality here, we consider $\theta_0$ unknown.

We recall the notation 
$A_{d_1}(\epsilon) = \{f \in \mathcal{F}; \: \Tilde{d}_{1T}(f,f_0) \leq  \epsilon \}.$ and from \eqref{def:pposterior_dist},
$D_T =  \int_{\mathcal{F}} e^{L_T(f) - L_T(f_0)} d\Pi(f).$
For a sequence $\e_T$ verifying the assumptions of Theorem \ref{thm:conc_g} and for $i \geq 1$, we denote
\begin{align}\label{def:slices}
S_i = 
\{f \in \mathcal{F}_T ; \: K i \epsilon_T \leq  \Tilde{d}_{1T}(f,f_0) \leq K(i+1) \epsilon_T\},
\end{align}
where $\mathcal{F}_T = \{f = (\nu, h) \in \mathcal{F}; \: h = (h_{lk})_{l,k} \in \mathcal{H}_T, \: \nu \in \Upsilon_T \}$.
Let  $M_T' = M' \sqrt{\kappa_T}$ with $M' > 0$ and $\kappa_T$ defined in \eqref{kappaT}. Using the  decomposition \eqref{decomposition} with $A = A_{d_1}(M_T' \epsilon_T)^c$ (and $B = \mathcal{F}$) , for any test function $\phi \in [0,1]$, we have
\begin{align}
    \mathbb{E}_0[\Pi(A_{d_1}(M_T' \epsilon_T)^c|N)] &\leq \mathbb{P}_0(\Tilde{\Omega}_{T}^c) + \mathbb{P}_0\left(\{D_T < e^{-\kappa_T  T \epsilon_T^2} \Pi(B_\infty(\epsilon_T))\} \cap \eve\right) + \mathbb{E}_0[\phi \mathds{1}_{\Tilde{\Omega}_{T}}] \nonumber \\
    &+ \frac{e^{\kappa_T T \epsilon_T^2} }{\Pi(B_\infty(\epsilon_T))}\Pi(\mathcal{F}_T^c) + \frac{e^{\kappa_T T \epsilon_T^2} }{\Pi(B_\infty(\epsilon_T))} \left( + \sum_{i=M_T'}^{+\infty} \int_{\mathcal{F}_T} \mathbb{E}_0\left[\Exf{\mathds{1}_{\Tilde{\Omega}_{T}} \mathds{1}_{f \in S_i} (1-\phi)] | \mathcal{G}_0}\right] d\Pi(f)\right) \label{eq:post_e1}.
\end{align}
For the first term on the RHS of \eqref{eq:post_e1}, we have $\mathbb{P}_0(\Tilde{\Omega}_{T}^c) = o(1)$ by Lemma \ref{lem:main_event} in \cite{main}. For the fourth term of the RHS of \eqref{eq:post_e1}, under  \textbf{(A0)} and \textbf{(A1)}, we have that
\begin{equation*}
    \frac{e^{\kappa_T T \epsilon_T^2} }{\Pi(B_\infty(\epsilon_T))} \Pi(\mathcal{F}_T^c) \leq e^{(\kappa_T + c_1) T \epsilon_T^2} (\Pi(\mathcal{H}_T^c) + \Pi(\Upsilon_T^c)) = o(1).
\end{equation*}
The second term of \eqref{eq:post_e1} is controlled  by \eqref{DT} and goes to 0.

 We now deal with the  third and fifth  terms on the RHS of \eqref{eq:post_e1}, which require to define a suitable test function $\phi$. Let $i \in \N, i\geq M_T'$ and $f \in S_i$. On $\Tilde{\Omega}_{T}$, with $A_2(T)$ defined in \eqref{def:A2}, we have that
\begin{align*}
     T \Tilde{d}_{1T}(f, f_0) &= \sum_{l=1}^K \int_{A_2(T)} \left| \lambda_l^k(f)  - \lambda_l^k(f_0) \right| dt = \sum_{l=1}^K \sum_{j=1}^{J_T-1} \int_{\tau_j}^{\xi_j} \left| \lambda_l^k(f)  - \lambda_l^k(f_0) \right| dt \\
     &\geq  \sum_{l=1}^K \sum_{j=1}^{J_T-1} \int_{\tau_j}^{U_j^{(1)}} |r_l^f - r_l^0|dt  \geq \sum_{j=1}^{J_T-1} (U_j^{(1)} - \tau_j)  \sum_l |r_l^f - r_l^0| \geq \frac{T}{2 \norm{r_0}_1 \Exz{\Delta \tau_1}} \sum_l |r_l^f - r_l^0|,
\end{align*}
with $r_f = (\phi_1(\nu_1), \dots, \phi_K(\nu_K)),  \: r_0  = (\phi_1(\nu_1^0), \dots, \phi_K(\nu_K^0)) $ and $\tau_j, \: \xi_j, \: U_j^{(1)}, \: 1 \leq j \leq J_T-1$ defined in Sections~\ref{sec:lem_excursions} and \ref{sec:proof_conc_g} of \cite{main}. Consequently, for any $l \in [K]$, since $\Tilde{d}_{1T}(f, f_0) \leq K(i+1)\epsilon_T$, we obtain that
\begin{align}\label{eq:bound_d1T}
    r_l^f \leq r_l^0 + 2K(i+1) \norm{r_0}_1 \Exz{\Delta \tau_1} \epsilon_T \leq r_l^0 + 1 + 2K \norm{r_0}_1 \Exz{\Delta \tau_1} i \epsilon_T,
\end{align}
for $T$ large enough. Moreover, using Assumption~\ref{ass-psi}, $\phi_l^{-1}$ is $L'$-Lipschitz on $J_l = \phi_l(I_l)$ and $r_l^0 \in J_l$. With $\varepsilon > 0$ from Assumption~\ref{ass-psi}, we now separate the set of indices $i$ in two subsets.

\textbf{Case 1:} $i$ is such that $2 L' \norm{r_0}_1 \Exz{\Delta \tau_1} K (i+1) \epsilon_T < \varepsilon$. Then we have that $r_l^f \in J_l$ and $\nu_l \in I_l$ since $
|r_l^f - r_l^0|  = \leq 2 \norm{r_0}_1 \Exz{\Delta \tau_1} K (i+1) \epsilon_T$. Consequently, $\frac{1}{L'} |\nu_l - \nu_l^0| \leq |r_l^f - r_l^0| \leq L |\nu_l - \nu_l^0|$ and in particular,
\begin{align*}
    \nu_l \leq \nu_l^0 + 2K L' (i+1) \norm{r_0}_1 \Exz{\Delta \tau_1} \epsilon_T .
\end{align*}
Defining
$$\mathcal{F}_{i} = \left \{ f  \in \mathcal{F}_T; \: \nu_l^f  \leq \nu_l^0 + 1 + 2K L' \norm{r_0}_1 \Exz{\Delta \tau_1} i \epsilon_T, \forall l \in [K] \right \}, $$
we therefore have that for any $f \in S_i$ and $T$ large enough, $f \in \mathcal{F}_{i}$
Let $(f_{i,n})_{n=1}^{\mathcal{N}_i}$ be the centering points of a minimal $L_1$-covering of $\mathcal{F}_i$ by $\mathcal{N}_i$ balls of radius $\zeta i \epsilon_T$ with $\zeta = 1/(6N_0)$, and $N_0$ defined in the proof of Lemma \ref{lem:tests} in Section \ref{app:proof_test}.  There exists $C_0 > 0$ such that we have
\begin{align*}
    \mathcal{N}_i \leq \left(\frac{C_0(1+i\epsilon_T)}{\zeta i \epsilon_T/2}\right)^K   \mathcal{N}(\zeta i \epsilon_T/2, \mathcal{H}_T, \norm{.}_1) .
\end{align*}
If $i \epsilon_T \leq 1$,
\begin{align*}
    \mathcal{N}_i \leq \left(\frac{4 C_0}{\zeta i \epsilon_T}\right)^K   \mathcal{N}(\zeta i \epsilon_T/2, \mathcal{H}_T, \norm{.}_1) = \left(\frac{4 C_0}{\zeta}\right)^K e^{-K \log(i\epsilon_T)}   \mathcal{N}(\zeta i \epsilon_T/2, \mathcal{H}_T, \norm{.}_1).
\end{align*}
Otherwise, if $i\epsilon_T \geq 1$,
$$\mathcal{N}_i \leq \left( \frac{4 C_0}{\zeta}\right)^K   \mathcal{N}(\zeta i \epsilon_T/2, \mathcal{H}_T, \norm{.}_1).
$$


Moreover, since $i \mapsto  \mathcal{N}(\zeta i \epsilon_T/2, \mathcal{H}_T, \norm{.}_1)$ is non-increasing, and if $i \geq 2 \zeta_0/\zeta$, we have that $\mathcal{N}(\zeta i \epsilon_T/2, \mathcal{H}_T \norm{.}_1) \leq \mathcal{N}(\zeta_0 \epsilon_T, \mathcal{H}_T, \norm{.}_1) \leq e^{x_0 T \epsilon_T^2}$ using \textbf{(A2)}. Consequently, since $\epsilon_T > \epsilon_T^2 > \frac{1}{T}$ when $T$ is large enough, $e^{-\log (i \epsilon_T)} \leq e^{\log(\frac{\zeta}{2\zeta_0} T)}$ and we obtain
\begin{align*}
    \mathcal{N}_i &\leq \left( \frac{4 C_0}{\zeta}\right)^K   \left(\frac{\zeta}{2\zeta_0}\right)^K e^{K \log T}    \mathcal{N}(\zeta i \epsilon_T/2, \mathcal{H}_T, \norm{.}_1) =  \left( \frac{2 C_0}{\zeta_0}\right)^K   e^{K \log T}    \mathcal{N}(\zeta i \epsilon_T/2, \mathcal{H}_T, \norm{.}_1) \\ &\leq  C_K e^{K \log T}   e^{x_0 T \epsilon_T^2},
\end{align*}
denoting $C_K =  \left( \frac{2 C_0}{\zeta_0}\right)^K$.

\textbf{Case 2:} $2 L' \norm{r_0}_1 \Exz{\Delta \tau_1} K (i+1) \epsilon_T > \varepsilon$. Then in this case we define $\mathcal{F}_i = \mathcal{F}_T$ and $\nu_T = e^{c_2 T \epsilon_T^2}$, and the $L_1$-covering number of $\mathcal{F}_i$ is now upper bounded by
\begin{align*}
 \mathcal{N}_i \leq \left(\frac{ \nu_T }{\zeta i \epsilon_T/2}\right)^K   \mathcal{N}(\zeta i \epsilon_T/2, \mathcal{H}_T, \norm{.}_1) \leq C_0' e^{(x_0+c_2 K) T \epsilon_T^2},
\end{align*}
with $C_0' > 0$ a constant.

In both cases, considering the tests $\phi_{i} = \max \limits_{n \in [\mathcal{N}_i]} \phi_{f_{i,n}}$ with $\phi_{f_{i,n}}$, $\gamma_1 = \min_l x_{1l}$ defined in Lemma \ref{lem:tests}, and $C_K' = C_K \vee C_0'$, $x_0' = x_0 + c_ K$, we have
\begin{align*}
    &\mathbb{E}_0[\mathds{1}_{\Tilde{\Omega}_{T}}\phi_i] \leq \mathcal{N}_i  e^{-\gamma_1 T (i^2 \epsilon_T^2 \wedge i \epsilon_T)} \leq  C_K'  (2K + 1) e^{K \log T}   e^{x_0' T \epsilon_T^2} e^{-\gamma_1 T (i^2 \epsilon_T^2 \wedge i \epsilon_T)}, \\
    &\Exz{\Exf{\mathds{1}_{\Tilde{\Omega}_T} \mathds{1}_{f \in S_i} (1 - \phi_{i})|\mathcal{G}_0}}  \leq (2K + 1) e^{-\gamma_1 T (i^2 \epsilon_T^2 \wedge i \epsilon_T)}.
\end{align*}
Choosing $\phi = \max \limits_{M_T' \leq i \leq \mathcal{N}_i} \phi_i$ and since $M_T' \geq 2 \zeta_0/\zeta$ for $T$ large enough, we obtain
\begin{align}\label{eq:e_0_phi}
     \mathbb{E}_0[\mathds{1}_{\Tilde{\Omega}_{T}}\phi] &\leq  C_K' (2K+1) e^{K \log T}   e^{x_0' T \epsilon_T^2} \left[ \sum_{i=M_T'}^{\epsilon_T^{-1}}e^{-\gamma_1 i^2 T \epsilon_T^2 } + \sum_{i > \epsilon_T^{-1}} e^{-\gamma_1   i T \epsilon_T} \right] \nonumber \\
     &\leq  C_K' (2K+1) e^{K \log T}   e^{x_0' T \epsilon_T^2} \left[ \sum_{i=M_T'}^{\epsilon_T^{-1}}e^{-\gamma_1  i M_T' T \epsilon_T^2 } + \sum_{i > \epsilon_T^{-1}} e^{-\gamma_1 T  i \epsilon_T} \right] \nonumber \\
     &\leq  C_K'  (2K+1) e^{K \log T}   e^{x_0' T \epsilon_T^2} \left[ 2 e^{-\gamma_1 M_T'^2 T \epsilon_T^2 } + 2 e^{-\gamma_1 T} \right] \nonumber \\
      &\leq 4 C_K'  (2K+1) [e^{-\gamma_1  M_T'^2 T \epsilon_T^2} + e^{-\gamma_1T}],
\end{align}
since $\log^3 T = O(T \epsilon_T^2)$ by assumption. Therefore, we arrive at $ \mathbb{E}_0[\mathds{1}_{\Tilde{\Omega}_{T}}\phi] = o(1)$. Similarly, we can obtain 
\begin{align*}
    \Exz{\sum_{i \geq M_T'} \int_{\mathcal{F}_T} \Exf{\mathds{1}_{\Tilde{\Omega}_T} \mathds{1}_{f \in S_i} (1 - \phi)|\mathcal{G}_0} d\Pi(f)} &\leq (2K + 1) \left[ \sum_{i=M_T'}^{\epsilon_T^{-1}}e^{-\gamma_1 i^2 T \epsilon_T^2 } + \sum_{i > \epsilon_T^{-1}} e^{-\gamma_1 T  i \epsilon_T} \right] \nonumber \\
    &\leq  4 (2K+1)[e^{-\gamma_1 M_T'^2 T \epsilon_T^2} + e^{-\gamma_1T}].
\end{align*}

Therefore, using \textbf{(A0)}, we have for the second term in \eqref{eq:post_e1},
\small
\begin{equation} \label{eq:e_f_phi}
\begin{split}
    \frac{e^{\kappa_T T \epsilon_T^2} }{\Pi(B_\infty(\epsilon_T))} & \left( \sum_{i=M_T'}^{+\infty} \int_{\mathcal{F}_T} \mathbb{E}_0\left[\Exf{\mathds{1}_{\Tilde{\Omega}_{T}} \mathds{1}_{f \in S_i} (1-\phi) | \mathcal{G}_0}\right] d\Pi(f)\right) \leq  \frac{e^{\kappa_T T \epsilon_T^2} }{e^{-c_1 T \epsilon_T^2}} 4 (2K+1) [e^{-\gamma_1  M_T'^2 T \epsilon_T^2} + e^{-\gamma_1T}] \\
    &\leq 4 (2K+1)e^{-\gamma_1  M_T'^2 T \epsilon_T^2/2}  = o(1),
\end{split}
\end{equation}
\normalsize
for $M_T' > \sqrt{c_1 + \kappa_T}$, which holds true if $M_T' = M' \sqrt{\kappa_T}$ with $M'$ large enough. Aggregating the upper bounds previously obtained, we can finally conclude that
\begin{align*}
     \mathbb{E}_0[\Pi(A_{d_1}(M_T' \epsilon_T)^c|N)] \leq \mathbb{P}_0(\Tilde{\Omega}_{T}^c) + o(1) = o(1),
\end{align*}
which terminates the proof of Theorem \ref{thm:d1t}.

\subsection{Proof of Lemma \ref{lem:KL}} \label{sec:prlem:KL3}

\ds{In this section, we prove a control on the log-likelihood ratio of the form
$\Probz{ L_T(f_0) - L_T(f)  \geq 5 z_T } = o(1) $, where  $z_T= T \epsilon_T^2 (\log T)^r$ where $r = 0,1,2$ is defined in Lemma \ref{lem:t_j} and depends on the assumptions on the link function.} We have
\begin{align*}
L_T(f_0) - L_T(f)  &= \sum_k \int_{0}^{T} \log \left(\frac{\lambda^k_t(f_0)}{\lambda^k_t(f)} \right) dN_t^k - \int_{0}^{T} (\lambda^k_t(f_0) - \lambda^k_t(f)) dt\\
&= W_0 + \sum_{j=1}^{J_T-1} T_j +  W_T,
\end{align*}
with
\begin{align}
& W_0 := \sum_k\int_{0}^{\tau_1}\log \left(\frac{\lambda^k_t(f_0)}{\lambda^k_t(f)} \right) dN_t^k - \int_0^{\tau_1}(\lambda^k_t(f_0) - \lambda^k_t(f)) dt, \nonumber \\
& W_T := \sum_k\int_{\tau_{J_T}}^{T}\log \left(\frac{\lambda^k_t(f_0)}{\lambda^k_t(f)} \right) dN_t^k - \int_{\tau_{J_T}}^{T}(\lambda^k_t(f_0) - \lambda^k_t(f)) dt. \nonumber
\end{align}
Let $\mathcal{L}_T = L_T(f_0 ) - L_T(f) - \Exz{L_T(f_0 ) - L_T(f)} = L_T(f_0 ) - L_T(f) - KL(f_0,f)$, with $KL(f_0,f)$ the Kullback-Leibler divergence defined in \eqref{eq:def_KL_brute}. Then 
\begin{equation}\label{borne-p0}
\begin{split}
&\Probz{\mathcal{L}_T \geq  4z_T} = \Probz{\sum_{j=1}^{J_T-1} T_j + W_0+W_T - KL(f_0,f) \geq 4 z_T }  \\
&= \Probz{\sum_{j=1}^{J_T-1} (T_j  - \Exz{T_j}) +  \sum_{j=1}^{J_T-1} \Exz{T_j} - \Exz{\sum_{j=1}^{J_T-1} T_j} +W_T - \Exz{W_T} + W_0 - \Exz{W_0} \geq 4z_T } \\ 
&\leq \Probz{\sum_{j=1}^{J_T-1} T_j  - \Exz{T_j} \geq  z_T} + \Probz{  (J_T - \Exz{J_T}) \Exz{T_1} - \Exz{\sum_{j=0}^{J_T-1} T_j - \Exz{T_j}}\geq z_T } + \Probz{ W_T - \Exz{W_T} \geq z_T } \\
&+ \Probz{ W_0 - \Exz{W_0} \geq z_T },
\end{split}
\end{equation}
using equation \eqref{eq:def_KL_dev} and that
\begin{align*}
    KL(f_0,f) &= \underbrace{\sum_k \Exz{\int_{\tau_{0}}^{\tau_1}\log \left(\frac{\lambda^k_t(f_0)}{\lambda^k_t(f)} \right) dN_t^k - \int_{0}^{\tau_1}(\lambda^k_t(f_0) - \lambda^k_t(f)) dt}}_{\Exz{W_0}} \\
    &+ \underbrace{\sum_k \Exz{\int_0^{\tau_{J_T}}\log \left(\frac{\lambda^k_t(f_0)}{\lambda^k_t(f)} \right) dN_t^k - \int_0^{\tau_{J_T}}(\lambda^k_t(f_0) - \lambda^k_t(f)) dt}}_{= \Exz{ \sum_{j=1}^{J_T-1} T_j}} \\
    &+ \underbrace{\sum_k \Exz{\int_{\tau_{J_T}}^{T}\log \left(\frac{\lambda^k_t(f_0)}{\lambda^k_t(f)} \right) dN_t^k - \int_{\tau_{J_T}}^{T}(\lambda^k_t(f_0) - \lambda^k_t(f)) dt}}_{\Exz{W_T}}.
\end{align*}
From Lemma \ref{lem:t_j}, we have that $ \Probz{\sum_{j=1}^{J_T-1} T_j  - \Exz{T_j} \geq  z_T} = o(1)$. We now deal with the second term on the RHS of \eqref{borne-p0}. Using Lemma \ref{lem:t_j}, we have
\begin{align*}
    \Exz{  \sum_{j=1}^{J_T-1} T_j - \Exz{T_j} } 
    &=\Exz{\sum_{j=\lfloor T/\Exz{\Delta \tau_1}\rfloor}^{J_T-1} T_j - \Exz{T_j}}  \\
    &\leq \Exz{\sum_{J \in \mathcal{J}_T} \mathds{1}_{J_T = J}\left( \sum_{j=\lfloor T/\Exz{\Delta \tau_1}\rfloor}^{J-1} |T_j - \Exz{T_j}|\right)} + \sqrt{\Probz{J_T \notin \mathcal{J}_T}} \sqrt{T^2\Exz{T_1^2}}  \\
    &\leq \Exz{ \sum_{j=\lfloor\frac{T}{\Exz{\Delta \tau_1}} (1 - c_\beta \sqrt{\frac{\log T}{T}})\rfloor}^{\lfloor\frac{T}{\Exz{\Delta \tau_1}} (1 + c_\beta \sqrt{\frac{\log T}{T}})\rfloor} |T_j - \Exz{T_j}|} + T^{1-\beta/2} \sqrt{\Exz{T_1^2}} \\
    &\leq  \frac{2 c_\beta}{\Exz{\Delta \tau_1}}\Exz{|T_1 - \Exz{T_j}|} \sqrt{T \log T } + T^{1-\beta/2} \sqrt{\Exz{T_1^2}} \\
    &\lesssim \sqrt{\Exz{T_1^2}} \sqrt{T \log T } \lesssim \sqrt{T} (\log T)^{3/2} \epsilon_T = o(z_T),
\end{align*}
since $\log^3 T = O(z_T)$ by assumption. Consequently,
\begin{align*}
     \Probz{  (J_T - \Exz{J_T}) \Exz{T_1} - \Exz{\sum_{j=0}^{J_T-1} T_j - \Exz{T_j}}\geq z_T } &\leq  \Probz{  J_T - \Exz{J_T} \geq \frac{z_T}{2 \Exz{T_1}} } \\
      &\leq \Probz{  J_T - \frac{T}{\Exz{\Delta \tau_1}} \geq \frac{z_T}{4 \Exz{T_1}}},
\end{align*}
using that $J_T - \Exz{J_T} = J_T - \frac{T}{\Exz{\Delta \tau_1}} + \frac{T}{\Exz{\Delta \tau_1}} - \Exz{J_T}$ and $\frac{T}{\Exz{\Delta \tau_1}} - \Exz{J_T} \leq \frac{z_T}{4 \Exz{T_1}}$ for $T$ large enough. Consequently, since $\Exz{T_1} \leq \sqrt{\frac{z_T}{T}}$, we have with $\eta_T = \sqrt{\frac{ z_T}{4 \Exz{T_1}}}$,
and using the computations as for the proof of Lemma~\ref{lem:main_event},
\begin{align*}
    \Probz{  J_T - \frac{T}{\Exz{\Delta \tau_1}} \geq \eta_T} &\leq \Probz{  \tau_{\lfloor T/\Exz{\Delta \tau_1} + \eta_T\rfloor} \leq T} \\
    &= \Probz{ \sum_{j=1}^{\lfloor T/\Exz{\Delta \tau_1} + \eta_T\rfloor} B_j \leq T - \lfloor T/\Exz{\Delta \tau_1} + \eta_T\rfloor \Exz{\Delta \tau_1}} \\
    &\leq \Probz{ \sum_{j=1}^{\lfloor T/\Exz{\Delta \tau_1} + \eta_T\rfloor} B_j \leq -  \Exz{\Delta \tau_1} \eta_T + \Exz{\Delta \tau_1}} \\
    &\leq \frac{4 \lfloor T/\Exz{\Delta \tau_1} + \eta_T\rfloor \Exz{\Delta \tau_1^2}}{\Exz{\Delta \tau_1}^2 \eta_T^2} \lesssim \frac{T}{\eta_T^2} + \frac{1}{\eta_T} \lesssim \frac{1}{z_T} = o(1).
\end{align*}
For the third term on the RHS of \eqref{borne-p0}, applying Bienayme-Chebyshev's inequality, we have
\begin{align}\label{eq:bienayme}
   &\Probz{W_T -  \Exz{W_T} \geq z_T  } \leq \frac{\Exz{W_T^2}}{z_T^2}.
\end{align}
Using similarly computations as in Lemma \ref{lem:t_j}, we obtain
\begin{align*}\label{control-variance}
     \Exz{W_T^2} &=\Exz{\left(\sum_k \int_{\tau_{J_T}}^{T}\log \left(\frac{\lambda^k_t(f_0)}{\lambda^k_t(f)} \right) dN_t^k - \int_{\tau_{J_T}}^{T}(\lambda^k_t(f_0) - \lambda^k_t(f)) dt \right)^2} \nonumber\\
      &\lesssim  \Exz{ (T - \tau_{J_T}) \int_{\tau_{J_T}}^{T} \left[ \log \left(\frac{\lambda^k_t(f_0)}{\lambda^k_t(f)} \right)  \lambda_t^k(f_0) -  (\lambda^k_t(f_0) - \lambda^k_t(f))\right]^2 dt } +  \Exz{\int_{\tau_{J_T}}^{T}\log^2  \left(\frac{\lambda^k_t(f_0)}{\lambda^k_t(f)} \right) \lambda_t^k(f_0) dt}.
\end{align*}
Then since
\begin{align*}
 &\Exz{ (T - \tau_{J_T}) \int_{\tau_{J_T}}^{T} \left[ \log \left(\frac{\lambda^k_t(f_0)}{\lambda^k_t(f)} \right)  \lambda_t^k(f_0) -  (\lambda^k_t(f_0) - \lambda^k_t(f))\right]^2 dt } \leq  \Exz{ \Delta \tau_1 \int_{\tau_{1}}^{\tau_2} \chi \left(\frac{\lambda^k_t(f_0)}{\lambda^k_t(f)} \right)^2  \lambda_t^k(f_0)^2 dt }, \\
    &\Exz{\int_{\tau_{J_T}}^{T}\log^2  \left(\frac{\lambda^k_t(f_0)}{\lambda^k_t(f)} \right) \lambda_t^k(f_0) dt} \leq \Exz{\int_{\tau_{1}}^{\tau_2}\log^2  \left(\frac{\lambda^k_t(f_0)}{\lambda^k_t(f)} \right) \lambda_t^k(f_0) dt},
\end{align*}
we can use the bounds derived for $\Exz{T_j^2}$ in Lemma \ref{lem:t_j}.

We finally obtain
\begin{align*}
    \Probz{W_T -  \Exz{W_T} \geq z_T  } \lesssim \frac{ (\log^2 T)\e_T^2}{z_T^2} \lesssim \frac{ \log^2 T }{T^2 \e_T^2} = o(1).
\end{align*}
With similar computations, we also obtain that $ \Probz{W_0 -  \Exz{W_0} \geq z_T  } = o(1)$. Consequently, reporting into \eqref{borne-p0} and using Lemma \ref{lem:KLdecomp}, we finally obtain that
\begin{align*}
    \Probz{L_T(f_0) - L_T(f) > 5z_T} \leq \Probz{\mathcal{L}_T > 5z_T - u_T} \leq \Probz{\mathcal{L}_T > 4z_T}  = o(1),
\end{align*}
since $KL(f_0,f) \leq u_T \leq z_T$ using Lemmas \ref{lem:KLdecomp} and \ref{lem:t_j}.

\section{Proof of Theorem \ref{cor:graph_estim} and Proposition \ref{prop:graph_restr}}\label{sec:supp:graph}

\subsection{Proof of Theorem \ref{cor:graph_estim} (Case 2)}\label{sec:proof_cor_graph_suite}
In this section, we prove the second case of the proof of Theorem \ref{cor:graph_estim} in Section \ref{sec:proof_cor_graph} of \cite{main}. We recall that in this case we consider $(l,k) \in I(\delta_0)$, i.e., $\delta_{lk}^0 = 1$. We also recall the notation $S_{lk}^0 = \|h_{lk}^0\|_1$ and $M_T = M \sqrt{\kappa_T}$ with $M > 0$.

We first note that if $S_{lk}^0 > M_1 \sqrt{\kappa_T} \epsilon_T$ with $M_1 >M$ and $ 1 - F(S_{lk}^0/2) \geq 2e^{-\gamma T\epsilon_T^2}$ for some $\gamma > \kappa_T + c_1 =: \kappa_T'$, then if $\delta_{lk} = 0 $, $ f \in A_{L_1}(M_1\sqrt{\kappa_T}\epsilon_T)^c$ and
$$
\Pi(\delta_{lk} = 0 | N) \leq \Pi( A_{L_1}(M_T\epsilon_T)^c | N) , \quad \text{ and } \quad S_{lk}^0 - M_T\epsilon_T \geq S_{lk}^0 /2 .
$$
Therefore, since $F$ is non-increasing, $F(S_{lk}^0 - M_T\epsilon_T) \leq F(S_{lk}^0/2)$ and 
\begin{align*}
  & \Probz{\hat \delta_{lk}^{\Pi, L} = 0} \leq  \Probz{\Pi((1-F(S_{lk}))\mathds{1}_{\delta=1}(\mathds{1}_{S_{lk} \geq S_{lk}^0 - M_T\epsilon_T} + \mathds{1}_{S_{lk} < S_{lk}^0 - M_T\epsilon_T}) | N) \leq \Pi(A_{L_1}(2M_T\epsilon_T)^c|N) }\\
   &\leq \Probz{(1 - F(S_{lk}^0 /2)) \Pi(S_{lk} > S_{lk}^0 - M_T\epsilon_T| N) + \Pi((1-F(S_{lk}))\mathds{1}_{S_{lk} < S_{lk}^0 - M_T\epsilon_T}) | N) \leq \Pi(A_{L_1}(M_1\sqrt{\kappa_T}\epsilon_T)^c|N) }\\
  &\leq  \Probz{2e^{- \gamma T \epsilon_T^2} \Pi( S_{lk} > S_{lk}^0 - M_T\epsilon_T | N) \leq \Pi(A_{L_1}( M_1\sqrt{\kappa_T}\epsilon_T)^c|N) }\\
   &\leq \Probz{\Pi(S_{lk} > S_{lk}^0 - M_T \epsilon_T | N) \leq 1/2 } 
   +  \Probz{\eve \cap \left\{ e^{- \gamma T \epsilon_T^2} \leq \Pi(A_{L_1}( M_1\sqrt{\kappa_T}\epsilon_T)^c|N) \right \}} + \Probz{\eve^c}.
 \end{align*}

 Similar to the first case where $\delta_{lk}^0 = 0$, we have that   $\Probz{\eve \cap \left\{e^{- \kappa_T' \epsilon_T^2} \leq \Pi(A_{L_1}(M_T\epsilon_T)^c|N) \right\}} = o(1)$, and since $\gamma \geq \kappa_T'$, 

 \begin{align*}
     \Probz{\hat \delta_{lk}^{\Pi,L} = 0} &\leq \Probz{\Pi(S_{lk} > S_{lk}^0 - M_T \epsilon_T | N) \leq 1/2} + o(1) = \Probz{\Pi(S_{lk} < S_{lk}^0 - M_T \epsilon_T | N) >1/2 } +o(1)\\
     &\leq \Probz{\eve \cap \{\Pi(A_{L_1}(M_1\sqrt{\kappa_T}\epsilon_T)^c|N) >1/2\}} + \Probz{\eve^c} = o(1),
 \end{align*}
which terminates this proof.

\subsection{Proof of Proposition \ref{prop:graph_restr}}\label{sec:proof_prop_restr}

In this section, we prove our posterior consistency result on the posterior distribution in the restricted models, the \textbf{All equal model} and \textbf{Receiver dependent  model}, defined  in Section \ref{sec:post_consistency} of \cite{main}.

In the \textbf{All equal model}, if $I(\delta_0) \neq \emptyset $ then $ \exists (l_1, k_1) \in [K]^2, \delta_{l_1 k_1}^0 = 1$, and $h_0\neq 0$.  Consequently, for $T$ large enough,
\begin{align*}
\{f \in \mathcal{F}; \: \delta_{l_1 k_1} \neq \delta_{l_1 k_1}^0 \} &= \left\{f \in \mathcal{F};  \:  \delta_{l_1 k_1} = 0\right\} \subset \left\{f \in \mathcal{F};  \:  \|h^0_{l_1 k_1} - h_{l_1 k_1}\|_1 = \norm{h_0}_1\right\} \subset A_{L_1}(M_T \epsilon_T)^c,
\end{align*}
leading to $\Exz{\Pi(\delta_{l_1 k_1} \neq \delta_{l_1 k_1}^0|N)} = o(1)$ using Theorem \ref{thm:conc_g}. This would hold for the same resasons for any $(l,k) \in I(\delta_0)$. For $(l,k) \notin I(\delta_0)$, we have instead that for $T$ large enough,
\begin{align*}
\{f \in \mathcal{F};  \:  \delta_{l k} \neq \delta_{l k}^0 \} &= \{f \in \mathcal{F}; \: \delta_{l k} = 1\} \subset \{f \in \mathcal{F}; \: \norm{h^0_{l k} - h_{l k}}_1 = \norm{h}_1\} \\
 &\subset \{f \in \mathcal{F}; \: \norm{h}_1 + \norm{h^0_{l_1 k_1} - h_{l_1 k_1}}_1 \geq \norm{h_0}_1\}  \subset A_{L_1}(M_T \epsilon_T)^c,
\end{align*}
as soon as $\norm{h_0}_1 \geq 3M_T \epsilon_T$, since $\norm{h}_1 + \norm{h^0_{l_1 k_1} - h_{l_1 k_1}}_1 \geq \norm{h}_1 + \norm{h_0}_1 \wedge \norm{h - h^0}_1  \geq (\norm{h}_1 + \norm{h_0}_1) \wedge (\norm{h}_1 + \norm{h-h_0}_1) \geq \norm{h_0}$. Similarly to the proof of Theorem \ref{thm:post_graph} in Section \ref{sec:proof_post_graph}, we then obtain $\Exz{\Pi(\delta \neq \delta_0 | N)} = o(1)$.

If  $I(\delta_0) = \emptyset $, then $\forall (l,k) \in [K]^2, \delta_{l k}^0 = 0$,  and $h_0 =0$, and in this case we first show that there exists $\cste>0 $ such that
\begin{equation} \label{eq:D_T_Poisson}
\Probz{\left \{D_T < \cste T^{-K/2}  \right \} \cap \eve} = o(1).
\end{equation}
Since $h_0=0$, the log-likelihood function is the one off $K$ independent homogeneous Poisson PP with parameter $r_0$, i.e.,
\begin{align*}
L_T(f_0) = L_T(r_0) = \sum_k \log(r_k^0)N^k[0,T) - r_k^0T,
\end{align*}
with $r_k^0 = \phi_k(\nu_k^0)$. Let $\bar{A} = \{f \in \mathcal{F}_T; \: h = 0\}$. For any $f \in \bar{A}$, we also have $L_T(f) = L_T(r_f) = \sum_k \log(r_k^f)N^k[0,T) - r_k^fT$ and the model is also a simple Poisson PP,  which is a regular model, and which parameter is $\phi(\nu)$. 
Therefore, we have
\begin{align*}
L_T(r) - L_T(r_0) &= \sum_k \log(\frac{r_k}{r_k^0})N^k[0,T) - (r_k^f- r_k^0)T \\
&= \sum_k \left[\frac{r_k^f- r_k^0}{r_k^0} - \frac{1}{2}\left(\frac{r_k^f- r_k^0}{r_k^0} \right)^2 + O_{\mathbb{P}_0}(r_k^f- r_k^0)^3\right]N^k[0,T) - (r_k^f- r_k^0)T \\
&= \sum_k \left(\frac{N^k[0,T)}{r_k^0} - T\right)(r_k^f- r_k^0) - \frac{N^k[0,T)}{2}\left(\frac{r_k^f- r_k^0}{r_k^0} \right)^2 + O_{\mathbb{P}_0}(T(r_k^f- r_k^0)^3).
\end{align*}
\ds{Also, let $\tilde \pi_r$ be the prior density of $r_k^f= \phi_k(\nu_k)$ given by $\tilde\pi_r(x) = \phi(\nu) \pi_\nu(\nu).$}
\ds{Note that in the case of partially known link functions of the form $\phi_k(x) = \theta_k + \psi(x)$, the parameter of the Poisson PP is now $(\nu, \theta)$ and we can consider a marginal prior density of $r_k^f= \theta_k + \psi(\nu_k)$ given by 
$$ \tilde\pi_r(x) = \int_0^{\psi^{-1}(x)} \pi_\theta(x-\psi(\nu)) \pi_\nu(\nu)d\nu.$$ The regularity assumptions on $\pi_\nu$ (and $\pi_\theta$) and $\phi^{-1}$ imply that $\tilde \pi_r$ is continuous and positive at $r_k^0$ for all~$k$.}

Defining $\bar{A}_T = \bar{A}\cap\{\|r_f -r_0\|_1 \leq \epsilon\}$ for $\epsilon>0$ small enough, we thus have
\begin{align*}
D_T &= \int_{\mathcal{F}_T} e^{L_T(f) - L_T(f_0)} d\Pi(f) \geq \int_{\bar{A}_T} e^{L_T(r) - L_T(r_0)} d\Pi(f) \\
&\geq  \int_{\bar{A}_T} \prod_{k=1}^K \exp \left \{ \left(\frac{N^k[0,T)}{r_k^0} - T\right)(r_k^f- r_k^0) - \frac{N^k[0,T)}{2}\left(\frac{r_k^f- r_k^0}{r_k^0} \right)^2(1 +\epsilon) \right\}\tilde \pi(r_k)d r_k\\
&=  \prod_{k=1}^K \tilde \pi_r(r_k^0)(1+o_{\mathbb{P}_0}(1))e^{\frac{r_k^0}{2(1+ \epsilon)N^k[0,T)} \left(\frac{N^k[0,T)}{ r_k^0} - T\right)^2} \times \\
& \quad \int_{|r_k^f-r_k^0|\leq \epsilon/K}\exp \left \{- \frac{N^k[0,T)}{2 (r_k^0)^2}(1 - \epsilon) \left(r_k^f- r_k^0  - \frac{(r_k^0)^2}{(1 + \epsilon)N^k[0,T)} \left(\frac{N^k[0,T)}{ r_k^0} - T\right)\right)^2  \right\}d r_k^f\\
&\geq  \prod_{k=1}^K \tilde \pi_r(r_k^0)r_k^0 \frac{ \sqrt{2\pi} }{[N^k[0,T)(1 + \epsilon)]^{1/2}}  (1 + o_{\mathbb{P}_0}(1))  \geq \prod_{k=1}^K \frac{ \sqrt{2\pi } \tilde \pi_r(r_k^0)r_k^0  }{[T(1 + \epsilon)]^{1/2}}  (1 + o_{\mathbb{P}_0}(1)) ,
\end{align*}
since $N^k[0,T)$ is a Poisson random variable with parameter $r_k^0 T$ so that $|N^k[0,T)/T - r_k^0|\leq M_T/\sqrt{T}  $ with probability going to $1$ and $\{|r_k^f -r_k^0|\leq \epsilon/K\}$ contains the set 
$$ \left| r_k^f- r_k^0  - \frac{(r_k^0)^2}{(1 - \epsilon)N^k[0,T)} \left(\frac{N^k[0,T)}{ r_k^0} - T\right)\right| \leq \frac{ \epsilon }{ 2 K}, $$
for $T$ large enough.
Therefore we obtain \eqref{eq:D_T_Poisson} and deduce that $\epsilon_T \lesssim \sqrt{\log T/T}$ using the same arguments as in the proofs of Theorems  \ref{thm:d1t} and \ref{thm:conc_g}. As in Theorem \ref{thm:conc_g}, it is thus sufficient that
\begin{align*}
&\Pi(\{0 < \norm{h}_1 \leq M\sqrt{\log T/T} \} \cap \{ \max_k|r_k^f- r_k^0| \leq M\sqrt{\log T/T}\}) \\
&\leq \Pi(\{0 < \norm{h}_1 \leq M\sqrt{\log T/T} \} \cap \{ \max_k|\nu_k - \nu_k^0| \leq \frac{M}{L} \sqrt{\log T/T}\}) = o(T^{-K/2}),
\end{align*}
for $M$ large enough which boils down to assuming that 
\begin{align*}
\Pi(\{0 < \norm{h}_1 \leq M\sqrt{\log T/T} \} ) = o((\log T)^{-K/2}),
\end{align*}
to conclude that $\Exz{\Pi \left(\delta \neq \delta_0 |N \right)} = o(1)$.

In the \textbf{Receiver node dependent model}, i.e., $\forall (l,k) \in [K]^2, h_{lk} = \delta_{lk}h_{k}$, we obtain the  result similarly to  the \textbf{All equal model} since the likelihood is also a product of likelihoods per node: 
$$L_T(f)  = \sum_{k=1}^KL_T(\nu_k, h_k, \delta{(k)}, \theta_k), \quad \text{ with } \delta{(k)} := (\delta_{lk}, 1\leq l\leq K).$$ 
If the priors on $(\theta_k, \nu_k, h_k , \delta(k))$ are independent, the posteriors are also independent and we can directly apply the previous result.


\section{Proof of Corollary \ref{cor:post_mean}}\label{sec:proof_cor_post_mean}

In this section we prove our result on the  convergence rate of the posterior mean estimator. \ds{ In all the considered models with known link functions, the convergence of the posterior mean $\hat f = (\hat{\nu}, \hat{h})$ results from the same arguments as in Corollary 1 of \cite{donnet18:supplement} (proof in Section 2.3 in the supplementary material). In the case of the shifted ReLU model with unknown shift, we can also use similar computations for $\hat f = (\hat{\nu}, \hat{h})$ and $\hat{\theta}$.} We first recall some notation from the proofs of Theorems~ \ref{thm:conc_g} and Proposition \ref{prop:relu}: $\bar A(\Tilde{M}_T \e_T) = \{\theta\in \Theta, \norm{\theta - \theta_0}_1 < \Tilde{M}_T \e_T \}  $,  $ A_{L_1}(M_T \e_T) = \{ (f, \theta) \in \mathcal{F} \times \Theta, \norm{\theta + \nu - \theta_0 - \nu_0}_1 + \norm{h - h_0}_1 < M_T \e_T \}  $ and $\Tilde{M}_T = \Tilde{M} \sqrt{\kappa_T}, \: M_T = M \sqrt{\kappa_T}$, $\Tilde{M} > M > 0$. We note that 
\begin{align*}
&\norm{\hat{\theta} - \theta_0}_1 \leq \Tilde{M}_T \e_T + \mathbb{E}^\Pi[\norm{\theta - \theta_0}_1 \mathds{1}_{\bar A(\Tilde{M}_T \e_T)^c} |N].
\end{align*}
Then, splitting $\bar A (\Tilde{M}_T \e_T)^c \times \mathcal{F}_T$  into $\bar A( \Tilde{M}_T \e_T)^c \times  \mathcal{F}_T \cap A_{L_1}(M_T \epsilon_T)$ and $\bar A(\Tilde{M}_T \e_T)^c \times  \mathcal{F}_T \cap A_{L_1}(M_T \epsilon_T)^c$, we control $\mathbb{E}^\Pi[\norm{\theta - \theta_0}_1 \mathds{1}_{B_T} |N]$ using the following arguments with $B_T$ representing either 
$\bar A(\Tilde{M}_T \e_T)^c  \times \mathcal{F}_T \cap A_{L_1}( M_T \epsilon_T)$ or $ A_{L_1}(M_T \epsilon_T)^c$. Using the decomposition \eqref{decomposition}, with $\kappa_T' = \kappa_T + c_1$, we have
\begin{align*}
\Probz{\mathbb{E}^\Pi[\norm{\theta - \theta_0}_1  \mathds{1}_{B_T}|N] >  \e_T} &\leq \Exz{\phi \mathds{1}_{\eve}} + \Probz{\{D_T < e^{-\kappa_T' T\epsilon_T^2}\} \cap \eve} + \Probz{\eve^c} + \frac{e^{\kappa_T' T\epsilon_T^2}}{\e_T}  \Pi( \mathcal F_T^c)\\
&+ \frac{e^{\kappa_T' T\epsilon_T^2}}{\e_T} \int_{\mathcal{F}_T \cap B_T}  \norm{\theta - \theta_0}_1 \Exz{\Exf{(1 - \phi)\mathds{1}_{\eve}} \Big|\mathcal{G}_0} d\Pi(f) \\
&\leq o(1) + o\left( \int  \norm{\theta - \theta_0}_1 d\Pi(f) \right)  = o(1),
\end{align*}
using the tests defined In Lemma \ref{lem:test_theta} if $B_T =\bar A(\Tilde{M}_T \e_T)^c   \times \mathcal{F}_T  \cap A_{L_1}(M_T \epsilon_T) $ or the tests defined in Lemma \ref{lem:tests} if $B_T = A_{L_1}(M_T \epsilon_T)^c$, and also that $\log T = o(T \epsilon_T^2)$ to obtain that $\frac{e^{\kappa_T' T\epsilon_T^2}}{\e_T}  \Pi( \mathcal F_T^c) \leq \Pi( \mathcal H_T^c) e^{\kappa_T' T\epsilon_T^2 - \log \e_T} = o(1)$, whichs terminates this proof.


\section{Proofs of some results on prior distributions}\label{sec:prior_dist}

In this section, we present an alternative construction of the prior distribution using mixtures of Beta distributions and the proof of Lemma \ref{lem:histo}, which gives one example of model where the condition \eqref{condA3} can be verified.

\subsection{Mixtures of Betas priors}\label{sec:mixture}


This family of prior distributions can be also considered alongside the ones presented in Section \ref{sec:priors} of \cite{main}. The following construction is similar to Section 2.3.2 of \cite{donnet18:supplement}, which is based on \cite{rousseau:09}. Using the hierarchical structure \eqref{parametrisation1} from Section \ref{sec:priors} ,  we define $\pi_h$ as follows. Let
$$   \tilde h_{\alpha, M}(x) = \int_u g_{\alpha,u}(x) dM(u), \quad g_{\alpha, u} ( x) = \frac{ \Gamma(\alpha/u(1-u)) }{ \Gamma(\alpha/u) \Gamma(\alpha/(1-u))} x^{-\alpha/(1-u)-1}(1-x)^{-\alpha/u-1},$$
and $\pi_{\tilde h}$ be the push forward distribution of $\Pi_\alpha \times \Pi_M$ by the transformation $(\alpha, M) \rightarrow h_{\alpha,M}$, where $\Pi_\alpha$ and $\Pi_M$   are respectively the probability distribution on $\alpha$ and $M$. Therefore $\pi_{\tilde h}$ is a bounded signed measure on $[0,1]$.  As in \cite{donnet18:supplement}, we choose $\sqrt{\alpha} $ to follow a Gamma distribution and  define $\Pi_M$ by 
$$ M(u) = \sum_{j=1}^J r_j p_j \delta_{u_j}(u), \quad u_j \stackrel{i.i.d.}{\sim} G_0, \quad J \sim \mathcal P(\lambda),$$
where $r_j$ are independent Rademacher random variables and $(p_1,\cdots, p_J) \sim \mathcal D(a_1, \cdots, a_J)$ with $\sum_{j=1}^J a_j \leq C$ for some fixed $C>0$. Note that since 
$\|h_{\alpha, M}\|_1 \leq 1$, we can define 
$$ h_{lk} = \tilde{S}_{lk} \tilde h_{lk}, \quad \|\tilde S^+\| \leq 1, \quad  \tilde h_{lk} \stackrel{i.i.d.}{\sim} \pi_{\tilde h},$$
so that the  prior distribution on $h$ is the push forward distribution of $\pi_{\tilde h}^{\otimes |I(\delta)|} \times \pi_S(\cdot|\delta)$ by the above transformation, with $\pi_S$ defined in \textbf{(S2)} in Section \ref{sec:priors} of \cite{main}.
Since $\tilde S$ is a (component-wise) upper bound on the matrix $S$,  $\|\tilde S^+\| \leq 1$ implies $\| S^+\| \leq 1$. We then arrive at the following result. 

\begin{corollary}{}\label{cor:mixture}
\ds{Let $N$ be a Hawkes process with link functions $\phi = (\phi_k)_k$ and parameter $f_0 = (\nu_0, h_0)$ such that $(\phi, f_0)$ verify the conditions of Lemma~\ref{lem:existence}, and Assumption \ref{ass-psi}.} Under the above spline  prior, if the prior on $S$ satisfies the conditions defined in \textbf{(S1)} (Section \ref{sec:priors} of \cite{main}), and also
if  $\forall (l,k) \in [K]^2, \: h_{lk}^0 \in \mathcal{H}(\beta,L)$ with $\beta >0$ and $\| S_0^+\| < 1$ then for $M$ large enough,
\begin{align*}
    \mathbb{E}_0 \left[\Pi\big(\|f - f_0 \|_1 > M  T^{-\beta/(2\beta+1)} \sqrt{\log \log T}(\log T)^q \big| N \big) \right]  = o(1),
\end{align*}
where $q= 5\beta/(4\beta+2)$ if $\phi$ verifies Assumption \ref{ass-psi}(i) , and $q=1/2+5\beta/(4\beta+2)$  if $\phi$ verifies Assumption \ref{ass-psi}(ii) .
\end{corollary}
 
%
%

\subsection{Proof of Lemma \ref{lem:histo}}\label{sec:proof_lem_histo}

\begin{lemma}[Lemma \ref{lem:histo}]
\ds{Let $N$ be a Hawkes process with ReLU link functions $\phi_k(x) = (x)_+, \forall k \in [K]$, and parameter $f_0 = (\nu_0, h_0)$ such that $(\phi, f_0)$ verify condition \textbf{(C1bis)} and  for all $l$, there exists $J_0 \in \mathbb{N}^*$ such that}
$$
    h_{lk}^{0}(t) = \sum_{j=1}^{J_0} \omega^{lk}_{j0} \mathds{1}_{I_j}(t), \quad \omega_{j0}^{lk} \in \mathbb Q, \quad \forall j \in [J_0],
$$
with $\{I_j\}_{j=1}^{J_0}$ a partition of $[0,A]$. Then, condition \eqref{condA3} of Proposition ~\ref{prop:relu} holds, i.e.,
\begin{align*}
    \lim \sup_{T \to \infty}\frac{1}{T} \mathbb{E}_0\left(\int_0^T \frac{\mathds{1}_{\lambda^k_t(f_0) > 0}}{\lambda^k_t(f_0)}dt \right) < + \infty, \quad   k \in [K].
\end{align*}
\end{lemma}
\begin{proof}
Let $f_0$ verifying the conditions of the lemma. We first show that there exists $c_0 > 0$ that depends only on the parameters $\{\nu_k^0, \{\omega_{j0}^{kl}\}_{j=1}^J\}_{k,l=1}^K$ such that $\forall k \in [K], \forall t \geq 0$, $\lambda^k_t(f_0) > 0 \implies \lambda^k_t(f_0) \geq c_0$.
We prove here the result for the unidimensional Hawkes model with $K =1$, but  our proof can be easily generalized to $K>1$. We therefore use the notation $\nu_0$ and   $w_{j0}$ for $\nu_1^0$ and $w_{j0}^{11}$. 

Since $w_{j0} \in \mathbb{Q}$, let $p_j, q_j \in \Z$ such that  $w_{j0} = p_j/q_j$ and let 
$q \in \Z$ be the least common multiple of $(p_j, q_j)$. Thus there exists $a_j \in \mathbb{Z}$ such that
$    \omega_{j0} = a_j/q.
$
and for any $t \geq 0$, let $n_j(t) = \int_{t-A}^t \mathds{1}_{I_j}(t-s) dN_s$ be the number of events that "activate" the bin $j$ at $t$. With this notation, we can then write
\begin{align*}
    \lambda_t(f_0) &= \left(\nu_0 + \sum_{j=1}^{J_0} n_j(t)\frac{a_j}{q}\right)_+ =  \left(\nu_0  + \sum_{j=1}^{J_0} n_j(t)\frac{a_j}{q} \right)_+  \\
    &= \left(\frac{1}{q}\left[\nu_0  q + \sum_{j=1}^{J_0} n_j(t)a_j\right] \right)_+.
\end{align*}
Let $\varepsilon > 0$ such that
    $ \varepsilon  = \min_{u \in \mathbb{Z},  \nu_0  q + u > 0} \nu_0 q + u
$. Then $\varepsilon >0$ and for any $t \geq 0$ such that $\Tilde \lambda_t(f_0)>0$, since $\sum_{j=1}^{J_0} n_j(t)a_j \in \mathbb{Z}$, then $\nu_0  q + \sum_{j=1}^{J_0} n_j(t)a_j \geq \varepsilon > 0$ and 
$  \lambda_t(f_0) \geq \varepsilon/q  =: c_0 > 0$, which proves that (i) holds.
Therefore, in this model, we have
\begin{align*}
    \frac{1}{T} \mathbb{E}_0\left(\int_0^T \frac{\mathds{1}_{\lambda_t(f_0) > 0}}{\lambda_t(f_0)}dt \right) &\leq \frac{1}{T} \mathbb{E}_0\left(\int_0^T \frac{\mathds{1}_{\lambda_t(f_0) > 0}}{c_0}dt \right) \leq \frac{1}{T} \mathbb{E}_0\left(\int_0^T \frac{1}{c_0}dt \right) = \frac{1}{c_0} < +\infty,
\end{align*}
which proves that  \eqref{condA3} is satisfied.
\end{proof}

\begin{remark}
\ds{We could similarly show that  if also $\forall l \in [K], \forall j \in [J]$, $\nu_k^0 \in \mathbb{R} \backslash \mathbb{Q}$, then there exists $d_0 < 0$ depending on $\{\nu_k^0, \{\omega_{j0}^{kl}\}_{j=1}^J\}_{k,l=1}^K$ such that $\forall k \in [K], \forall t \geq 0$,   $\lambda^k_t(f_0) = 0 \implies \Tilde{\lambda}^k_t(\nu_0, h_0) \leq d_0$.}
\end{remark}


\section{ Lemmas on tests }\label{sec:technical_lemmas}

\ds{In this section we prove two technical lemmas on the test functions used in  the proofs of Theorem \ref{thm:d1t} and Proposition \ref{prop:relu}. 
In Section \ref{sec:lem:tests}, we state and prove our first lemma, Lemma \ref{lem:tests}, which relates to the elementary test functions used in the proof of Theorem \ref{thm:d1t} (Section \ref{sec:proof_thm_d1t}) and in Section \ref{app:proof_test}, we prove Lemma \ref{lem:test_theta}, which provides the bound on the error of the tests used in the proof of Case 2 of Proposition \ref{prop:relu}.}

\subsection{Lemma \ref{lem:tests}: test used in the proof of Theorem \ref{thm:d1t}} \label{sec:lem:tests}
\begin{lemma}\label{lem:tests}
For $i \geq 1$, let $\mathcal{F}_i = \{ f \in \mathcal F_T; \nu_l \leq \nu_l^0+ 2K\|r_0\|_1\mathbb E_0(\Delta \tau_1) i \epsilon_T, \: \forall l \in [K]\}$ and $f_1 \in \mathcal{F}_i$. We define the test
\begin{equation}
    \phi_{f_1,i} = \max \limits_{l \in [K]} \mathds{1}_{\{N^l(A_{1l}) -  \Lambda^l(A_{1l},f_0)  \geq iT\epsilon_T/8\}} \wedge \mathds{1}_{\{N^l(A_{1l}^c) -  \Lambda^l(A_{1l}^c,f_0)  \geq iT\epsilon_T/8\}}, \nonumber
\end{equation}
where for all $l \in [K]$, $A_{1l} = \{ t \in [0,T]; \: \lambda^l_t(f_1) \geq \lambda^l_t(f_0) \}$, $\Lambda^l(A_{1l},f_0) = \int_0^T \mathds{1}_{A_{1l}}(t) \lambda^l_t(f_0)dt$ and $\Lambda^l(A_{1l}^c,f_0) = \int_0^T \mathds{1}_{A_{1l}^c}(t) \lambda^l_t(f_0)dt$. 
Then
\begin{equation}
    \mathbb{E}_0[\mathds{1}_{\Tilde{\Omega}_{T}}\phi_{f_1,i}] + \sup_{\|f-f_1\|_1 \leq i\epsilon_T/(12N_0)} \mathbb{E}_0 \left[\mathbb{E}_f[\mathds{1}_{\Tilde{\Omega}_{T}} \mathds{1}_{f \in S_i} (1 - \phi_{f_1,i})|\mathcal{G}_0] \right] \leq (2K+1) \max \limits_{l \in [K]}  e^{-x_{1l} Ti\epsilon_T(\sqrt{\mu_l^{0}} \land i \epsilon_T)}, \nonumber
\end{equation}
where for $l \in [K]$, $x_{1l} > 0$ is an absolute constant, $\mu_l^0 = \Exz{\lambda_t^l(f_0)}$, $N_0= 1 + \sum_{l=1}^K \mu_{l}^0$ and $S_i$ is defined in \eqref{def:slices} in Section \ref{sec:proof_thm_d1t}.
\end{lemma}

\begin{proof} 

For $l \in [K]$, let
$$\phi_{il} = \phi_{il}(f_1) = \mathds{1}_{\{N^l(A_{1l}) -  \Lambda^l(A_{1l},f_0)  \geq iT\epsilon_T/8\}}.$$

Mimicking the proof of Lemma 1 of \cite{donnet18:supplement}, we obtain that \begin{equation}\label{Err11}
    \Exz{\phi_{il} \mathds{1}_{\eve}} \leq e^{-x_1 i T \epsilon_T \min(\sqrt{\mu_l^{0} }, i\epsilon_T)}.
\end{equation}

We first consider the  event $\{\Lambda^l(A_{1l},f_1) -  \Lambda^l(A_{1l},f_0) \geq  \Lambda^l(A_{1l}^c,f_1) -  \Lambda^l( A_{1l}^c,f_0)\}$. Let $f \in \mathcal{F}_i$ such that $\|f-f_1\|_1 \leq \zeta i \epsilon_T$ with $\zeta = 1/(6N_0)$ and $N_0 = 1 + \sum_{l} \mu_l^{0}$. On $\Tilde{\Omega}_{T}$, using that $\phi_l$ is $L$-Lipschitz for any $l$, we have
\begin{align*}
    T \Tilde{d}_{1T}(f,f_1) &= \sum_{l=1}^K \int_0^T \mathds{1}_{A_2(T)}(t) |\lambda^l_t(f) - \lambda^l_t(f_1)| dt \leq \sum_{l=1}^K \int_0^T |\lambda^l_t(f) - \lambda^l_t(f_1)| dt\\
    &\leq   L \sum_{l=1}^K \int_0^T |\Tilde \lambda^l_t(\nu, h) - \Tilde \lambda^l_t(\nu_1, h_1)| dt\\
    &\leq T L \sum_{l=1}^K  |\nu_l - \nu_l^1| + L \sum_{l=1}^K\sum_{k=1}^K  \int_0^T \int_{t-A}^t |(h_{kl} - h_{kl}^1)(t-s)| N^k(ds)  \\
    &\leq T  L\norm{\nu - \nu_1}_1 + \max_l N^l[-A,T] L \sum_{l=1}^K\sum_{k=1}^K \|h_{kl} - h_{kl}^1\|_1 \\
    &\leq L N_0 T \norm{f - f_1}_1 \leq L N_0 T \zeta i \e_T.
\end{align*}
Moreover, since $f \in S_i$, on $\eve$, we also have that
\begin{align*}
    \int_0^T \mathds{1}_{A_2(T)}\lambda^l_t(f) dt \leq \int_0^T  \mathds{1}_{A_2(T)} \lambda^l_t(f_0) dt + K T(i+1) \epsilon_T \leq 2 T \mu_l^{0} + K T(i+1) \epsilon_T =: \Tilde{v}.
\end{align*}
Applying again inequality (7.7) of \cite{Hansen:Reynaud:Rivoirard} with $v = \Tilde{v}$ and using the computations of \cite{donnet18:supplement}, we arrive at
$$\Exf{\mathds{1}_{\eve} \mathds{1}_{f \in S_i} (1 - \phi_{il})\Big|\mathcal{G}_0} \leq 2K e^{-x_{1l} i T \epsilon_T \min(\sqrt{\mu_l^{0}}, i\epsilon_T)},$$
for some $x_{1l} > 0$. We can obtain similar results for
$$\phi_{il}' = \mathds{1}_{\{N^l(A_{1l}^c) -  \Lambda^l(\bar A_{1l}^c,f_0)  \geq iT\epsilon_T/8\}}.$$

Finally, with $\phi_{f_1, i} = \max \limits_{l \in [K]} \phi_{il} \wedge \phi_{il}'$, we arrive at the final results of this lemma:
\begin{align*}
&\Exz{\phi_{f_1, i} \mathds{1}_{\eve}} \leq \max_l e^{-x_{1l} i T \epsilon_T \min(\sqrt{\mu_l^{0} }, i\epsilon_T)} \leq e^{- (\min \limits_l x_{1l}) i T \epsilon_T \min(\sqrt{\mu_l^{0} }, i\epsilon_T)}  \\
&\mathbb{E}_f[\mathds{1}_{\eve} \mathds{1}_{f \in S_i} (1 - \phi_{f_1, i})|\mathcal{G}_0] \leq \min_l \mathbb{E}_f[\mathds{1}_{\eve} \mathds{1}_{f \in S_i} (1 - \phi_{i l})|\mathcal{G}_0] \leq 2K e^{- (\min \limits_l x_{1l}) i T \epsilon_T \min(\sqrt{\mu_l^{0}}, i\epsilon_T)}.
\end{align*}

\end{proof}

\subsection{Proof of Lemma \ref{lem:test_theta}}\label{app:proof_test}

\ds{In Lemma \ref{lem:test_theta}, we establish the bound on the type I and type II errors of the tests to estimate the parameter $\theta$ in the shifted ReLU link function considered in Case 2 of Proposition \ref{prop:relu}. }




We recall that $\Theta = \R_+ \backslash \{0\}^K$ and $\bar A(R) = \{\theta \in \Theta ; \: \norm{\theta - \theta_0}_1 \leq R\}$. Let $\zeta > 0$ and
$$(f_1, \theta_1) = (\nu_1, h_1, \theta_1) = ((\nu^1_k)_k, (h_{lk}^1)_{l,k}, (\theta^1_k)_k) \in (\bar A(\Tilde{M}_T \epsilon_T)^c \times \mathcal{F} )\cap A_{L_1}(M_T \epsilon_T),$$
with  $\tilde{M}_T = \Tilde{M} \sqrt{\kappa_T}$ , $M_T =M \sqrt{\kappa_T}$ and  $\tilde{M} \geq M$. Let $(f, \theta) \in (\bar A(\Tilde{M}_T \epsilon_T)^c \times \mathcal{F}) \cap A_{L_1}(M_T \epsilon_T)$  such that $\norm{f - f_1}_1 \leq \zeta \epsilon_T$, i.e,
\begin{align*}
& \sum_{k}|\nu_k - \nu_k^1| + |\theta_k - \theta_k^1| + \sum_{l,k} \norm{h_{l k}-h_{l k}^1}_1 \leq \zeta \epsilon_T.
\end{align*}
Since $\theta \in  \bar A(\Tilde{M}_T \epsilon_T)^c$, there exists $k \in [K]$ such that
$
  |\theta_k^0 - \theta_k| \geq \Tilde{M}_T \epsilon_T / K.
$
For this $k$, from assumption \eqref{ass:identif_theta}, there exists $ l \in [K]$ and $x_1,x_2, c_\star > 0$ such that $\forall x \in [x_1,x_2]$, $h_{lk}^0(x) \leq - c_\star<0$.

We first consider the case $\theta_k < \theta_k^0 - \Tilde{M}_T\epsilon_T/K$ and
%
%
recall the notation of Section \ref{sec:proof_conc_f}:  $\delta'=(x_2-x_1)/3$, $n_1 = \lfloor 2 \nu_k^1/(\kappa_1 c_\star ) \rfloor +1$ 
for some $\kappa_1 \in (0,1)$ and the subset of excursions
\begin{align*} 
\mathcal{E} = \{j \in [J_T]; \: N[\tau_{j}, \tau_{j} + \delta') = N^l[\tau_{j}, \tau_{j} + \delta') = n_1, \: N[\tau_{j} + \delta', \: \tau_{j+1}) = 0 \}.
\end{align*}
We recall that
\begin{align*}
I_k^0(f_1, \theta_1) = \left \{t \in [0,T]; \: \lambda_t^k(f_1, \theta_1) = \theta_k^1, \, \lambda^k_t(f_0, \theta_0) = \theta_k^0 \right  \},
\end{align*}
and we first state a preliminary lemma on $I_k^0(f_1, \theta_1) $, which is proved at the end of this proof.

\begin{lemma}\label{lem:I0_model1} 
\ds{In the Hawkes model with shifted ReLU link function , for any $f_0 \in \mathcal{F}$ such that \eqref{ass:identif_theta} is satisfied and  any $(f_1, \theta_1) \in (\bar A(\Tilde{M}_T \epsilon_T)^c \times \Theta) \cap A_{L_1}(M_T\epsilon_T)$, on $\eve'$, }
it holds that
$$
|I_k^0(f_1, \theta_1)| \geq \frac{x_2-x_1}{2} \sum_{j \in [J_T]} \mathds{1}_{j \in \mathcal{E}}, 
$$
with $\mathcal{E} $ defined in   \eqref{E:model1}. 
\end{lemma}
Let 
\begin{align*}
    \phi_k(f_1, \theta_1) := \mathds{1}_{N^k(I_k^0(f_1, \theta_1)) - \Lambda_k(I_k^0(f_1, \theta_1), f_0) < - v_T} \vee \mathds{1}_{|\mathcal{E}| < \frac{p_0 T}{2 \Exz{\Delta \tau_1}}},
\end{align*}
with $ \Lambda_k(I_k^0(f_1, \theta_1), f_0) = \int_0^T \mathds{1}_{I_k^0(f_1, \theta_1)}(t) \lambda_t^k(f_0)dt$, $p_0 = \Probz{j \in \mathcal{E}}$, $v_T = w_T T \epsilon_T > 0$ with  $w_T > 0$ chosen later. We have by definition 
\begin{align}\label{eq:type1_error_m2_c1}
\Exz{\phi_k(f_1, \theta_1)\mathds{1}_{\Tilde{\Omega}_T'} } \leq \Probz{\left \{ |\mathcal{E}| < \frac{p_0  T}{2\Exz{\Delta \tau_1}} \right \} \cap \Tilde{\Omega}_T'} + \Probz{ \left \{ N^k(I_k^0(f_1, \theta_1)) - \Lambda_k(I_k^0(f_1, \theta_1), f_0) < - v_T \right \} \cap \Tilde{\Omega}_T'}.
\end{align}
For the first term on the RHS of \eqref{eq:type1_error_m2_c1}, we apply Hoeffding's inequality with $X_j = \mathds{1}_{j \in \mathcal{E}} \overset{i.i.d.}{\sim} \mathcal{B}(p_0)$:
\begin{align*}
\Probz{\left \{ |\mathcal{E}| < \frac{p_0 T}{2\Exz{\Delta \tau_1}} \right \} \cap \Tilde{\Omega}_T'} &\leq \Probz{\left \{\sum_{j=1}^{J_T} X_j < \frac{p_0 T}{2\Exz{\Delta \tau_1}} \right \} \cap \eve'} \\
 &\leq \Probz{\sum_{j=1}^{T/(2\Exz{\Delta \tau_1})} X_j < \frac{p_0 T}{2\Exz{\Delta \tau_1}}}
\lesssim e^{-\frac{ T p_0^2}{8\Exz{\Delta \tau_1}} } = o(e^{-u_0 T \epsilon_T^2)}),
\end{align*}
for $u_0 < p_0^2/ (8 \Exz{\Delta \tau_1})$ and using that  on $\Tilde{\Omega}_T'$, $J_T > T/(2\Exz{\Delta \tau_1})$.
For the second term of the RHS of \eqref{eq:type1_error_m2_c1}, we apply inequality (7.7) in \cite{Hansen:Reynaud:Rivoirard}, with $H_t = \mathds{1}_{I_k^0(f_1, \theta_1)}(t)$, $H_t^2 \circ \Lambda_t^k(f_0) = \int_0^T \mathds{1}_{I_k^0(f_1, \theta_1)}(t) \theta_k^0 dt = \theta_k^0 |I_k^0(f_1, \theta_1)| \leq \theta_k^0 T$, $x=x_3T\epsilon_T^2$, $x_3 > 0$.
If $\sqrt{2\theta_k^0 T x} + x/3 \leq w_T T\epsilon_T$ and $x_3 > u_0$, then by (7.7) of \cite{Hansen:Reynaud:Rivoirard}, 
\begin{align*}
 \Probz{ \left \{ N^k(I_k^0(f_1, \theta_1)) - \Lambda_k(I_k^0(f_1, \theta_1), f_0) < - v_T \right \} \cap \Tilde{\Omega}_T'} \leq e^{-x_3 T \epsilon_T^2}  = o(e^{- u_0 T \epsilon_T^2}).
\end{align*}
Reporting into \eqref{eq:type1_error_m2_c1}, we obtain that
$
\Exz{\phi_k(f_1)\mathds{1}_{\Tilde{\Omega}_T'} }  = o(e^{- u_0 T \epsilon_T^2})
$, which proves the first part of Lemma \ref{lem:test_theta}. To prove the second part of Lemma \ref{lem:test_theta}, we first note that
\begin{align}\label{eq:type2_error_m2_c1}
\Exf{(1 - \phi_k(f_1, \theta_1))\mathds{1}_{\Tilde{\Omega}_T'} }  = \Probf{\left \{N^k(I_k^0(f_1, \theta_1)) - \Lambda_k(I_k^0(f_1, \theta_1), f_0) \geq - v_T \right \} \cap \left \{ |\mathcal{E}| \geq \frac{p_0 T}{2\Exz{\Delta \tau_1}} \right \} \cap \Tilde{\Omega}_T'}.
\end{align}
We also have
\begin{align}\label{Lambda:2terms}
\Lambda_k(I_k^0(f_1, \theta_1),f_0) - \Lambda_k(I_k^0(f_1, \theta_1),f) &= \Lambda_k(I_k^0(f_1, \theta_1),f_0) - \Lambda_k(I_k^0(f_1, \theta_1),f_1) + \Lambda_k(I_k^0(f_1, \theta_1),f_1) - \Lambda_k(I_k^0(f_1, \theta_1),f).
\end{align}
Firstly, if $|\mathcal{E}| > \frac{p_0}{2\Exz{\Delta \tau_1}} T$, then from Lemma \ref{lem:I0_model1},
\begin{align}\label{eq:lower_bound_I}
    |I_k^0(f_1, \theta_1)| \geq  \frac{(x_2-x_1) p_0}{4 \Exz{\Delta \tau_1}} T
\end{align}
and
\begin{align}\label{Lambda:first}
\Lambda_k(I_k^0(f_1, \theta_1),f_0) - \Lambda_k(I_k^0(f_1, \theta_1),f_1) = (\theta_k^0 - \theta_k^1)|I_k^0(f_1, \theta_1)| \geq \frac{(x_2-x_1) p_0 }{8 K\Exz{\Delta \tau_1}} \Tilde{M}_T T \epsilon_T,
\end{align}
since $\|\theta - \theta_1\|_1 \leq \zeta \e_T$ therefore $\theta_k^0 - \theta_k^1 \geq |\theta_k^0 - \theta_k| -  |\theta_k - \theta_k^1| \geq \Tilde{M}_T\e_T/K - \zeta \e_T \geq \frac{\Tilde{M}_T}{2K}\e_T$ for $T$ large enough.

Secondly, since $\forall t \in I_k^0(f_1, \theta_1),$ $\Tilde \lambda_t^k(\nu_1, h_1) \leq 0$ and $\Tilde \lambda_t^k(\nu, h) \leq 0$, we have
\begin{align}
\Lambda_k(I_k^0(f_1, \theta_1), f_1) - \Lambda_k(I_k^0(f_1, \theta_1),f) 
&=(\theta_k^1 - \theta_k)|I_k^0(f_1, \theta_1)| - \int_{I_k^0(f_1, \theta_1)}\!\!\!\! \Big((\tilde{\lambda}_t^k(\nu, h))_+ - (\tilde{\lambda}_t^k(\nu_1, h_1))_+\Big)dt \nonumber \\
&\geq(\theta_k^1 - \theta_k)|I_k^0(f_1, \theta_1)| - \int_{I_k^0(f_1, \theta_1)}\!\!\!\! |\tilde{\lambda}_t^k(\nu, h) - \tilde{\lambda}_t^k(\nu_1, h_1)|dt \nonumber \\
&\geq - \zeta T \epsilon_T - \int_{0}^T |\tilde{\lambda}_t^k(\nu, h) - \tilde{\lambda}_t^k(\nu_1, h_1))|dt\label{control:Lambda},
\end{align}
 where we have used the fact that by definition $|I_k^0(f_1, \theta_1)| \leq T$. Using Fubini's theorem, for any $l \in [K]$, we have
\begin{align}\label{eq:delta_tilde}
 \int_{0}^T& |\tilde{\lambda}^k_t(\nu, h) - \tilde{\lambda}_t^k(\nu_1, h_1))|dt = \int_{0}^T \left|\nu_k - \nu_k^1 + \sum_l \int_{t-A}^{t^-} (h_{lk} - h_{lk}^1)(t - s)dN_s^l \right| dt \nonumber \\
 &\leq T|\nu_k - \nu_k^1| + \sum_l \int_{T-A}^T \int_{s}^{s+A} |h_{lk} - h_{lk}^1|(t)|dt dN_s^l = T|\nu_k - \nu_k^1| + \sum_l \norm{h_{lk} - h_{lk}^1}_1 N^l[-A,T] \nonumber\\
 &\leq T \norm{f - f_1} \left(1 + \sum_l (\mu_l^0 + \delta_T)\right) \leq \zeta T \epsilon_T \left(1 + 2 \sum_l \mu_l^0 \right),
\end{align}
using the definition of $\eve'$ in Section \ref{sec:proof_conc_g}. Consequently, reporting the previous upper bound into \eqref{control:Lambda}, we obtain
\begin{align*}
\Lambda_k(I_k^0(f_1, \theta_1),f_1) - \Lambda_k(I_k^0(f_1, \theta_1),f) &\geq -\zeta T \epsilon_T ( 2+ 2 \sum_l \mu_l^0).
\end{align*}
Therefore, using now  \eqref{Lambda:first} and \eqref{control:Lambda} in \eqref{Lambda:2terms}, we arrive at
\begin{align*}
\Lambda_k(I_k^0(f_1, \theta_1),f_0) - \Lambda_k(I_k^0(f_1, \theta_1),f)
&\geq \frac{\Tilde{M}_T (x_2-x_1) p_0 }{8 K \Exz{\Delta \tau_1}} T \e_T -  \zeta T \epsilon_T ( 2+ 2 \sum_l \mu_l^0) \geq \frac{\Tilde{M}_T (x_2-x_1) p_0}{16 K \Exz{\Delta \tau_1}} T\epsilon_T ,
\end{align*}
since for $T$ large enough, $\Tilde{M}_T >  \frac{16 K \zeta  \Exz{\Delta \tau_1} (2+ 2 \sum_l \mu_l^0)}{(x_2-x_1) p_0}$. Reporting into \eqref{eq:type2_error_m2_c1}, we obtain
\begin{align*}
&\Probf{\left \{N^k(I_k^0(f_1, \theta_1)) - \Lambda_k(I_k^0(f_1, \theta_1), f_0) \geq - v_T \right \} \cap \left \{ |\mathcal{E}| \geq \frac{p_0 T}{2 \Exz{\Delta \tau_1}} \right \} \cap \Tilde{\Omega}_T'} \\
&\leq \Probf{\{N^k(I_k^0(f_1, \theta_1)) - \Lambda_k(I_k^0(f_1, \theta_1), f) \geq - v_T + \frac{\Tilde{M}_T (x_2-x_1) p_0  }{16 \Exz{\Delta \tau_1}} T \e_T \} \cap \eve' } \\
&\leq \Probf{\{N^k(I_k^0(f_1, \theta_1)) - \Lambda_k(I_k^0(f_1, \theta_1), f) \geq v_T\} \cap \eve'}, 
\end{align*}
if $ \Tilde{M}_T > \frac{16 w_T \Exz{\Delta \tau_1}}{(x_2-x_1) p_0}$, which is true for $\Tilde{M}$ large enough (recall that $ \Tilde{M}_T = \Tilde{M} \sqrt{\kappa_T} $) if $w_T \leq C \sqrt{\kappa_T}$ with $C>0$ a constant.

Similarly to the proof of Lemma 1 in \cite{donnet18:supplement}, we can adapt inequality (7.7) from \cite{Hansen:Reynaud:Rivoirard} with $H_t = \mathds{1}_{I_k^0(f_1, \theta_1)}(t)$ to the conditional probability $\Exf{.|\mathcal{G}_0}$ and the supermartingale $\int_0^T \mathds{1}_{I_k^0(f_1, \theta_1)}(t) (dN_t - \lambda_t^k(f, \theta)dt)$. With $\tau = T$, $x_T = x_1T \epsilon_T^2$, we obtain
\begin{align}\label{eq:type_2_error}
 \Probf{\{N^k(I_k^0(f_1, \theta_1)) - \Lambda_k(I_k^0(f_1, \theta_1), f) > v_T \} \cap \eve'} \leq e^{-x_T T \epsilon_T^2} = o(e^{-(\kappa_T + c_1) T \epsilon_T^2}), \quad \text{ if } x_T >\kappa_T + c_1.
\end{align}
For this to be true, we also need $v_T  > \sqrt{2 \Tilde{v} (\kappa_T + c_1) T \epsilon_T^2} + (\kappa_T + c_1)T \epsilon_T^2/3$ where $\tilde{v}$ is an upper bound of $H_t^2 \circ \Lambda^k_t(f)$. Using the fact that $\forall t \in I_k^0(f_1, \theta_1), \: \Tilde \lambda_t^k(\nu_1, h_1) \leq 0$, we have
\begin{align*}
H_t^2 \circ \Lambda^k_t(f) &= \int_{I_k^0(f_1, \theta_1)} \lambda_t^k(f, \theta) dt = \theta_k |I_k^0(f_1, \theta_1)| + \int_{I_k^0(f_1, \theta_1) \cap \{\Tilde{\lambda}_t^k(\nu, h) > 0\}} \Tilde{ \lambda}^k_t(\nu, h) dt \\
&\leq \theta_k |I_k^0(f_1, \theta_1)| + \int_{I_k^0(f_1, \theta_1) \cap \{\Tilde{\lambda}^k_t(\nu, h) > 0\}} |\Tilde{\lambda}^k_t(\nu, h) - \Tilde{\lambda}^k_t(\nu_1, h_1)| dt \\
&\leq \theta_k |I_k^0(f_1, \theta_1)| + \zeta T \epsilon_T \left(1 + 2 \sum_l \mu_l^0 \right) \leq T (\theta_k + \Tilde{M}_T \epsilon_T/K) \leq \theta_k^0 T =: \Tilde{v},
\end{align*}
using \eqref{eq:delta_tilde} and since for $T$ large enough, $\zeta K (1 + 2 \sum_l \mu_l^0) < M_T\leq  \Tilde{M}_T$. Consequently, if $w_T   > \sqrt{2 \theta_k^0 (\kappa_T + c_1)} + (\kappa_T + c_1) \epsilon_T/3$ and $w_T \leq C \sqrt{\kappa_T}$ (which is possible since $\epsilon_T = o(1/\sqrt{\kappa_T})$ by assumption), then \eqref{eq:type_2_error} holds and we can finally
conclude that $\Exf{(1 - \phi_k(f_1, \theta_1))\mathds{1}_{\Tilde{\Omega}_T'} } = o(e^{-(\kappa_T + c_1)  T \epsilon_T^2})$ is verified, which leads to the second part of Lemma \ref{lem:test_theta}.

In the alternative case where $\theta_k > \theta_k^0 + \Tilde{M}_T \epsilon_T/K$, similar arguments can be applied with $I_k^0(f_1, \theta_1)$ defined as in \eqref{def:I_0_model1}  and $\mathcal{E}$ defined as in \eqref{E:model1} except that $n_1 =\lfloor   2 \nu_k^0/(\kappa_1 c_\star)\rfloor + 1$. We then use the following test, with $v_T= w_T T \e_T$
$$\phi_k(f_1, \theta_1):= \mathds{1}_{N^k(I_k^0(f_1, \theta_1)) - \Lambda_k(I_k^0(f_1, \theta_1), f_0) > v_T} \vee \mathds{1}_{|\mathcal{E}| < \frac{p_0 T}{2\Exz{\Delta \tau_1}}}.$$
Then Hoeffding's inequality and inequality (7.7) from \cite{Hansen:Reynaud:Rivoirard} lead to $
\Exz{\phi_k(f_1, \theta_1)\mathds{1}_{\Tilde{\Omega}_T} } = o(e^{-u_0 T \epsilon_T^2})$.
For the second part of  Lemma \ref{lem:test_theta}, we first note that in this case, since $\forall t \in I_k^0(f_1, \theta_1), \: \lambda_t^k(f, \theta) \geq \theta_k$ (and also $\lambda_t^k(f_0, \theta_0) = \theta_k^0$, $\lambda_t^k(f_1, \theta_1) = \theta_k^1$), then on the event $|\mathcal{E}| \geq \frac{p_0 T}{2\Exz{\Delta \tau_1}}$,
\begin{align*}
\Lambda_k(I_k^0(f_1, \theta_1),f_0) &- \Lambda_k(I_k^0(f_1, \theta_1),f) 
\leq (\theta_k^0 - \theta_k^1)|I_k^0(f_1, \theta_1)| + (\theta_k^1 - \theta_k)|I_k^0(f_1, \theta_1)| \\
&\leq (- \Tilde{M}_T \epsilon_T / K + \zeta \epsilon_T) |I_k^0(f_1, \theta_1)|
\leq - \frac{\Tilde{M}_T \epsilon_T |I_k^0(f_1, \theta_1)|}{  2 K} \leq - \frac{(x_2-x_1) p_0}{8 K \Exz{\Delta \tau_1}}\Tilde{M}_T T \epsilon_T,
\end{align*}
for $T$ large enough and using \eqref{eq:lower_bound_I}.
Consequently,
\begin{align*}
&\Probf{\{N^k(I_k^0(f_1, \theta_1)) - \Lambda_k(I_k^0(f_1, \theta_1), f_0) \leq v_T\} \cap \left \{ |\mathcal{E}| \geq \frac{p_0 T}{2\Exz{\Delta \tau_1}} \right \} \cap \eve'} \\
&\leq \Probf{\{N^k(I_k^0(f_1, \theta_1)) - \Lambda_k(I_k^0(f_1, \theta_1), f) \leq  v_T -  \frac{  (x_2-x_1) p_0}{8\Exz{\Delta \tau_1}} \Tilde{M}_T T \e_T\} \cap \eve'} \\
&\leq \Probf{\{N^k(I_k^0(f_1, \theta_1)) - \Lambda_k(I_k^0(f_1, \theta_1), f) \leq  - v_T\} \cap \eve'},
\end{align*}
if $\Tilde{M}_T > \frac{16 K \Exz{\Delta \tau_1}}{(x_2-x_1) p_0} w_T$. Applying inequality (7.7) from \cite{Hansen:Reynaud:Rivoirard}, we can finally obtain
\begin{align*}
\Exf{(1 - \phi_k(f_1, \theta_1))\mathds{1}_{\Tilde{\Omega}_T'} } = o(e^{-(\kappa_T + c_1) T \epsilon_T^2}),
\end{align*}
which ends the proof of Lemma \ref{lem:test_theta}. 

\medskip

\textbf{Proof of Lemma \ref{lem:I0_model1} }




Let $(f_0, \theta_0) \in \mathcal{F} \times \Theta$, $(f_1, \theta_1) \in ( \mathcal{F} \times \bar A(\Tilde{M}_T \epsilon_T)^c) \cap A_{L_1}(M_T \epsilon_T)$ and $k \in [K]$ such that $|\theta^1_k - \theta_k^0| > \Tilde{M}_T \epsilon_T/K$. For this $k$, from assumption \eqref{ass:identif_theta}, there exists $ l \in [K]$ and $x_1,x_2, c_\star > 0$ such that $\forall x \in [x_1,x_2]$, $h_{lk}^0(x) \leq - c_\star<0$. We first consider the case $\theta_k^1 < \theta_k^0 - \Tilde{M}_T\epsilon_T/K$. Since $(f_1, \theta_1) \in  A_{L_1}(M_T \epsilon_T)$, we also have that $|\theta_k^1 + \nu_k^1 - \theta_k^0 -\nu_k^0| \leq M_T \epsilon_T$, which implies that $\nu_k^1 > \nu_k^0 - (M_T -  \Tilde{M}_T/K) \epsilon_T > \nu_k^0/2$. For $0<\kappa_1 <1$, we define
\begin{align*}
B_1 = \{x \in [0,A]; \: h_{lk}^{1-}(x) > \kappa_1 c_\star \}  , \quad n_1 = \left\lfloor \frac{   2\nu_k^1 }{ \kappa_1 c_\star } \right\rfloor +1.
\end{align*}
Moreover, since $\norm{h_{lk}^0 - h_{lk}^1}_1 \leq M_T\epsilon_T$ and $h_{lk}^{0-} (x) \geq c_\star$ for $x \in [x_1,x_2]$, 
\begin{align*}
&| [x_1,x_2] \cap B_1^c| c_\star(1 - \kappa_1) \leq \int_{[x_1,x_2] \cap B_1^c} (h_{lk}^1-h_{lk}^0)(x)dx \leq M_T\epsilon_T \\
&\implies |[x_1,x_2] \cap B_1| \geq (x_2 - x_1) - \frac{M_T\epsilon_T}{c_\star(1 - \kappa_1)} \geq 3(x_2 - x_1)  / 4, 
\end{align*}
for $T$ large enough. 

Now let  $\delta' = (x_2 - x_1) / 4$.  For $j \in \mathcal{E}$, we denote $T_1, \dots, T_{n_1}$ the $n_1$ events occurring on $[\tau_j, \tau_j + \delta']$. For $t \in [\tau_j + x_1 + \delta', \tau_j + x_2]$ , we have $t - T_i \in [x_1, x_2]$ for any $i \in [n_1]$ and
\begin{align*}
\Tilde \lambda^k_t(\nu_0, h_0) &= \nu_k^0 + \sum_{i \in [n_1]} h_{lk}^0(t - T_i) < \nu_k^0 - n_1 c_\star < 2\nu_k^1 - n_1 \kappa_1 c_\star < 0,
\end{align*}
by definition of $n_1$. Similarly, for $t \in B_1 + [\tau_j, \tau_j + \delta']$, we have $t - T_i \in B_1$ and therefore
\begin{align*}
\Tilde \lambda_t(\nu_1, h_1) &= \nu_k^1 + \sum_{i \in [n_1]} h_{lk}^1(t - T_i) < 2 \nu_k^1 - n_1 \kappa_1 c_\star < 0.
\end{align*}
Consequently, for $t \in ([x_1,x_2] \cap B_1) +  [\tau_j, \tau_j + \delta']$,
$
\lambda^k_t(f_0, \theta_0) = \theta_k^0 $  and $ \lambda_t^k(f_1, \theta_1) = \theta_k^1,$
and thus $([x_1,x_2] \cap B_1) +  [\tau_j, \tau_j + \delta'] \subset I_k^0(f_1, \theta_1)$. Moreover, we have
\begin{align*}
\left| ([x_1,x_2] \cap B_1) +  [\tau_j, \tau_j + \delta'] \right|\geq 3(x_2 - x_1)/4 - (x_2 - x_1)/4  \geq (x_2 - x_1)/2.
\end{align*}
Consequently,
\begin{align*}
|I_k^0(f_1, \theta_1)| &= \sum_{j=0}^{J_T} [\tau_j,\tau_{j+1}] \cap \{t \geq 0 ; \: \lambda_t^k(f_0, \theta_0) = \theta_0, \: \lambda_t^k(f_1, \theta_1) = \theta_1\} \geq \sum_{j \in [J_T]} \frac{x_2 - x_1}{2} \mathds{1}_{j \in \mathcal{E}}.
\end{align*}
In the alternative case $\theta_k^1 > \theta_k^0 + \Tilde{M}_T\epsilon_T/K$, similar computations can be derived by defining  $n_1$ as
$
n_1 = \min \{n \in \N; \: n \kappa_1 c_\star > \nu_k^0 \}.
$



\section{Lemmas on $L_T(f_0) - L_T(f)$} \label{sec:lem:KL}

 For $f_0, f \in \mathcal{F}$, we define the Kullback-Leibler (KL) divergence in the Hawkes model as
\begin{equation}\label{eq:def_KL_brute}
    KL(f_0,f) = \mathbb{E}_0[L_T(f_0) - L_T(f)].
\end{equation}
\ds{With a slight abuse of notation, we still use the same notations $L_T(f_0), L_T(f),   KL(f_0,f)$ in the nonlinear model with shifted ReLU link function with the additional shift parameter $\theta$. We also note that with the standard ReLU link function, the KL divergence can be infinite for some $f \in \mathcal{F}$ - e.g., if there exists $t \in [0,T]$ such that $dN_t^k = 1$ and $ \lambda^k_t(f) = 0$. However, in this model, for any $f \in B_\infty(\epsilon_T)$, $\lambda^k_t(\nu, h) \geq  \lambda^k_t(\nu_0, h_0) $, which implies that $ KL(f_0,f) < +\infty$. The next lemma provides some upper bound on the KL divergence on $B_\infty(\epsilon_T)$ with all the link functions considered in Theorem \ref{thm:conc_g} and Proposition \ref{prop:relu}. } 

\subsection{Lemma to bound the  Kullback - Leibler divergence }  \label{sec:KL}

\begin{lemma}\label{lem:KLdecomp}
\ds{Under the assumptions of Theorem \ref{thm:conc_g} and of Case 2 of Proposition \ref{prop:relu}, for any $f \in B_\infty(\epsilon_T)$ and $T$ large enough,
\begin{equation*}
    0 \leq KL(f_0,f) \leq     \kappa_1 T \epsilon_T^2, 
\end{equation*}
and, under the assumptions of Case 1 of  Proposition \ref{prop:relu}, we similarly have
\begin{equation*}
    0 \leq KL(f_0,f) \leq  \kappa_2  (\log T)^2 T,
\end{equation*}    
with $\kappa_1, \kappa_2 > 0$ constants that only depends on $(\phi_k)_k$ and $f_0$.}
\end{lemma}


\begin{remark}\label{rem:other_rates}
\ds{For the models considered in Theorem  \ref{thm:conc_g} and with the shifted ReLU link function (Case 2 of Proposition \ref{prop:relu}),  for $f \in B_2(\epsilon_T,B)$, we instead obtain
\begin{align*}
    0 \leq KL(f_0,f) \lesssim (\log \log T) T \epsilon_T^2.
\end{align*}
Moreover, with the standard ReLU link function (Case 1 of Proposition \ref{prop:relu}), without assuming that the additional condition \eqref{condA3} holds, we can also obtain the sub-obtimal bound
\begin{align*}
     0 \leq KL(f_0,f) \lesssim  T \epsilon_T,
\end{align*}
which would also lead to the sub-optimal posterior concentration rate $\sqrt{\epsilon_T}$.}
\end{remark}

\begin{proof} 
\ds{For simplicity of exposition, throughout this proof, we use the notation $\lambda^k_t(f), \lambda^k_t(f_0)$ for the intensity in all models, therefore including the case  $\lambda^k_t(f, \theta), \lambda^k_t(f_0, \theta_0)$ (Case 2 of Proposition \ref{prop:relu}).} 

Firstly, similarly to the proof of Lemma 2 of \cite{donnet18:supplement}, we can easily prove that $KL(f_0,f) \geq 0$. Secondly, since intensities are predictable, we have
\begin{align}\label{eq:exp_martingale}
\Exz{\int_{0}^T \log \left( \frac{ \lambda^k_t(f_0) }{\lambda^k_t(f)}\right)(dN_t^k-\lambda^k_t(f_0)dt)}=0.
\end{align} 
Since
\begin{equation}\label{eq:def_KL_dev}
KL(f_0, f)= \sum_k \Exz{\int_{0}^T \log \left(\frac{ \lambda^k_t(f_0) }{\lambda^k_t(f)}\right)dN_t^k +  \int_0^T (\lambda^k_t(f) - \lambda^k_t(f_0)) dt },
\end{equation}
then, with
\begin{align}\label{eq:RT}
R_T= \sum_k \Exz{ \mathds{1}_{\eve^c} \int_{0}^T\lambda^k_t(f_0) \log \left( \frac{ \lambda^k_t(f_0) }{\lambda^k_t(f)}\right)dt} + \Exz{\mathds{1}_{\eve^c}  \int_0^T (\lambda^k_t(f) - \lambda^k_t(f_0)) dt},
\end{align}

\begin{align}\label{eq:def_kl_rt}
    KL(f_0, f) 
= \sum_k \Exz{ \mathds{1}_{\eve} \left(\int_{0}^T\lambda^k_t(f_0) \log \left( \frac{ \lambda^k_t(f_0) }{\lambda^k_t(f)}\right)dt+  \int_0^T (\lambda^k_t(f) - \lambda^k_t(f_0)) dt \right)} + R_T.
\end{align}

We first show that $R_T = o(T \e_T^2)$. For the first term on the RHS of \eqref{eq:RT}, if $f \in B_\infty(\epsilon_T)$, we use that $\log x \leq x - 1$ for $x \geq 1$ and we have
\small
\begin{align}
\sum_{k} &\Exz{\mathds{1}_{\eve^c}  \int_0^T \log \frac{\lambda^k_t(f)}{\lambda^k_t(f_0)} \lambda^k_t(f_0)dt } \leq \sum_{k} \Exz{\mathds{1}_{\eve^c}\int_0^T   \mathds{1}_{\lambda_t^k(f) > \lambda_t^k(f_0)}\log \frac{\lambda^k_t(f)}{\lambda^k_t(f_0)} \lambda^k_t(f_0)dt } \nonumber \\
&\leq \sum_{k} \Exz{\int_0^T \mathds{1}_{\eve^c} \mathds{1}_{\lambda_t^k(f_0) > 0}\left( \lambda^k_t(f)- \lambda^k_t(f_0)\right)dt} \nonumber \\
&\leq \sum_{k} T L \left(|\nu_k^0 - \nu_k| + \sum_l \norm{h_{lk} - h_{lk}^0}_\infty \Exz{ \mathds{1}_{\eve^c} \sup \limits_{t \in [0,T]} N^l[t-A,t)} \right) \nonumber \\
&\leq T L \sum_{k} \left(|\nu_k^0 - \nu_k| + \sum_l \norm{h_{lk} - h_{lk}^0}_\infty\right)  \Exz{\mathds{1}_{\eve^c} \max_l \sup \limits_{t \in [0,T]} N^l[t-A,t)}\leq  L T^{1-\beta} \epsilon_T \label{eq:bound_RT}
\end{align}
\normalsize
for $T$ large enough, using Lemma \ref{lem:main_event} for $\beta > 0$. \ds{If the model verifies Assumption \ref{ass-psi}(i),}
and $f \in B_2(\epsilon_T, B)$, we have
\begin{align*}
    \frac{\lambda^k_t(f)}{\lambda^k_t(f_0)} \vee \frac{\lambda^k_t(f_0)}{\lambda^k_t(f)} \leq 2\frac{2\theta_k^0 + 2 L \nu_k^0 +  L (B + \max_l \norm{h_{lk}^0}_\infty) \sup_{t} N[t-A,t)}{\ds{\inf_x \phi_k(x)}},
\end{align*}
therefore
\begin{align}
\Exz{\mathds{1}_{\eve^c}  \int_0^T \left|\log \frac{\lambda^k_t(f)}{\lambda^k_t(f_0)}\right| \lambda^k_t(f_0)dt } &\lesssim \Exz{\mathds{1}_{\eve^c} \max_l \sup_{t \in [0,T]} N^l[t-A,t) \int_0^T \lambda^k_t(f_0)dt } \nonumber \\
&\lesssim T \Exz{\mathds{1}_{\eve^c} \left(\sup_{t \in [0,T]} N[t-A,t)\right) \left(\nu_k^0  +  \max_l \norm{h_{lk}^0}_\infty \sup_{t \in [0,T]} N^l[t-A,t)\right) } \nonumber \\
&\lesssim T \Exz{\mathds{1}_{\eve^c} \max_l \left(\sup \limits_{t \in [0,T]} N^l[t-A,t)\right)^2} \lesssim  T^{1-\beta}. \nonumber
\end{align}
\ds{If instead the model verifies Assumption \ref{ass-psi}(ii),}
using that $\log \phi_k$ is $L_1$-Lipschitz for any $k$, we can alternatively use that 
\small
\begin{align}
\sum_{k} & \Exz{\mathds{1}_{\eve^c}  \int_0^T \left|\log \frac{\lambda^k_t(f)}{\lambda^k_t(f_0)} \right|\lambda^k_t(f_0)dt }
\leq L_1 \sum_{k} \Exz{\int_0^T \mathds{1}_{\eve^c} \lambda_t^k(f_0)|\Tilde \lambda^k_t(f)- \Tilde \lambda^k_t(f_0)|dt} \nonumber \\
&\lesssim \sum_{k} T \left(|\nu_k^0 - \nu_k| + \sum_l \norm{h_{lk} - h_{lk}^0}_\infty \Exz{ \mathds{1}_{\eve^c} \max_l \left(\sup \limits_{t \in [0,T]} N^l[t-A,t)\right)^2} \right) \lesssim  T^{1-\beta} \nonumber.
\end{align}
\normalsize
We can additionally bound the second term of \eqref{eq:RT} in a similar fashion and conclude that, in all cases, $R_T = O(T^{1-\beta})  = o(T \epsilon_T^2)$ for $\beta$ large enough.


\ds{To bound the first term of the RHS of \eqref{eq:def_kl_rt}, we consider separately the models satisfying Assumption \ref{ass-psi}(i) and (ii) and Case 1 and Case 2 of Proposition \ref{prop:relu}. }

\noindent 
\textbf{Scenario 1: under Assumption \ref{ass-psi}(i) or Case 2 of Proposition \ref{prop:relu} } 

\ds{Under Assumption \ref{ass-psi}(i), for any $f \in B_\infty(\epsilon_T)$ or $f \in B_2(\epsilon_T,B)$  and $t \geq 0$,  $\lambda^k_t(f) \geq \inf_x \phi_k(x)\geq \min_k \inf_x \phi_k(x)$ and $\lambda^k_t(f_0) \leq L\nu_k^0 + L\sup_{t \in [0,T]} N[t-A,t) \sum_l\| h_{lk}^0\|_\infty $.  In Case 2 of Proposition \ref{prop:relu}, for $T$ large enough, $t \in [0,T]$ and $\theta \in B_\infty^\Theta(\epsilon_T)$,  $\lambda^k_t(f, \theta) \geq \theta_k \geq \theta_k^0/2$ and $\lambda^k_t(f_0, \theta_0) \leq \theta_k^0 + L\nu_k^0 + L\sup_{t \in [0,T]} N[t-A,t) \sum_l\| h_{lk}^0\|_\infty $. Therefore, in this scenario,  on $\eve$,  $\lambda^k_t(f_0)/\lambda^k_t(f) \leq \ell_0 \log T$ for some $\ell_0 > 0$. Thus, with $\chi(x) = - \log x + x - 1$, we have}
\begin{align*}
    KL(f_0,f) - R_T &= \sum_k \Exz{ \mathds{1}_{\eve} \left(\int_{0}^T\lambda^k_t(f_0) \left(\log \left( \frac{ \lambda^k_t(f_0) }{\lambda^k_t(f)}\right) + \frac{\lambda^k_t(f)}{\lambda^k_t(f_0)} - 1\right) dt \right)} \\
    &= \sum_k \Exz{ \mathds{1}_{\eve} \left(\int_{0}^T\lambda^k_t(f_0) \chi\left( \frac{ \lambda^k_t(f) }{\lambda^k_t(f_0)}\right) dt \right)} \\
    &\leq \frac{ 4 \log (\ell_0 \log T )  }{  \ds{\min_k \inf_x \phi_k(x)}  } \sum_k \Exz{ \mathds{1}_{\eve} \int_{0}^T  (\lambda^k_t(f_0) - \lambda^k_t(f))^2 dt},
\end{align*}
since for any $r_T\in(0,1/2]$ and $x \geq r_T,$ we have $\chi(x) \leq 4 \log r_T^{-1} (x-1)^2$ (see the proof of Lemma 2 of \cite{donnet18:supplement}). Note that if $f \in B_\infty(\epsilon_T)$, $\forall t \in [0,T]$,  $\lambda^k_t(f) \geq \lambda^k_t(f_0)$ and  we obtain instead
\begin{align*}
KL(f_0, f) - R_T & \leq  \frac{1}{\ds{\min_k \inf_x \phi_k(x)} } \sum_k \Exz{ \mathds{1}_{\eve} \int_{0}^T (\lambda^k_t(f_0) - \lambda^k_t(f))^2 dt  }.
\end{align*}
Moreover, since $\phi_k$ is $L$-Lipschitz, under Assumption \ref{ass-psi},
\begin{align*}
    |\lambda^k_t(f_0) - \lambda^k_t(f)| &=   |\phi_k(\Tilde \lambda^k_t(\nu_0, h_0))  - \phi_k(\Tilde \lambda^k_t(\nu, h))| \leq  L |\Tilde \lambda^k_t(\nu_0, h_0) - \Tilde \lambda^k_t(\nu, h)| \\
    &\leq L |\nu_k - \nu_k^0| + L \sum_{l} \int_{t-A}^{t^-}|h_{lk} - h_{lk}^0|(t-s)dN^l_s,
\end{align*}
\ds{and in Case 2 of Proposition \ref{prop:relu}, we have 
\begin{align*}
    |\lambda^k_t(f_0, \theta_0) - \lambda^k_t(f, \theta)| &=   |\theta_k^0 + \phi_k(\Tilde \lambda^k_t(\nu_0, h_0)) - \theta_k - \phi_k\Tilde \lambda^k_t(\nu, h))| \leq  |\theta_k^0 - \theta_k|  + L |\Tilde \lambda^k_t(\nu_0, h_0) - \Tilde \lambda^k_t(\nu, h)| \\
    &\leq |\theta_k - \theta_k^0| + L |\nu_k - \nu_k^0| + L \sum_{l} \int_{t-A}^{t^-}|h_{lk} - h_{lk}^0|(t-s)dN^l_s.
\end{align*}}
Using the same computations as in the proof of Lemma 2 of \cite{donnet18:supplement}, we obtain 
$$\sum_k  \Exz{ \mathds{1}_{\eve} \left(\int_{0}^T (\lambda^k_t(f_0) - \lambda^k_t(f))^2 \right)dt  } \leq \gamma_0 T \left( |\nu_k - \nu_k^0|^2 + \sum_{l}\| h_{lk} - h_{lk}^0\|_2^2\right) \leq \gamma_0 T \epsilon_T^2,$$
or
$$\sum_k  \Exz{ \mathds{1}_{\eve} \left(\int_{0}^T (\lambda^k_t(f_0, \theta_0) - \lambda^k_t(f, \theta))^2 \right)dt  } \leq \gamma_0 T \left( \sum_k |\theta_k - \theta_k^0|^2 + |\nu_k - \nu_k^0|^2 + \sum_{l}\| h_{lk} - h_{lk}^0\|_2^2\right) \leq \gamma_0 T \epsilon_T^2,$$
with $\gamma_0 := \max(1, L)\left[3 + 6K \sum_k \left(A\Exz{\lambda^k_0(f_0)^2}+ \Exz{\lambda^k_0(f_0)}\right)\right]$. Consequently,
\begin{align}\label{eq:def_kappa_0}
KL(f_0, f) - R_T &\leq  \begin{cases}
\frac{ 4 \log (\ell_0 \log T )  }{\ds{\min_k \inf_x \phi_k(x)} }  \gamma_0 T \epsilon_T^2 & \text{if} \quad f \in B_2(\epsilon_T, B) \\
\frac{ \gamma_0}{\ds{\min_k \inf_x \phi_k(x)} } T \epsilon_T^2 & \text{if} \quad f \in B_\infty(\epsilon_T). 
\end{cases}
\end{align}
Therefore, $KL(f_0,f) \leq \kappa_1' (\log \log T) T \epsilon_T^2$,
with $\kappa_1' =  \frac{ 8 \gamma_0}{\ds{\min_k \inf_x \phi_k(x)} }$ if $f \in B_2(\epsilon_T,B)$ - or $KL(f_0,f) \leq \kappa_1 T \epsilon_T^2$ with $\kappa_1 = \frac{2}{\ds{\min_k \inf_x \phi_k(x)} }$  if $f \in B_\infty(\epsilon_T)$ .

\vspace{3mm}

\noindent
\textbf{Scenario 2:  Under Assumption \ref{ass-psi}(ii), i.e., $\phi_k > 0$, and $\log \phi_k$ and $\sqrt{ \phi_k}$ are $L_1$-Lipschitz, $L_1>0$.}

For $k \in [K]$, let $\Lambda^k(f) := \int_0^T  \lambda_t^k(f) dt$. Then for $t \in [0,T]$, we define
$$\alpha^k_t(f) = \frac{\lambda^k_t(f)}{\Lambda^k(f)}. $$
From \eqref{eq:def_kl_rt}, we have
\begin{align}
 KL(f_0, f) - R_T
&= \sum_k \Exz{ \mathds{1}_{\eve} \left(\int_0^T \lambda^k_t(f_0) \log \left( \frac{ \lambda^k_t(f_0) }{\lambda^k_t(f)}\right)dt+  \int_0^T (\lambda^k_t(f) - \lambda^k_t(f_0)) dt \right)} \nonumber\\ 
 &=  \sum_k \Exz{\mathds{1}_{\eve} \left( \Lambda_{A}^k(f_0) \int_{A^k(T)} \alpha^k_t(f_0) \log \left( \frac{ \alpha^k_t(f_0) }{\alpha^k_t(f)}\right)dt + \Lambda^k(f_0) \log \left(\frac{\Lambda^k(f_0)}{\Lambda^k(f)}\right) + (\Lambda^k(f)-\Lambda^k(f_0)) \right)} \nonumber \\
&\leq   \sum_k \Exz{\mathds{1}_{\eve} \left( \Lambda^k(f_0) \int_0^T \alpha^k_t(f_0) \log \left( \frac{ \alpha^k_t(f_0) }{\alpha^k_t(f)}\right)dt + \frac{( \Lambda^k(f_0 ) - \Lambda^k(f))^2 }{ \Lambda^k(f_0) } \right)} \nonumber,
\end{align}
where in the last inequality we have used that $\chi(x) \leq (x-1)^2$ for $x \geq 1/2$, with $x = \frac{\Lambda^k(f)}{\Lambda^k(f_0)}$. In fact, we have
\begin{align*}
    |\Lambda^k(f) - \Lambda^k(f_0)| &\leq   T L |\nu_k - \nu_k^0| + L  \sum_l \norm{h_{l k} - h^0_{l k}}_1 N^l[-A,T] \leq T L \epsilon_T (1 + 2  \max_l \mu_l^0 ),
\end{align*}
using that on $\eve$,
$
 N^l[-A,T] \leq  T \mu_l^0 + T \delta_T  \leq 2 T \mu_l^0.
$
Moreover, on $\eve$, using the notations of Section \ref{sec:proof_conc_g}, we have
\begin{align*}
    \Lambda^k(f_0) \geq \phi_k(\nu_k^0) \sum_{j=1}^{J_T-1} (U_j^{(1)} - \tau_j) \geq \phi_k(\nu_k^0) 
     \frac{T}{2\mathbb{E}_0[\Delta \tau_1] \|r_0\|_1}  =: y_0 T, 
\end{align*}
for some $y_0 > 0$. Similarly, for $f \in B_2(\epsilon_T,B)$ or $f \in B_\infty(\epsilon_T)$, we have
\begin{align*}
  \Lambda^k(f) \geq \phi_k(\nu_k) \sum_{j=1}^{J_T-1} (U_j^{(1)} - \tau_j) \geq \phi_k(\nu_k^0/2) 
     \frac{T}{2\mathbb{E}_0[\Delta \tau_1] \|r_0\|_1}.
\end{align*}
Consequently,
\begin{align*}
     \frac{1}{2} \leq 1 - \frac{|\Lambda^k(f) - \Lambda^k(f_0)|}{\Lambda_A(f_0)} \leq \frac{\Lambda^k(f)}{\Lambda^k(f_0)} \leq  1 + \frac{|\Lambda^k(f) - \Lambda^k(f_0)|}{\Lambda_A(f_0)}  \leq 1 + \frac{1 + 2A \max_l \mu_l^0 }{ y_0} \epsilon_T = 1 + O(\e_T),
\end{align*}
for $T$ large enough, and
\begin{align*}
   \frac{(\Lambda^k(f) - \Lambda^k(f_0))^2}{\Lambda^k(f_0)}  &\leq \frac{L^2 T^2\epsilon_T^2(1 + 2 \max_l \mu_l^0 )^2}{\Lambda^k(f_0)} \leq \frac{L^2 T \epsilon_T^2 (1 + 2A \max_l \mu_l^0 )^2}{y_0}.
\end{align*}
Additionally, on $\eve$, on the one hand, for $f \in B_2(\epsilon_T,B)$, we also have that for any $t \in [0,T]$, since $ \lambda_t^k(f_0) \leq \lambda_t^k(f) + \epsilon_T + B C_\beta \log T \implies \frac{\lambda_t^k(f_0)}{\lambda_t^k(f)} \leq M_0 \log T$ for some $M_0 > 0$, then
\begin{align*}
    \frac{\alpha^k_t(f_0)}{\alpha^k_t(f)} =  \frac{\lambda^k_t(f_0) \Lambda^k(f)}{\lambda^k_t(f) \Lambda^k(f_0)} \leq M_0 \log T \frac{\Lambda^k(f)}{\Lambda^k(f_0)} \leq M \log T + O(M_0 \log T \epsilon_T).
\end{align*}
Applying Lemma 8.7 from \cite{ghosal00}, we have, for any $M \geq M_0$,
\small
\begin{align*}
    \int_0^T \alpha^k_t(f_0) \log \left( \frac{ \alpha^k_t(f_0) }{\alpha^k_t(f)}\right)dt &\leq \log (M \log T) \int_{0}^T \left(\sqrt{\alpha^k_t(f_0)} - \sqrt{\alpha^k_t(f)}\right)^2dt .
\end{align*}
\normalsize
 Moreover,
\small
\begin{align*}
 \int_0^T \left(\sqrt{\alpha^k_t(f_0)} - \sqrt{\alpha^k_t(f)}\right)^2dt &\leq   \int_0^T \frac{1}{\Lambda^k(f_0)} \left(\sqrt{\lambda^k_t(f_0)} - \sqrt{\frac{\Lambda^k(f_0)}{\Lambda^k(f)}\lambda^k_t(f)}\right)^2dt \\
&\leq \frac{2}{\Lambda^k(f_0)}  \int_0^T  \left(\sqrt{\lambda^k_t(f_0)} - \sqrt{\lambda^k_t(f)}\right)^2dt  + \frac{1}{\Lambda^k(f_0)}  \int_0^T  \lambda^k_t(f) \left(1 - \sqrt{\frac{\Lambda^k(f_0)}{\Lambda^k(f)}}\right)^2dt \\
    &\lesssim \frac{1}{\Lambda^k(f_0)}  \int_0^T  \left(\sqrt{\lambda^k_t(f_0)} - \sqrt{\lambda^k_t(f)}\right)^2dt + \frac{(\Lambda^k(f) - \Lambda^k(f_0))^2}{\Lambda^k(f_0)^2}.
\end{align*}
\normalsize
On the other hand, if $f \in B_\infty(\epsilon_T)$, then $\lambda_t^k(f_0) \leq \lambda_t^k(f)$ and we have
\begin{align*}
       \int_0^T \alpha^k_t(f_0) \log \left( \frac{ \alpha^k_t(f_0) }{\alpha^k_t(f)}\right)dt \leq \frac{2}{\Lambda^k(f_0)}  \int_0^T \left(\sqrt{\lambda^k_t(f_0)} - \sqrt{\lambda^k_t(f)}\right)^2dt + \frac{4(\Lambda^k(f) - \Lambda^k(f_0))^2}{\Lambda^k(f_0)^2}.
\end{align*}
Moreover, in this case, 
\begin{align*}
    \int_0^T  \left(\sqrt{\lambda^k_t(f_0)} - \sqrt{\lambda^k_t(f)}\right)^2dt &=  \int_0^T  \left(\sqrt{\phi_k(\Tilde \lambda^k_t(\nu_0, h_0))} - \sqrt{\phi_k(\Tilde \lambda^k_t(\nu, h))}\right)^2dt \\
    &\leq L_1^2 \int_{A^k(T)}  \left(\Tilde \lambda^k_t(\nu_0, h_0) - \Tilde \lambda^k_t(\nu, h)\right)^2dt \lesssim T \epsilon_T^2.
\end{align*}
Finally, we obtain that
\begin{align*}
    KL(f_0, f) \lesssim \begin{cases} 
    (\log \log T) T \e_T^2 & \text{if} \quad f \in B_2(\epsilon_T, B) \\
 T \e_T^2  & \text{if} \quad f \in B_\infty(\epsilon_T)
\end{cases}.
\end{align*}

\noindent
\textbf{Scenario 3: Case 1 of Proposition \ref{prop:relu}, i.e., $\phi_k(x) = (x)_+, \forall k \in [K]$.} 

In a Hawkes model with the standard ReLU link function, we can obtain two types of rates, under and without condition \eqref{condA3}. We consider $f \in B_{\infty}(\e_T)$ so that $\forall t \in [0,T], \Tilde \lambda^k_t(\nu, h) \geq \Tilde \lambda^k_t(\nu_0, h_0)$. Since for any $t\in[0,T]$, $\log(\lambda^k_t(f_0)/\lambda^k_t(f))\leq 0$, we can use that
\begin{align*} 
KL(f_0, f) & \leq  \sum_k \Exz{ \int_0^T (\lambda^k_t(f) - \lambda^k_t(f_0)) dt } =  \sum_k \Exz{ \Lambda^k(f) - \Lambda^k(f_0) },
\end{align*}
with for any $1\leq k \leq K$,
$
     \Lambda^k(f) := \int_0^T \lambda^k_t(f) dt $, and $
     \Lambda^k(f_0) := \int_0^T \lambda^k_t(f_0) dt.$
Since for any $t$, $ \Tilde \lambda^k_t(\nu, h) \geq \Tilde \lambda^k_t(\nu_0, h_0)$, we have
\begin{align}\label{eq:bound_S_k_f}
   0\leq \Lambda^k(f) &- \Lambda^k(f_0)  = \int_0^T ((\Tilde \lambda^k_t(\nu, h))_+ - (\Tilde \lambda^k_t(\nu_0, h_0))_+ dt \leq  \int_0^T |\Tilde \lambda^k_t(\nu, h) - \Tilde \lambda^k_t(\nu_0, h_0)| dt \nonumber\\
    &\leq  T |\nu_k - \nu_k^0| + \sum_l \int_0^T \int_{t-A}^{t^-} |h_{l k} - h^0_{l k}|(t - s) dN^l_s dt \leq   T (\nu_k - \nu_k^0) + \sum_l \norm{h_{l k} - h^0_{l k}}_1 N^l[-A,T].
\end{align}
Consequently, we arrive at
\begin{align*}
KL(f_0, f) &\leq K T \epsilon_T (1 + \max_l \Exz{N^l[-A,T]}) + R_T \\
&\leq T\epsilon_T K (1 +  2 \max_l \mu_l^0 ) + o(T \epsilon_T^2) \lesssim T\epsilon_T.
\end{align*}

To refine this bound, we will assume that \eqref{condA3} holds. For $k \in [K]$ and $t \in [0,T]$, we define
$    p^k_t(f) =\lambda^k_t(f)/\Lambda^k(f) $
and similarly for $p^k_t(f_0)$. Using \eqref{eq:def_kl_rt}, we then have
\small
\begin{align}
KL(f_0, f) - R_T &=  \sum_k \Exz{\mathds{1}_{\eve} \left( \Lambda^k(f_0) \int_0^T \mathds{1}_{\lambda^k_t(f_0)>0} p^k_t(f_0) \log \left( \frac{ p^k_t(f_0) }{p^k_t(f)}\right)dt + \Lambda^k(f_0) \log \left(\frac{\Lambda^k(f_0)}{\Lambda^k(f)}\right) + (\Lambda^k(f)-\Lambda^k(f_0)) \right)} \nonumber \\
&\leq   \sum_k \Exz{\mathds{1}_{\eve} \left( \Lambda^k(f_0) \int_0^T \mathds{1}_{\lambda^k_t(f_0)>0} p^k_t(f_0) \log \left( \frac{ p^k_t(f_0) }{p^k_t(f)}\right)dt + \frac{( \Lambda^k(f_0 ) - \Lambda^k(f))^2 }{ \Lambda^k(f_0) } \right)} \label{eq:KL_decomp_2},
\end{align}
\normalsize
where in the last inequality, we have used the fact that $-\log x + x - 1 \leq (x-1)^2$ for $x \geq 1/2$, with $x = \frac{ \Lambda^k(f)}{ \Lambda^k(f_0)} \geq 1$. Moreover, from \eqref{eq:bound_S_k_f}, we have on $\eve$,
\begin{align*}
    \Lambda^k(f) - \Lambda^k(f_0) &\leq T \epsilon_T (1 + 2  \max_l \mu_l^0 ).
\end{align*}
 Besides, on $\eve$, using $A_2(T)$ defined in \eqref{def:A2} and noting that in this case, $r_k^0 = \nu_k^0, \forall k$,
\begin{align*}
    \Lambda^k(f_0) &\geq \int_{A_2(T)}  \lambda^k_t(f_0) dt \geq \sum_{j=1}^{J_T-1} \int_{\tau_{j}}^{U_j^{(1)}}  \lambda^k_t(f_0) dt = \nu_k^0 \sum_{j=1}^{J_T-1} (U_j^{(1)} - \tau_j) \\
    &\geq \frac{  \nu_k^0 T}{\mathbb{E}_0(\Delta \tau_1) \|\nu_0\|_1} \left(1 - 2c_\beta \sqrt{\frac{\log T}{T}}\right) \geq \frac{\nu_k^0 T}{2\mathbb{E}_0(\Delta \tau_1) \|\nu_0\|_1}. 
\end{align*}
Therefore,
\begin{align}
    \Lambda^k(f_0) \leq \Lambda^k(f) &\leq \Lambda^k(f_0) + T \epsilon_T(1 + 2 \max_l \mu_l^0 ) \nonumber \\
  &\leq \Lambda^k(f_0) + \frac{2\Lambda^k(f_0) (1 + 2A \max_l \mu_l^0 ) \mathbb{E}_0(\Delta \tau_1) \|\nu_0\|_1}{\nu_k^0} \epsilon_T \nonumber \\
    &\leq \Lambda^k(f_0) \left(1 + \frac{2(1 + 2A \max_l \mu_l^0 )\mathbb{E}_0(\Delta \tau_1) \|\nu_0\|_1}{\nu_k^0 } \epsilon_T \right) \leq 2\Lambda^k(f_0) \label{eq:maj_Sk} ,
\end{align}
for $T$ large enough. Besides, this implies that $p^k_t(f) = \frac{\lambda^k_t(f)}{\Lambda^k(f)} \geq \frac{\lambda^k_t(f_0)}{2\Lambda^k(f_0)} \geq p^k_t(f_0)/2$. Using again the inequality $-\log x + x -1 \leq (x - 1)^2$ with $x = \frac{p^k_t(f)}{p^k_t(f_0)} \geq \frac{1}{2}$ and the fact that $\int_0^T p^k_t(f) dt = \int_0^T p^k_t(f_0) dt = 1$, we have
\small
\begin{align*}
    &\int_0^T \mathds{1}_{\lambda^k_t(f_0)>0} p^k_t(f_0) \log \left( \frac{ p^k_t(f_0) }{p^k_t(f)}\right)dt = \int_0^T  p^k_t(f_0) \log \left( \frac{ p^k_t(f_0) }{p^k_t(f)}\right)dt + \int_0^T (p^k_t(f) - p^k_t(f_0)) dt \nonumber \\
    &= \int_0^T  p^k_t(f_0) \left(\log \left( \frac{ p^k_t(f_0) }{p^k_t(f)}\right) + \frac{p^k_t(f)}{p^k_t(f_0)} - 1 \right)dt \leq \int_0^T \mathds{1}_{\lambda^k_t(f_0)>0} \frac{ (p^k_t(f_0) -p^k_t(f))^2}{p^k_t(f_0)}dt \nonumber \\
     &\leq \frac{1}{\Lambda^k(f_0)} \int_0^T \mathds{1}_{\lambda^k_t(f_0)>0} \frac{ 2 \left(\lambda^k_t(f_0) - \lambda^k_t(f)\right)^2 + 2 \lambda^k_t(f)^2 \left(1 -\frac{\Lambda^k(f_0)}{\Lambda^k(f)} \right)^2}{\lambda^k_t(f_0)} dt \nonumber \\
    &\leq \frac{2}{\Lambda^k(f_0)} \left[\int_0^T \mathds{1}_{\lambda^k_t(f_0)>0} \frac{3\left(\lambda^k_t(f_0) - \lambda^k_t(f)\right)^2}{\lambda^k_t(f_0)} + 2\Lambda^k(f_0) \times \frac{(\Lambda^k(f) - \Lambda^k(f_0))^2}{\Lambda^k(f)^2}\right] \nonumber \\
    &\leq \frac{6}{\Lambda^k(f_0)} \int_0^T \mathds{1}_{\lambda^k_t(f_0)>0} \frac{ 2 \left(\lambda^k_t(f_0) - \lambda^k_t(f)\right)^2}{\lambda^k_t(f_0)} dt + 4 \frac{(\Lambda^k(f) - \Lambda^k(f_0))^2}{\Lambda^k(f_0)^2}. 
\end{align*}
\normalsize
In the previous inequalities, we have used $\Lambda^k(f_0) \leq \Lambda^k(f)$,
and for $T$ large enough, we have the following intermediate result:
\begin{align}
    KL(f_0,f) - R_T &\leq  \sum_k \Exz{ \mathds{1}_{\eve} \left( 6\int_{0}^T \mathds{1}_{\lambda^k_t(f_0)>0}\frac{ (\lambda^k_t(f_0) -  \lambda^k_t(f))^2}{  \lambda^k_t(f_0)}dt + 4  \frac{( \Lambda^k(f_0 ) - \Lambda^k(f))^2 }{ \Lambda^k(f_0) } \right)} .\label{KL-int}
\end{align}

Moreover, on $\eve$, using \eqref{eq:maj_Sk} 
\begin{align*}
    \Lambda^k(f_0) &= \int_0^T \left( \nu^0_k + \sum_l \int_{t-A}^{t^-} h^0_{lk}(t-s) dN^l_s\right)_+ dt  \leq T \nu_k^0 + \sum_l \|h^{0+}_{lk}\|_1 N^l[-A,T) \\
    &\leq T \nu_k^0 + \frac{3}{2}T \sum_l \|h^{0+}_{lk}\|_1 (\mu^0_l + \delta_T) \leq 2  T \left(\nu_k^0 + \sum_l \|h^{0+}_{lk}\|_1 \mu^0_l \right),
\end{align*}
for T large enough, since $\delta_T = \delta_0 \sqrt{\frac{\log T}{T}}$. Thus,
\begin{align*}
     \frac{( \Lambda^k(f_0 ) - \Lambda^k(f))^2 }{ \Lambda^k(f_0) } &\leq \Lambda^k(f_0) \left(\frac{2(1 + 2A \max_l \mu_l^0 )\mathbb{E}_0(\Delta \tau_1) \|\nu_0\|_1}{\nu_k^0} \right)^2 \epsilon_T^2 \leq = c_2^0 T \epsilon_T^2,
\end{align*}
with
\begin{align*}
c_2^0 =  8\left(\nu_k^0 + \sum_l \|h^{0+}_{lk}\|_1 \mu^0_l \right) \left(\frac{(1 + 2A \max_l \mu_l^0 )\mathbb{E}_0(\Delta \tau_1) \|\nu_0\|_1}{\nu_k^0} \right)^2 .
\end{align*}
Therefore, reporting into \eqref{KL-int} we have 
\begin{align*}
    KL(f_0,f) - R_T &\leq 6\sum_k \Exz{ \mathds{1}_{\eve}  \int_{0}^T \mathds{1}_{\lambda^k_t(f_0)>0}\frac{ (\lambda^k_t(f_0) -  \lambda^k_t(f))^2}{  \lambda^k_t(f_0)}dt} + 4 K c_2^0 T \epsilon_T^2.
\end{align*}
We now bound the first term on the RHS of the previous equation.
\small
\begin{align*}
\sum_k \Exz{ \mathds{1}_{\eve}  \int_{0}^T \mathds{1}_{\lambda^k_t(f_0)>0}\frac{ (\lambda^k_t(f_0) -  \lambda^k_t(f))^2}{  \lambda^k_t(f_0)}dt} \leq  \sum_k \mathbb{E}_{ 0} \left[ \mathds{1}_{\Omega_{\eve}} \sup_{t \in [0,T]} \mathds{1}_{\lambda^k_t(f_0)>0}(\lambda^k_t(f) - \lambda^k_t(f_0))^2  \int_{0}^{T} \frac{\mathds{1}_{\lambda^k_t(f_0)>0}}{\lambda^k_t(f_0)} dt \right].
\end{align*}
\normalsize
Moreover, for any $k \in [K]$ and $t \in [0,T]$, we have on $B_\infty(\epsilon_T)$
\begin{align} 
\mathds{1}_{\eve} \mathds{1}_{\lambda^k_t(f_0)>0}(\lambda^k_t(f) - \lambda^k_t(f_0))^2 dt   &\leq  2 (\nu_k - \nu_k^0)^2 + 2 K \max_l \|h_{lk} - h^0_{lk}\|_\infty^2  \sup_{t \in [0, T]}N^{l}[t-A,t)^2 \nonumber \\
&\leq 2 \e_T^2+ 2K  C_\beta^2 \log^2 T \e_T^2 \leq 4K  C_\beta^2 \log^2 T \e_T^2 \nonumber.
\end{align}
Consequently,
\begin{align*}
\sum_k \Exz{ \mathds{1}_{\eve}  \int_{0}^T \mathds{1}_{\lambda^k_t(f_0)>0}\frac{ (\lambda^k_t(f_0) -  \lambda^k_t(f))^2}{  \lambda^k_t(f_0)}dt} &\leq 4C_\beta^2K  (\log T)^2 T \e_T^2  \sum_k \mathbb{E}_{ 0} \left[ \frac{1}{T} \int_{0}^{T} \frac{\mathds{1}_{\lambda^k_t(f_0)>0}}{\lambda^k_t(f_0)} dt \right] \\
&=  4C_\beta^2 c_1^0 K  (\log T)^2 T \e_T^2, 
\end{align*}
using \eqref{condA3}, with
\begin{align*}
c_1^0 := \lim \sup \limits_{T \to \infty} \mathbb{E}_{ 0} \left[ \frac{1}{T} \int_{0}^{T} \frac{\mathds{1}_{\lambda^k_t(f_0)>0}}{\lambda^k_t(f_0)} dt \right] < +\infty.
\end{align*} 

Consequently, reporting into \eqref{KL-int}, we finally obtain
\begin{align*}
  KL(f_0,f) &\leq  4C_\beta^2 c_1^0 K L (\log T)^2 T \e_T^2 + 4K c_2^0 T \epsilon_T^2 + o(T\epsilon_T^2) \nonumber \\
  &\leq 8 K C_\beta^2 c_1^0(\log T)^2 T \e_T^2 = \kappa_2  (\log T)^2 T \e_T^2,
\end{align*}
with $\kappa_2 :=  8 K C_\beta^2 c_1^0$, which terminates the proof of this lemma.

\end{proof} 

\subsection{Deviations on the log likelihood ratio: Lemma \ref{lem:t_j} } \label{sec:lemTj}

The next lemma is a control under $\mathbb{P}_0$ over the centered sum of i.i.d. variables that are used to decompose the log-likelihood ratio in Lemma \ref{lem:KL}.

\begin{lemma}\label{lem:t_j}
Under the assumptions of Lemma \ref{lem:KLdecomp}, for $f \in B_\infty(\epsilon_T)$ and $j \geq 1$, let
\begin{align}\label{def_tj}
    T_j := \sum_k\int_{\tau_j}^{\tau_{j+1}}\log \left(\frac{\lambda^k_t(f_0)}{\lambda^k_t(f)} \right) dN_t^k - \int_{\tau_j}^{\tau_{j+1}} (\lambda^k_t(f_0) - \lambda^k_t(f)) dt.
\end{align}
Then it holds that $\Exz{T_j^2} \lesssim z_T / T$, with
\begin{align*}
    &z_T =  \begin{cases}
    T \epsilon_T^2 & \text{(under Assumption \ref{ass-psi}(i)} \\
    (\log T) T \epsilon_T^2 & \text{(under Assumption \ref{ass-psi}(ii))}\\
    (\log T)^2 T \epsilon_T^2 & \text{(ReLU link)}
    \end{cases}
\end{align*}
Moreover, if $\log^3 T = O(z_T)$,
\begin{align*}
    \Probz{\sum_{j=0}^{J_T-1} T_j  - \Exz{T_j} \geq z_T} = o(1).
\end{align*}
\end{lemma}

\begin{remark}
Under Assumption \ref{ass-psi}, for $f \in B_2(\epsilon_T,B)$, we also obtain similar results with $z_T =  (\log \log T)^2 T \e_T^2$.
\end{remark}

\begin{proof}

Firstly, using the fact that $\tau_1, \tau_2$ are stopping times, we have
\small
\begin{align}\label{control-variance}
     \Exz{T_1^2} &= \Exz{\left(\sum_k \int_{\tau_{1}}^{\tau_2}\log \left(\frac{\lambda^k_t(f_0)}{\lambda^k_t(f)} \right) dN_t^k - \int_{\tau_{1}}^{\tau_2}(\lambda^k_t(f_0) - \lambda^k_t(f)) dt \right)^2} \nonumber\\
     &\lesssim \sum_k \Exz{\left(\int_{\tau_{1}}^{\tau_2}\log \left(\frac{\lambda^k_t(f_0)}{\lambda^k_t(f)} \right) \lambda_t^k(f_0)dt + \int_{\tau_{1}}^{\tau_2}\log \left(\frac{\lambda^k_t(f_0)}{\lambda^k_t(f)} \right) (dN_t^k - \lambda_t^k(f_0)dt) - \int_{\tau_{1}}^{\tau_2}(\lambda^k_t(f_0) - \lambda^k_t(f)) dt \right)^2} \nonumber \\
      &\lesssim  \Exz{ \Delta \tau_1 \int_{\tau_{1}}^{\tau_2} \chi\left(\frac{\lambda^k_t(f)}{\lambda^k_t(f_0)} \right)^2 \lambda^k_t(f_0)^2dt } +  \Exz{\int_{\tau_{1}}^{\tau_2}\log^2  \left(\frac{\lambda^k_t(f_0)}{\lambda^k_t(f)} \right) \lambda_t^k(f_0) dt},
\end{align}
\normalsize
with $\chi(x) = - \log x + x - 1$. For any $x > 0$, we have $\chi^2(x) \leq 2 \log^2x + 2 (x-1)^2$. Now, if $f \in B_\infty(\epsilon_T)$, using that $\log^2x \leq (x-1)^2$ for $x = \lambda^k_t(f)/ \lambda^k_t(f_0)\geq 1$, we have $\chi\left(\frac{\lambda^k_t(f)}{\lambda^k_t(f_0)}\right)^2 \lambda^k_t(f_0)^2 \lesssim (\lambda^k_t(f_0) - \lambda^k_t(f))^2$ and $\log^2\left(\frac{\lambda^k_t(f)}{\lambda^k_t(f_0)}\right) \lambda^k_t(f_0) \lesssim \frac{(\lambda^k_t(f_0) - \lambda^k_t(f))^2}{\lambda^k_t(f_0)}$. Therefore, \eqref{control-variance} becomes
\begin{align}\label{control-variance-0}
    \Exz{T_1^2} &\lesssim \Exz{ \Delta \tau_1 \int_{\tau_{1}}^{\tau_2} (\lambda^k_t(f_0) - \lambda^k_t(f))^2 dt } + \Exz{\mathds{1}_{\eve^c}\int_{\tau_{1}}^{\tau_2}\log^2  \left(\frac{\lambda^k_t(f_0)}{\lambda^k_t(f)} \right) \lambda_t^k(f_0) dt} \\
    &+  \Exz{\mathds{1}_{\eve}\int_{\tau_{1}}^{\tau_2} \mathds{1}_{\lambda^k_t(f_0) > 0} \frac{(\lambda^k_t(f_0) - \lambda^k_t(f))^2}{\lambda^k_t(f_0)} dt}. \nonumber
\end{align}
\ds{With the ReLU link function, we can easily bound the third term on the RHS of \eqref{control-variance-0} using \eqref{condA3}:}
\begin{align*}
    \Exz{\mathds{1}_{\eve}\int_{\tau_{1}}^{\tau_2}\mathds{1}_{\lambda^k_t(f_0) > 0} \frac{(\lambda^k_t(f_0) - \lambda^k_t(f))^2}{\lambda^k_t(f_0)} dt} \lesssim \log^2 T \e_T^2 \Exz{\int_{\tau_{1}}^{\tau_2} \frac{\mathds{1}_{\lambda^k_t(f_0) > 0}}{\lambda^k_t(f_0)} dt} \lesssim \log^2 T \e_T^2.
\end{align*}
For the second term on the RHS of \eqref{control-variance-0}, using that $\log^2 (\lambda^k_t(f)) \lambda_t^k(f) \lesssim (\sup_{t} N[t-A,t))^3$ and similarly for $\lambda_t^k(f_0)$, we have 
\small
\begin{align*}
     \Exz{\mathds{1}_{\eve^c}\int_{\tau_{1}}^{\tau_2}\log^2  \left(\frac{\lambda^k_t(f_0)}{\lambda^k_t(f)} \right) \lambda_t^k(f_0) dt} &\lesssim  \Exz{\mathds{1}_{\eve^c}\int_{\tau_{1}}^{\tau_2}\log^2 (\lambda^k_t(f_0)) \lambda_t^k(f_0) dt} + \Exz{\mathds{1}_{\eve^c}\int_{\tau_{1}}^{\tau_2}\log^2 (\lambda^k_t(f)) \lambda_t^k(f) dt} \\
     &\lesssim \sqrt{\Exz{\mathds{1}_{\eve^c} (\sup_{t} N[t-A,t))^6}} \sqrt{\Exz{\Delta \tau_1^2}} \lesssim T^{-\beta/2} = o(\e_T^2),
\end{align*}
\normalsize
using Lemma \ref{lem:main_event}. For the first term on the RHS of \eqref{control-variance-0}, we have
\small
\begin{align*}
      \Exz{ \Delta \tau_1 \int_{\tau_1}^{\tau_2} (\lambda^k_t(f_0) - \lambda^k_t(f))^2 dt } &\lesssim \Exz{ \Delta \tau_1 \int_{\tau_1}^{\tau_2} (\Tilde \lambda^k_t(f_0) - \Tilde\lambda^k_t(f))^2 dt }  \\
      &\leq \Exz{\Delta \tau_1 \int_{\tau_1}^{\tau_{2}} (2 |\nu_k - \nu_k^0|^2 + 2K \sum_{l=1}^{K} \left(\int_{t-A}^t (h_{lk} - h_{lk}^0)(t-s) dN^l_s\right)^2 dt} \\
     &\leq 2 |\nu_k - \nu_k^0|^2 \Exz{\Delta \tau_1^2} + 2 K \sum_{l=1}^{K} \Exz{\Delta \tau_1 \int_{\tau_1}^{\tau_{2}} N^l(t-A, t) \int_{t-A}^t (h_{lk} - h_{lk}^0)^2 (t-s)dN^l_sdt} \\
    &= 2 |\nu_k - \nu_k^0|^2 \Exz{\Delta \tau_1^2} + 2 K \sum_{l=1}^{K} \norm{h_{lk} - h_{lk}^0}_2^2 \Exz{\Delta \tau_1 N^l[\tau_1,\tau_{2})^2} \\
    &\leq  2 |\nu_k - \nu_k^0|^2 \Exz{\Delta \tau_1} + 2 K \sum_{l=1}^{K} \norm{h_{lk} - h_{lk}^0}_2^2 \sqrt{\Exz{N^l[\tau_1, \tau_2)^4}} \sqrt{\Exz{\Delta \tau_1^2}}
    \lesssim \e_T^2.
\end{align*}
\normalsize
Thus, reporting into \eqref{control-variance-0}, we can conclude that if \eqref{condA3} holds, $\Exz{T_1^2} \lesssim \log^2 T \e_T^2$.

Under Assumption \ref{ass-psi}(i), if $f \in B_\infty(\epsilon_T)$, we can use the same computations. If $f \in B_2(\epsilon_T,B)$, for the first term on the RHS of \eqref{control-variance-0} and for the second term, we use instead that $\log^2 x \leq 4 \log^2(r_T^{-1}) (x-1)^2$ for $x \geq r_T$ with $x = \frac{\lambda^k_t(f_0)}{\lambda^k_t(f)} \gtrsim r_T:=(\log T)^{-1}$ and we obtain,
\begin{align*}
     & \Exz{\mathds{1}_{\eve} \int_{\tau_{1}}^{\tau_2}\log^2  \left(\frac{\lambda^k_t(f_0)}{\lambda^k_t(f)} \right) \lambda_t^k(f_0) dt} \lesssim (\log \log T)^2  \Exz{ \int_{\tau_1}^{\tau_2} (\lambda^k_t(f_0) - \lambda^k_t(f))^2 dt } \\
    & \quad \lesssim (\log \log T)^2  \Exz{ \int_{\tau_1}^{\tau_2} (\Tilde \lambda^k_t(\nu_0, h_0) - \Tilde\lambda^k_t(\nu, h))^2 dt } \lesssim (\log \log T)^2 \e_T^2,
\end{align*}
\ds{or, in the shifted ReLU model with unknown link (Case 2 of Proposition \ref{prop:relu}),}
\begin{align*}
     & \Exz{\mathds{1}_{\eve} \int_{\tau_{1}}^{\tau_2}\log^2  \left(\frac{\lambda^k_t(f_0, \theta_0)}{\lambda^k_t(f, \theta)} \right) \lambda_t^k(f_0, \theta_0) dt}  \\
     &\lesssim (\log \log T)^2  \left[ \Exz{\Delta \tau_1} (\theta_k - \theta_k^0)^2 + \Exz{ \int_{\tau_1}^{\tau_2} (\Tilde \lambda^k_t(\nu_0, h_0) - \Tilde\lambda^k_t(\nu, h))^2 dt }\right] \lesssim (\log \log T)^2 \e_T^2,
\end{align*}
using similar computations to the control of the first term of \eqref{control-variance-0}. The remaining term $\Exz{\mathds{1}_{\eve^c} \int_{\tau_{1}}^{\tau_2}\log^2  \left(\frac{\lambda^k_t(f_0)}{\lambda^k_t(f)} \right) \lambda_t^k(f_0) dt}$ is bounded as the second term of \eqref{control-variance-0}.

Finally, under Assumption \ref{ass-psi}(ii), using the fact that $\log \phi_k $ $L_1$-Lipschitz for any $k$, we have
\begin{align*}
   & \Exz{\int_{\tau_{1}}^{\tau_2}\log^2  \left(\frac{\lambda^k_t(f_0)}{\lambda^k_t(f)} \right) \lambda_t^k(f_0) dt } 
  \lesssim  \Exz{\int_{\tau_{1}}^{\tau_2} (\Tilde \lambda^k_t(\nu_0, h_0) - \Tilde \lambda^k_t(\nu , h))^2 \lambda_t^k(f_0)dt}\\
  & \quad \lesssim \log T \Exz{\int_{\tau_{1}}^{\tau_2} (\Tilde \lambda^k_t(\nu_0, h_0) - \Tilde \lambda^k_t(\nu, h))^2dt} + \Exz{\mathds{1}_{\eve^c}\int_{\tau_{1}}^{\tau_2} (\Tilde \lambda^k_t(f_0) - \Tilde \lambda^k_t(f))^2 \lambda_t^k(f_0)dt}\\
  &\qquad \lesssim (\log T) \e_T^2,
\end{align*}
and the first term of \eqref{control-variance} can be bounded similarly.

We now prove the second part of the lemma. We first note that
\begin{align}\label{eq:control_lt}
 \Probz{\sum_{j=0}^{J_T-1} T_j  - \Exz{T_j} \geq  z_T}  &\leq \sum_{J\in {\mathcal J}_T}{\mathbb P}_0\left( \sum_{j=0}^{J-1} T_j - \Exz{T_j} \geq  z_T\right) +{\mathbb P}_0\left(\eve^c\right) \nonumber\\
 &\leq T {\mathbb P}_0\left( \sum_{j=0}^{J-1} T_j - \Exz{T_j} \geq  z_T\right) + o(1).
\end{align}
Let $J \in \mathcal{J}_T$. Since the $\{T_j\}_{1 \leq j \leq J}$ are i.i.d.. random variables, we apply Fuk-Nagaev inequality (see Proposition \ref{fuk-nagaev}) to the sum of centered variables $T_j-{\mathbb E}[T_j]$ with $\lambda := z_T$ and $x := x_T$ with $x_T \to \infty$ a sequence determined later. We denote $v := J \Exz{T_1^2} \leq T \Exz{T_1^2} \lesssim z_T$.
Hence, we have $x \lambda/ v = x_T  z_T /v \gtrsim x_T$.
Since $x_T \to \infty$, 
\begin{align*}
&\left(1 + \frac{x \lambda}{v}\right) \log \left(1 + \frac{x \lambda}{v}\right) - \frac{x \lambda}{v} \geq \frac{x_T \lambda}{v}.
\end{align*}

From Fuk-Nagaev inequality, we have
\begin{align}\label{eq:fuk-nagaev}
{\mathbb P}_0\left( \sum_{j=1}^J(T_j-{\mathbb E}[T_j])\geq  z_T \right)
&\leq J \Probz{T_1 -{\mathbb E}[T_1] \geq x_T} + \exp \left \{- \frac{z_T}{x_T} \right \}.
\end{align}
We note that in the second term on the RHS of \eqref{eq:fuk-nagaev}, if $\frac{z_T}{x_T} \geq x_0 \log T$ with $x_0 > 0$ large enough, then $\exp \left \{- \frac{ z_T}{x_T} \right \} = o(\frac{1}{T})$. Since by assumption, $\log T = o(T \epsilon_T^2)$, then we can choose  $x_T = x_0' \frac{z_T}{\log T} \to \infty$ with $x_0' > 0$ a constant small enough. For the first term on the RHS of \eqref{eq:fuk-nagaev}, let us consider $j \in [J]$. From \eqref{def_tj}, we have 
\begin{align*}
    T_1 \leq \sum_k \left \{ \int_{\tau_1}^{\tau_{2}} |\lambda^k_t(f) - \lambda^k_t(f_0)| dt +  \int_{[\tau_1, \tau_{2})} |\log \lambda^k_{t}(f) - \log \lambda^k_{t}(f_0)| dN_t^k\right \}.
\end{align*}
Using the first part of the lemma and Cauchy-Schwarz inequality, we have that $\Exz{T_1} \leq \sqrt{\frac{z_T}{T}}\leq x_T$ since $x_T \gtrsim z_T / \log T$ and $\log^3T = O(z_T)$. Therefore,
\small
\begin{align*}
    \Probz{T_1 - \Exz{T_1} \geq x_T} \leq \Probz{\eve \cap \left \{\int_{\tau_1}^{\tau_{2}} |\lambda^k_t(f) - \lambda^k_t(f_0)| dt + \int_{[\tau_1, \tau_{2})} |\log \lambda^k_{t}(f) - \log \lambda^k_{t}(f_0)|  \geq x_T \right \}} + \Probz{\eve^c}.
\end{align*}
\normalsize
On the one hand, on $\eve$, \ds{under Assumption \ref{ass-psi}(i)}, using that $|\log x - \log y| \leq \frac{|x-y|}{y} $ for $x \geq y$,

\begin{align*}
\int_{[\tau_1, \tau_{2})}  |\log \lambda^k_{t}(f) -  \log  \lambda^k_{t}(f_0)|dN_t^k &\leq \frac{2}{ \ds{\min_k \inf_x \phi_k(x)}}\int_{[\tau_1, \tau_{2})} |\log \lambda^k_{t}(f) -  \log  \lambda^k_{t}(f_0)|dN_t^k\\
        & \leq \frac{2L N[\tau_1, \tau_{2}) }{\ds{\min_k \inf_x \phi_k(x)}}   |\nu_k - \nu_k^0| + \frac{2L}{\ds{\min_k \inf_x \phi_k(x)}}  \int_{[\tau_1, \tau_{2})}\int_{[\tau_1, \tau_{2})} |h_{lk} - h_{lk}^0|(t - s) dN_t^kdN_s^k\\
        &\leq \frac{4L}{\ds{\min_k \inf_x \phi_k(x)}} (\epsilon_T N[\tau_1, \tau_{2}) +  N[\tau_1, \tau_{2})^2 \norm{h_{lk} - h_{lk}^0}_\infty) \leq 3 LB N[\tau_1, \tau_{2})^2,
\end{align*}
for $T$ large enough. \ds{In Case 2 of Proposition \ref{prop:relu}, we similarly have}
\begin{align*}
\int_{[\tau_1, \tau_{2})}  |\log \lambda^k_{t}(f) -  \log  \lambda^k_{t}(f_0)|dN_t^k &\leq \frac{2}{\theta_k^0}\int_{[\tau_1, \tau_{2})} |\log \lambda^k_{t}(f) -  \log  \lambda^k_{t}(f_0)|dN_t^k\\
        & \leq \frac{2N[\tau_1, \tau_{2}) }{\theta_k^0}   (|\theta_k - \theta_k^0| + |\nu_k - \nu_k^0|) + \frac{2}{\theta_k^0}  \int_{[\tau_1, \tau_{2})}\int_{[\tau_1, \tau_{2})} |h_{lk} - h_{lk}^0|(t - s) dN_t^kdN_s^k\\
        &\leq \frac{4}{\theta_k^0}\epsilon_T N[\tau_1, \tau_{2}) + 2  N[\tau_1, \tau_{2})^2 \norm{h_{lk} - h_{lk}^0}_\infty \leq 3 B N[\tau_1, \tau_{2})^2,
\end{align*}

\ds{Under Assumption \ref{ass-psi}(ii), $\log \phi_k$ is $L_1$-Lipschitz, therefore,}
\begin{align*}
         \sum_{t_i \in [\tau_1, \tau_{2})} |\log \lambda^k_{t_i}(f) - \log \lambda^k_{t_i}(f_0)| &\leq L_1\sum_{t_i \in [\tau_1, \tau_{2})} | \Tilde \lambda^k_{t_i}(\nu, h) - \Tilde \lambda^k_{t_i}(\nu_0, h_0)| \leq L_1 B N[\tau_1, \tau_{2})^2.
\end{align*}

\ds{With the ReLU link function, we directly have that}
$
T_1 \leq \sum_k \int_{\tau_1}^{\tau_{2}} (\lambda^k_t(f) - \lambda^k_t(f_0)) dt.
$

\ds{In  Case 2 of Proposition \ref{prop:relu},
\begin{align*}
\int_{\tau_1}^{\tau_{2}} & |\lambda^k_t(f, \theta) - \lambda^k_t(f_0, \theta_0)| dt
\leq   |\theta_k^0 - \theta_k| \Delta \tau_1 + \int_{\tau_1}^{\tau_{2}} (\Tilde \lambda^k_t(\nu, h) - \Tilde\lambda^k_t(\nu_0, h_0)) dt \\
&\leq (|\theta_k^0 - \theta_k| + |\nu_k - \nu_k^0|) \Delta \tau_1  +  \sum_l \norm{h_{lk} - h_{lk}^0}_1 N^l[\tau_1, \tau_2) \leq [2\Delta \tau_1 + N[\tau_1, \tau_{2})]\epsilon_T .
\end{align*}
and in all other cases,
\begin{align*}
\int_{\tau_1}^{\tau_{2}} & |\lambda^k_t(f) - \lambda^k_t(f_0)| dt
\leq  L \int_{\tau_1}^{\tau_{2}} (\Tilde \lambda^k_t(\nu, h) - \Tilde\lambda^k_t(\nu_0, h_0)) dt \\
&\leq  L |\nu_k - \nu_k^0|) \Delta \tau_1  + L \sum_l \norm{h_{lk} - h_{lk}^0}_1 N^l[\tau_1, \tau_2) \leq L [2\Delta \tau_1 + N[\tau_1, \tau_{2})]\epsilon_T .
\end{align*}}
Consequently,
\begin{align*}
    T_1 \leq K C[2\Delta \tau_1 + N[\tau_1, \tau_2)]\epsilon_T  + 3 K C B N[\tau_1, \tau_{2})^2 \leq 4 K C B N[\tau_1, \tau_{2})^2,
\end{align*}
\ds{with $C=\max(1,L,L_1)$ or $C=\max(1,L)$ depending on the assumptions on the link functions, and}
\begin{align*}
\Probz{T_1 - {\mathbb E}_0[T_1] \geq 2x_T} \leq \Probz{N[\tau_1, \tau_{2})^2 > \frac{x_T}{2KCB}}.
\end{align*}

Using Lemma \ref{lem:excursions}, we have for some $s > 0$
\begin{align*}
     \Probz{N[\tau_1, \tau_{2})^2 > \frac{x_T}{2KCB}} \leq \Exz{e^{s N[\tau_1, \tau_2)}} e^{-s \sqrt{x_T /(2KCB)}} = o(T^{-2}),
\end{align*}
if $x_T \geq x_0'' \log^2 T$ for some $x_0''>0$ large enough, implying that $z_T \geq z_0 \log^3 T$ for some $z_0>0$. Finally, reporting into \eqref{eq:control_lt}, we can conclude that 
\begin{align*}
    {\mathbb P}_0\left( \sum_{j=1}^{J_T}(T_j-{\mathbb E}[T_j])\geq  z_T \right)
&\leq T^2 \Probz{T_1 -{\mathbb E}[T_1] \geq x_T} + T \Probz{\eve^c} + T \exp \left \{- \frac{z_T}{x_T} \right \} + o(1) = o(1).
\end{align*}

\end{proof}




\section{Proofs of our results related identifiability and to the regenerative properties of nonlinear Hawkes models}\label{app:proof_technical_lemmas}


\subsection{Proofs of Proposition \ref{lem:identif_param}, Proposition \ref{lem:identif_param2} and Lemma  \ref{lem:hawkes_exciting}}\label{sec:proof:lem:existence}

In this section, we prove our two propositions on the model identifiability, i.e., Propositions \ref{lem:identif_param} and \ref{lem:identif_param2}, as well as Lemma \ref{lem:hawkes_exciting} in the mutually-exciting Hawkes model. We recall the results in each case.

\begin{proposition}[Proposition \ref{lem:identif_param}]
 Let $N$ be a  nonlinear Hawkes process  as defined in \eqref{def:NLintensity} with link functions $(\phi_k)_k$ and parameter $f = (\nu, h)$ satisfying the conditions of Lemma~\ref{lem:existence} and Assumption~\ref{ass:identif_f}. If $N'$ is a Hawkes processes with the same link functions $(\phi_k)_k$ and parameter $f' = (\nu', h')$,  then
 \begin{align*}
 N \overset{\mathcal{L}}{=} N' \implies \nu = \nu' \quad \text{and} \quad h = h'.
\end{align*}
\end{proposition}

\begin{proof}
Let  $f' = (\nu', h')$ and $N' \sim \mathbb P_{f'}$. We recall that $N \sim \mathbb P_f$  and $N \overset{\mathcal{L}}{=} N'$ is equivalent to $ \lambda_t^l(f) = \lambda_t^l(f')$ for all $t>0$ and $l \in [K]$. Let $\tau_1$ be the first renewal time of the process $N$, as defined in Section \ref{sec:lem_excursions}. From the proof of Lemma \ref{lem:excursions}, with $U_1^{(1)}$ the time of the first event after $\tau_1$ and $V^{(1)} \in [K]$ the index of the component associated with this event, we have that $U_1^{(1)} \sim Exp(\norm{r_f}_1) \indep  V_1^{(1)}$ \ds{with $r_f = (r_1^f,\dots, r_K^f)$ and $r_k^f = \phi_k(\nu_k), \forall k$, } and
\begin{align*}
V_1^{(1)} \sim Multi\left(1; \frac{r_1^f}{\norm{r_f}_1}, \dots, \frac{r_K^f}{\norm{r_f}_1}\right).
\end{align*}
Therefore we can conclude that
\begin{align}\label{eq:identif_r}
    N \overset{\mathcal{L}}{=} N' \implies r_f = r_{f'} \iff  \phi_k(\nu_k) = \phi_k(\nu_k'), \: \forall k \in [K].
\end{align}
Since for all $k$, $\nu_k \in I_k$ defined by Assumption \ref{ass:identif_f} (ii), then $\nu_k'  =\phi_k^{-1}( \phi_k(\nu_k)) $ and since $\phi_k$ is monotone non-decreasing, we obtain $\nu_k = \nu_k', \forall k$.  

Moreover, for each $k \in [K]$, we define the event $\Omega_k$ as
\begin{align*}
     \Omega_k &= \left\{ \max_{k' \neq k} N^{k'}[\tau_1, \tau_2) = 0, N^k[\tau_1, \tau_1 + A] = 1,  N^k[\tau_1+A,\tau_2) = 0 \right\}.
\end{align*} 

On $\Omega_k$, for $t \in [\tau_1, \tau_2)\cap  [U_1^{(1)}, U_1^{(1)}+A]$ and $l \in [K],$
$\lambda_t^l(f) =  \phi_l(\nu_l + h_{kl}(t-U_1^{(1)}))$
and similarly for $\lambda_t^l(f')$. Then, for any $s = t-U_1^{(1)} \in [0,A]$, $\lambda_{U_1^{(1)}+s}^l(f) =   \phi_l(\nu_l + h_{kl}(s)) =  \phi_l(\nu_l + h_{kl}'(s))$. Consequently,  using that $\phi_l$ is injective on $I_l$,  $h_{kl} = h'_{kl}$ for all $1 \leq k,l \leq K$ which concludes the proof of this proposition. 

\end{proof}

\begin{proposition}{Proposition \ref{lem:identif_param2}}
Let $N$ be a Hawkes process with parameter $f = (\nu, h)$ and link function $\phi_k(x; \theta_k) = \theta_k + \psi_k(x)$ with $\theta_k\geq 0$ for any $k\in [K]$ satisfying the conditions of Lemma~\ref{lem:existence} and Assumption~\ref{ass:identif_f}. We also assume that for all $k\in [K]$, $\lim \limits_{x \to -\infty} \psi_k(x) = 0$ and
 \begin{equation}\label{ass:identif_theta}
  \exists l\in [K], \, x_1<x_2, \quad \text{such that } \, h_{lk}^{-}(x)>0, \quad \forall x \in [x_1, x_2].
  \end{equation} 
 Then if $N'$ is a Hawkes processes with link functions  $\phi_k(x; \theta_k') = \theta_k' + \psi_k(x)$, $\theta_k'\geq 0$ and parameter $f' = (\nu', h')$,
\begin{align*}
 N \overset{\mathcal{L}}{=} N' \implies \nu = \nu', \quad  h = h', \quad \text{and} \quad \theta = \theta', \quad \theta = (\theta_k)_{k=1}^K, \quad \theta'=(\theta_k')_{k=1}^K.
\end{align*}
Besides, in this case we have $\Probz{\inf \limits_{t \geq 0}  \lambda^k_t(f, \theta) = \theta_k} = 1.$
\end{proposition}

\begin{proof}
Using the proof of Proposition \ref{lem:identif_param},  we first obtain that $ \phi_k(\nu_k) = \phi_k(\nu_k')$, therefore
\begin{align*}
\theta_k + \psi_k(\nu_k)  = \theta_k' + \psi_k(\nu_k'), \: \forall k \in [K].
\end{align*}
Secondly, we also have that $\theta_l + \psi_k(\nu_l + h_{kl}(s)) =  \theta_l' + \psi_k(\nu_l' + h_{kl}'(s))$ for any $s \in [0,A]$ and all $1 \leq k,l \leq K$. 

We first prove that $\theta = \theta'$ and from the latter we can deduce that $\nu = \nu'$ and finally  that $h = h'$ by the injectivity of $\psi_k$ on $I_k$, for any $k$. The proof of the identification of $\theta$ relies on the construction of a specific excursion for each $k \in [K]$ in which there exists $t > 0$ such that  $\lambda^k_t(f) \in [\theta_k, \theta_k + \epsilon]$ for any $\epsilon > 0$. From that, we will deduce  that $N \overset{\mathcal{L}}{=} N' \implies \theta = \theta'$. 

Let $k \in [K]$ and consider $l \in [K]$ such that $h_{lk}$ satisfies  Assumption~\ref{ass:identif_theta}. We first note that
$$\lambda_t^k(f) = \theta_k + \psi_k(\Tilde \lambda_t^k(\nu, h)) \geq \theta_k.$$
Thus, we directly have that $\theta_k \leq \inf \limits_{t > 0} \lambda^k_t(f), $ a.s. Let $\epsilon > 0$. Using Assumption \ref{ass:identif_theta} (i), $\exists M > 0, \forall x \leq M, \: \psi_k(x) \leq \epsilon$. Using now Assumption~\ref{ass:identif_theta} (ii), let $l \in [K]$ and $x_1 < x_2$ such that $[x_1, x_2] \subset B_0 := \{x \in [0,A], h_{lk}(x) \leq - c_\ast \}$. Define $n_1 = \min \{n \in \N; \: nc_\ast > \nu_k^0 - M\}$, $\delta' = (x_2 - x_1) / 3$,
and we consider an excursion,  which we write $[0, \tau]$, and which satisfies
\begin{align*}
\mathcal{E} = \{N[0, \delta'] = N^l[0, \delta'] = n_1, \: N[\delta', \delta' +A] = 0\}.
\end{align*}
In other words the events only occur on the $l$-th component of the Hawkes process and only on $[0, \delta']$. 
Since $\psi_k$ is Lipschitz and injective on $I_k = (\nu_k - \max_l \norm{h_{lk}^-}_\infty - \varepsilon, \nu_k + \max_l \norm{h_{lk}^+}_\infty + \varepsilon)$, it holds that $\Probf{\mathcal{E}} > 0 $. For $t \in [x_1 + \delta', x_2]$, $\forall i \in [n_1]$, we have $x_1 \leq t - t_i \leq x_2$, and therefore, 
\begin{align*}
\Tilde \lambda_t^k(\nu, h) = \nu_k + \sum_{i \in [n_1]} h_{lk}(t - t_i) \leq \nu_k - n_1 c_\ast \leq M.
\end{align*}
Consequently, for $t \in [x_1 + \delta', x_2]$, $\lambda_t^k(f_0) = \theta_k + \psi_k(\Tilde \lambda_t^k(\nu, h)) \leq \theta_k + \epsilon$. We can then conclude that
\begin{align*}
\Probz{\exists t \geq 0, \: \lambda_{t}^k(f) \in [\theta_k, \theta_k + \epsilon] } >0,
\end{align*}
for any $\epsilon > 0$. This is equivalent to
\begin{align*}
 \theta_k =  \inf_{\omega \in \Omega} \inf_{t \in [0,\tau]} \lambda_t^k(f)(\omega),
\end{align*}
where $\lambda_t^k(f_0)(\omega)$ denotes the value of the random process $(\lambda_t(f_0))_t$ at time $t$.


Now, if $N'$ is a Hawkes process with parameter $f' \in \mathcal{F}$ and link functions $\phi_k = \theta_k' + \psi_k, \: k\in [K]$ such that $N \overset{\mathcal{L}}{=} N'$, then for any $t \geq 0$ and $k$ such $\lambda_t^k(f) \leq \theta_k + \epsilon$, we have $\theta_k' \leq \lambda_t^k(f') \leq \theta_k + \epsilon$ and thus, $\theta_k \geq \theta_k'$. Inversely, if   $\lambda_t^k(f') \leq \theta_k' + \epsilon$ then $\theta_k \leq \theta_k'$ and finally we can conclude that $\theta = \theta'$.

\end{proof}

\begin{lemma}[Lemma \ref{lem:hawkes_exciting}]
Let $N$ be a Hawkes process with parameter $f = (\nu, h)$ and link functions $\phi_k(x; \theta_k) = \theta_k + (x)_+, \: \theta_k \geq 0, \: k\in [K]$  satisfying Assumption \ref{ass:identif_f}, and let $k \in [K]$. If $\forall l \in [K], h_{l k} \geq 0$, then for any $\theta_k' \geq 0$ such that $\theta_k + \nu_k - \theta_k' > 0$, let $N'$ be the Hawkes process driven by the same underlying Poisson process $Q$ as $N$ (see Lemma \ref{lem:stoc_domination}) with parameter $f' = (\nu', h')$ and link functions $\phi_k(x; \theta_k') = \theta_k' + (x)_+, k\in [K]$ with $\nu' = (\nu_1, \dots, \nu_k+ \theta_k - \theta_k', \dots, \nu_K) \neq \nu, \: h' = h$, and $\theta' = (\theta_1, \dots, \theta_k', \dots, \theta_K) \neq \theta$. Then for any $t \geq 0$, $\lambda^k_t(f, \theta) = \lambda^k_t(f', \theta')$, and therefore $N \overset{\mathcal{L}}{=} N'$.
\end{lemma}


\begin{proof}
We consider $k \in [K]$ such that $\forall l \in [K], \: h_{lk} \geq 0 $. For any $t \geq 0$, we have
\begin{align*}
\Tilde{\lambda}^k_t(\nu, h) = \nu_k + \sum_l \int_{t-A}^{t^-} h_{lk}(t - s) dN_s^l \geq \nu_k> 0, 
\end{align*}
and thus
$
\lambda^k_t(f) =
\theta_k + (\Tilde{\lambda}^k_t(\nu, h))_+ = \theta_k + \Tilde{\lambda}^k_t(\nu, h).
$
Moreover, for any $t \geq 0$, we have
\begin{align*}
\Tilde{\lambda}^k_t(\nu', h') =  \nu_k + \theta_k - \theta_k' + \sum_l \int_{t-A}^{t^-} h_{lk}(t - s) dN_s^l \geq \nu_k + \theta_k - \theta_k' > 0,
\end{align*}
and
\begin{align*}
\lambda^k_t(f') &= \theta_k' + (\Tilde{\lambda}^k_t(\nu', h'))_+ =  \theta_k' + \Tilde{\lambda}^k_t(\nu', h') \\
&= \theta_k' + \nu_k + \theta_k - \theta_k' + \sum_l \int_{t-A}^{t^-} h_{lk}(t - s) dN_s^l = \theta_k + \Tilde{\lambda}^k_t(\nu, h) = \lambda^k_t(f).
\end{align*}
Therefore, we obtain that
$
    N =^{\mathcal{L}} N'.
$

\end{proof}

\subsection{Proofs of Lemmas \ref{lem:tau} and \ref{lem:conc_J}}\label{sec:proof_lem_tau}

In this section, we prove our lemmas related to the renewal properties of the nonlinear Hawles processes, in particular the existence of exponential moments for the generic renewal time $\Delta \tau_1$, and a concentration inequality on $J_T$. the nunber of excursions in the interval of observation $[0,T]$.

\begin{lemma}[Lemma \ref{lem:tau}]
\ds{Under the assumptions of Lemma \ref{lem:excursions}, the random variables $\Delta \tau_1$ and $N[\tau_1, \tau_2)$ admit exponential moments. More precisely, under condition \textbf{(C1bis)}, with  $m = \norm{S^+}<1$, we have 
$$
\forall s < \min (\norm{r_f}_1,\gamma/A), \quad \Exf{e^{s \Delta \tau_1}}  \leq \frac{1+m}{2m}, \quad \text{and}\quad \Exf{e^{sN[\tau_{1},\tau_2)}} < + \infty, \quad \gamma = \frac{1- m}{2\sqrt{K}} \log \left(\frac{1 + m}{2m}\right).
$$
Under condition \textbf{(C2)}, we have
$
\forall s < \min_k \Lambda_k), \: \Exf{e^{s \Delta \tau_1}}  \leq \frac{\norm{\Lambda}_1^2 }{ (\min_k \Lambda_k - s)^2}$ and $ \Exf{e^{sN[\tau_{1},\tau_2)}} < + \infty$.
In particular, this implies that $\Exf{N[\tau_{1},\tau_2) + N[\tau_1,\tau_2)^2} < + \infty.$}
\end{lemma}

\begin{proof}
\ds{Under condition \textbf{(C1bis)},} similarly to \cite{costa18_supplement}, we use the fact that the multivariate Hawkes model is stochastically dominated by a mutually-exciting process $N^+$ with parameter $f^+ = (\nu, (h^{+}_{lk})_{l,k})$, and driven by the same Poisson process as $N$ (see Lemma \ref{lem:stoc_domination}). For $N^+$, the stopping time $\Delta \tau_1^+$ corresponds to the length of the busy period of a $M^K/G^K/\infty$ queue (see Lemma \ref{lem:queue} which is a multi-type extension of existing results).

More precisely, since $N^+$ is mutually-exciting, the cluster representation is available \cite{reynaud06}, with the ancestor arrival process being a Poisson Point Process equal to the baseline rate $r_f$, defined in \eqref{eq:r_0}. For this process, the duration of the clusters then corresponds to the generic service time $H$ of a queue with an infinite number of servers. In the multidimensional case, this duration may depend on the type of the ancestor (or "customer" in the queuing framework) but the generic service time can be written in a compact form, and is independent of the arrival process
\begin{align*}
H = \sum_{k=1}^K \delta_k H^k,
\end{align*}
where $\delta_k = 1$ if and only if the ancestor is of type $k \in [K]$.  To apply Lemma \ref{lem:queue}, we only need to check that the cluster length $H^k, \: k \in [K]$ has exponential moments. This can be proved using results from \cite{donnet18:supplement}.

 For the process $N^+$, let $W^k$ be the number of events in a cluster with an ancestor of type $k$. By definition of a cluster of events, $H^k \leq A W^k$. Moreover, from Lemma 5 in the Supplementary Materials of \cite{donnet18:supplement}, for a mutually-exciting Hawkes process and for any $t \leq \frac{1 - \norm{S^+}_1}{2 \sqrt{K}} \log \left(\frac{1 + \norm{S^+}}{2 \norm{S^+}}\right)$ and $k \in [K]$,
\begin{align*}
\Exf{e^{tW^k}} \leq \frac{1 + \norm{S^+}}{2\norm{S^+}}.
\end{align*}
Therefore, we define
$\gamma = (1 - \norm{S^+}) \left[ \log \left(1 + \norm{S^+}\right) - \log (2 \norm{S^+}) \right]/(2 \sqrt{K})
$
and $s_0=  \frac{1 + \norm{S^+}}{2\norm{S^+}}$. For all $0 < t \leq \gamma$, we thus have
$
\Exf{e^{t H^k/A}} \leq s_0. 
$
Consequently, we deduce that the service time $H^k$ has exponential tails, i.e.,
$    \Probf{H^k \geq t} \leq s_0 e^{-t \gamma / A}. $
We can now use the fact that a.s. $\mathcal{T}_1 = \Delta \tau_1^+$ (cf Lemma \ref{lem:stoc_domination}), so that for any $s < \norm{r_f}_1 \wedge \gamma/A$, we have $\Exf{e^{s \Delta \tau_1^+}} < \infty$.
Finally using the second part of Lemma \ref{lem:stoc_domination}, we have that $\Probf{\Delta \tau_1 \leq \Delta \tau_1^+} = 1$ and, using Lemma~\ref{lem:queue}, we arrive at
$
 \forall s < \norm{r_f}_1 \wedge \gamma/A $, $ \Exf{e^{s \Delta \tau_1}} < \infty.$

\ds{Under condition \textbf{(C2)}, we use the fact that the process $N$ is dominated by a $K$-dimensional homogeneous Poisson point process $N_P = (N_P^1, \dots, N_P^K)$ with rate $\Lambda = (\Lambda_1, \dots, \Lambda_K)$.  For the latter process, the generic service time of an ancestor of type $k$, $H_k$, is exponentially distributed with mean $\Lambda_k$, i.e.,
\begin{align*}
\Probf{H_k > t} = e^{-\Lambda_k t}, \quad t \geq 0.
\end{align*}
Therefore, denoting $\Delta \tau_1^P$, the corresponding generic stopping time of $N^P$ (with the same definition as  in Lemma \ref{lem:excursions} for the Hawkes process (note that the Poisson point process is a renewal process), we have
\begin{align*}
\Probf{\Delta \tau_1^P > t} \leq \Exf{N^P[0,t]} e^{- \min_k \Lambda_k (t-A)} = \norm{\Lambda}_1 t e^{- \min_k \Lambda_k (t- A)}.
\end{align*}
Therefore, for any $s < \min_k \Lambda_k$, 
\begin{align*}
\Exf{e^{s \Delta \tau_1}} \leq \Exf{e^{s\Delta \tau_1^P}} &= \int_0^{+\infty} s e^{st} \Probf{\Delta \tau_1^P \geq t}dt \leq \norm{\Lambda}_1^2 e^{\min_k \Lambda_k A} \int_0^{+\infty} t e^{t(s - \min_k \Lambda_k)}dt \\
&\leq \norm{\Lambda}_1^2 \int_0^{+\infty} \frac{ e^{t(s - \min_k \Lambda_k)}}{\min_k \Lambda_k - s}dt =   \frac{\norm{\Lambda}_1^2 }{ (\min_k \Lambda_k - s)^2 }.
\end{align*}}

We now consider the number of events in a excursion  $N[\tau_1,\tau_2)$. \ds{Under condition \textbf{(C1bis)}}, From Lemma \ref{lem:stoc_domination}, we can also deduce that $\Exf{N[\tau_1,\tau_2)} \leq \Exf{N^+[\tau_1^+,\tau_2^+)}$. We once again use the cluster representation available for $N^+$. For the latter, let $n^\tau$ be the number of ancestors arriving in $[\tau_1^+, \tau_2^+)$ and $W_i$ be the number of points in the cluster with ancestor $i$ for $1 \leq i \leq n_\tau$. We denote $(NP_t)_t$ the homogeneous Poisson process of intensity $\norm{r_f}_1$ corresponding to the arrival times of the ancestors. By definition of $\tau_1^+, \tau_2^+$, we have
\begin{align}\label{eq:bound_N_tau}
N^+[\tau_1^+,\tau_2^+) &= \sum_{i=1}^{n_\tau} W_i.
\end{align}
Let $\gamma > s > 0$ and $u < \norm{r_f}_1 \wedge \gamma/A$. With $t = \Exf{e^{sW_1}} \leq s_0$, since the $W_i$'s are independent conditionally on $n_\tau$,
\begin{align*}
    \Exf{e^{sN[\tau_1,\tau_2)}} &\leq \Exf{e^{s \sum_{i=1}^{n_\tau} W_i}} =  \Exf{ \Exf{e^{s \sum_{i=1}^{n_\tau} W_i}|n_\tau}}
    =  \Exf{ \Exz{e^{s W_1}}^{n_\tau}} = \Exf{\sum_{l = A}^{+\infty} e^{s n_\tau}\mathds{1}_{\Delta \tau_1 \in [l,l+1)}}  \\
&    \leq \sum_{l = A}^{+\infty} \Exf{ e^{s NP[\tau_1, \tau_1 + l + 1)}\mathds{1}_{\Delta \tau_1 \geq l}} 
\leq \sum_{l = A}^{+\infty} \sqrt{\Exf{ e^{2 sNP[\tau_1,\tau_1 + l + 1)}}} \sqrt{\Probf{\Delta \tau_1 > l}} \\
 & \leq \sqrt{\Exf{e^{u \Delta \tau_1}}} \sum_{l = A}^{+\infty} \sqrt{\Exf{ e^{2sNP[\tau_1,\tau_1 + l + 1)}}}e^{-ul/2} =  \sqrt{\Exf{e^{u \Delta \tau_1}}} \sum_{l = A}^{+\infty}    e^{\norm{r_f}_1 (l+1)(e^{2s} - 1)/2} e^{-ul/2},
\end{align*}
since $NP$ is a homogoneous Poisson process with rate $\norm{r_f}_1$. Moreover, since for any $\alpha \in (0,1)$, $\Exf{e^{\alpha sW_1}} = (\Exf{e^{\alpha sW_1}}^{1/\alpha})^{\alpha} \leq \Exf{e^{s W_1}}^{\alpha} \leq s_0^\alpha$, with $t' = \Exf{e^{\alpha sW_1}}$,  we have that $\norm{r_f}_1 (l+1)(e^{2t'} - 1) < u/2$ for $\alpha$ small enough. Consequently,
\begin{align*}
    \Exf{e^{sN[\tau_1,\tau_2)}} &\leq  \sqrt{\Exf{e^{u \Delta \tau_1}}} \sum_{l = A}^{+\infty} e^{-ul/4} = \frac{ \sqrt{\Exz{e^{u \Delta \tau_1}}} }{1- e^{-u/4}} < \infty.
\end{align*}
In particular, this implies that $\Exf{N[\tau_1,\tau_2)} + \Exf{N[\tau_1,\tau_2)^2} < \infty$. \ds{Under condition \textbf{(C2)}, the dominating process $N^+$ is a homogeneous Poisson process with intensity $\Lambda = (\Lambda_1, \dots, \Lambda_K)$ and the previous computations remain valid by replacing $r_f$ by $\Lambda$ and with $W_i = 1$ for any $i \in [n_\tau]$ (since in this case each cluster only contains the ``ancestor" event).}

\end{proof}

\begin{lemma}[Lemma \ref{lem:conc_J}]
Under the assumptions of Lemma \ref{lem:excursions}, for any $\beta > 0$, there exists a constant $c_\beta > 0$ such that
$
\Probf{J_T \notin [J_{T,\beta, 1}, J_{T,\beta, 2}]} \leq  T^{-\beta},
$
with $J_T$ defined in \eqref{def:J_T} and
\begin{align*}
&J_{T,\beta, 1} = \left \lfloor \frac{T}{\Exf{\Delta \tau_1}} \left(1 - c_\beta \sqrt{\frac{\log T}{T}}\right)\right \rfloor, \quad J_{T,\beta, 2} = \left \lfloor \frac{T}{\Exf{\Delta \tau_1}} \left(1 + c_\beta \sqrt{\frac{\log T}{T}}\right)\right \rfloor.
\end{align*}
\end{lemma}

\begin{proof}
Let $c_\beta > 0$ and for $2 \leq j \leq J_T$, $B_j = \tau_{j} - \tau_{j-1} - \Exf{\Delta \tau_1}$. Using Lemma \ref{lem:excursions}, the random variables $\{B_j\}_{2 \leq j\leq J_T}$ are i.i.d.. By definition of $J_{T,\beta,2}$, we have
\begin{align*}
\frac{T}{\Exf{\Delta \tau_1}} \left(1 + c_\beta \sqrt{\frac{\log T}{T}}\right) - 1 < J_{T,\beta,2} \leq \frac{T}{\Exf{\Delta \tau_1}} \left(1 + c_\beta \sqrt{\frac{\log T}{T}}\right).
\end{align*}
Therefore,
\begin{align*}
\Probf{J_T\geq J_{T,\beta,2}} &= \Probz{\tau_{J_{T,\beta,2}} \leq T} = \Probf{\tau_0 + \sum_{j=1}^{J_{T,\beta,2}} B_{j} \leq T - J_{T,\beta,2}\Exf{\Delta \tau_1}} \\
&= \Probf{\sum_{j=1}^{J_{T,\beta,2}} B_{j} \leq T - J_{T,\beta,2} \Exf{\Delta \tau_1}} \leq \Probf{\sum_{j=1}^{J_{T,\beta,2}} B_{j} \leq T - T \left(1 + c_\beta \sqrt{\frac{\log T}{T}}\right) + \Exf{\Delta \tau_1} } \\
&= \Probf{\sum_{j=0}^{J_{T,\beta,2}} B_{j} \leq - c_\beta \sqrt{T \log T} + \Exf{\Delta \tau_1} } \leq \Probf{\sum_{j=1}^{J_{T,\beta,2}} B_{j} \leq - \frac{c_\beta \sqrt{T \log T}}{2}}.
\end{align*}
We can now apply the Bernstein's inequality. Using Lemma~\ref{lem:tau},  there exists $\alpha > 0$, such that $ \Exf{e^{\alpha \Delta \tau_1}} < + \infty$. Since
\begin{align*}
\Exf{e^{\alpha \Delta \tau_1}} = \sum_{k = 1}^{+\infty} \frac{\alpha^k \Exf{(\Delta \tau_1)^k}}{k!},
\end{align*}
we therefore have that
\begin{align*}
\Exf{(\Delta \tau_1)^k} \leq \frac{k!}{\alpha^k} \Exf{e^{\alpha \Delta \tau_1}} = \frac{1}{2} k! \alpha^{-k+2} \times 2 \frac{\Exf{e^{\alpha \Delta \tau_1}}}{\alpha^2}.
\end{align*}
In particular, $\Exf{(\Delta \tau_1)^2} \leq 2 \frac{\Exz{e^{\alpha \Delta \tau_1}}}{\alpha^2} =: v$. Consequently, with $b := 1/\alpha$, we obtain
$
\Exf{(\Delta \tau_1)^k} \leq \frac{1}{2} k! b^{k-2} v, 
$
and therefore,
\begin{align*}
\Probf{J_T\geq J_{T,\beta,2}} &\leq \exp \left \{ \frac{- c_\beta^2  T \log T }{8 (\sigma^2 + \frac{c_\beta}{2} \sqrt{T \log T} b)} \right \},
\end{align*}
with 
\begin{align*}
\sigma^2 &= \sum_{j=1}^{J_{T,\beta, 2}} \mathbb{V}_{f}(B_j) = J_{T,\beta, 2} \mathbb{V}_{f}(\Delta \tau_1) 
\leq  T \left(1 + c_\beta \sqrt{\frac{\log T}{T}}\right) \frac{\Exf{\Delta \tau_1^2} }{\Exf{\Delta \tau_1} } \leq 2 T \frac{\Exf{\Delta \tau_1^2} }{\Exf{\Delta \tau_1} }, 
\end{align*}
for $T$ large enough. Therefore, $\sigma^2 + \frac{c_\beta}{2} \sqrt{T \log T} b \leq 4 T\frac{\Exf{\Delta \tau_1^2} }{\Exf{\Delta \tau_1} }$ and
\begin{align*}
\Probf{J_T\geq J_{T,\beta,2}} &\leq \exp \left \{ \frac{- c_\beta^2 \log T \Exf{\Delta \tau_1}}{32\Exf{\Delta \tau_1^2}} \right \} = o(T^{-\beta}),
\end{align*}
for any $\beta > 0$, if $c_\beta > 0$ is chosen large enough. Consequently, with probability greater than $1 - \frac{1}{2}T^{-\beta}$, we have that $J_T \leq  \frac{T}{\Exf{\Delta \tau_1}}  \left(1 + c_\beta \sqrt{\frac{\log T}{T}}\right) $. Similarly, we obtain that
\begin{align*}
\Probf{J_T\leq J_{T,\beta, 1}} &\leq \Probf{\sum_{j=1}^{J_{T,\beta,1}} B_{j} \geq  c_\beta \sqrt{T \log T}}  \leq \exp \left \{ \frac{-c_\beta^2 T \log T }{2 (\sigma^2 + c_\beta  \sqrt{T \log T} b)} \right \} \\
&\leq \exp \left \{ \frac{- c_\beta^2 \log T\Exf{\Delta \tau_1}}{4\Exf{\Delta \tau_1^2}} \right \} = o(T^{-\beta}). 
\end{align*}
Finally, we conclude that with probability greater than $1 - T^{-\beta}$, $J_{T,\beta,1} \leq J_T \leq J_{T,\beta,2}$.

\end{proof}

\section{Proof of lemmas \ref{lem:main_event} and \ref{lem:ef}}\label{app:proof_bay_lemmas}

\subsection{Proof of Lemma \ref{lem:main_event}}\label{app:proof_lem_event}

\begin{lemma}[Lemma \ref{lem:main_event}]
Let $Q > 0$. We consider $\Tilde{\Omega}_T$ defined in \eqref{def:Omega} in Section \ref{sec:proof_conc_g}. For any $\beta > 0$, we can choose $C_\beta$ and $c_\beta$ in the definition of $\Tilde{\Omega}_T$ such that
\begin{align*}
	\probz{\Tilde{\Omega}_T^c} \leq T^{-\beta}.
\end{align*}
Moreover, for any $1 \leq q \leq Q$,
$
\Exz{\mathds{1}_{\eve^c} \max_l \sup \limits_{t \in [0,T]} \left(N^l[t-A,t)\right)^q} \leq 2 T^{-\beta/2}. 
$
\ds{Finally, the previous results hold when replacing $\eve$ by $\eve' = \eve \cap \Omega_A$ with $\Omega_A$ defined in Section \ref{sec:proof_conc_f} for the model with shifted ReLU link and unknown shift. }
\end{lemma}

\begin{proof}
Let $\beta > 0$. From the definition of $\eve$, we have that
\begin{align}\label{eq:decomp_event}
\probz{\Tilde{\Omega}_T^c} \leq \probz{\Omega_N^c} + 3 \probz{\Omega_J^c} + \probz{\Omega_J \cap \Omega_U^c}.   
\end{align}
For the second term on the RHS of \eqref{eq:decomp_event}, we can directly use Lemma \ref{lem:conc_J}, and we obtain $\probz{\Omega_J^c} \leq \frac{1}{12} T^{-\beta}$ for $c_\beta$ large enough. For the first term on the RHS of \eqref{eq:decomp_event}, we use the same strategy as in \cite{donnet18:supplement}. Firstly we have 
\begin{align}\label{eq:decomp_ev_N}
\probz{\Omega_N^c} \leq \Probz{\max \limits_{k \in [K]} \sup \limits_{t \in [0,T]} N^k[t-A,t) > C_\beta \log T} + \sum_{k = 1}^K \Probz{\left| \frac{N^k[0,T]}{T} - \mu_k^0 \right| \geq \delta_T}.
\end{align}
For the first term on the RHS of \eqref{eq:decomp_ev_N}, we use the coupling with the process $N^{+}$, i.e., the Hawkes process with parameter $f_0^+ = (\nu_0, h_0^+)$ driven by the same Poisson process. Then for any $l \in [K]$, $\sup \limits_{t \in [0,T]} N^l[t-A,t) \leq \sup \limits_{t \in [0,T]} (N^+)^l[t-A,t)$ and consequently,
\begin{align*}
\Probz{\max \limits_{k \in [K]} \sup \limits_{t \in [0,T]} N^k[t-A,t) > C_\beta \log T} \leq \Probz{\max \limits_{k \in [K]} \sup \limits_{t \in [0,T]} (N^+)^k[t-A,t) > C_\beta \log T}.
\end{align*}
Using Lemma 2 from \cite{donnet18:supplement}, we obtain that for any $\beta > 0$, there exists $C_{\beta}> 0$ such that
\begin{align*}
\Probz{\max \limits_{k \in [K]} \sup \limits_{t \in [0,T]} (N^+)^k[t-A,t) > C_{\beta} \log T} \leq \frac 1 4 T^{-\beta}. 
\end{align*}
For the second term on the RHS of \eqref{eq:decomp_ev_N}, we use the same arguments as in the proof of Lemma~3 in \cite{donnet18:supplement}. For $k \in [K]$, we have
\small
\begin{align}\label{eq:decomp_ev_N2}
 \Probz{\left| \frac{N^k[0,T]}{T} - \mu_k^0 \right| \geq \delta_T} \leq  \Probz{\left| N^k[0,T] - \int_0^T \lambda_t^k(f_0) \right| \geq T\delta_T / 2} + \Probz{\left| \int_0^T \lambda_t^k(f_0) - \mu_k^0 T \right| \geq T\delta_T / 2}.
\end{align}
\normalsize
For the second term on the RHS of \eqref{eq:decomp_ev_N2}, we can use Corollary 1.1 from \cite{costa18_supplement}. We have that
$
 \lambda_t^k(f_0) = Z(S_tN),
$
with 
\begin{align*}
  Z(N) =  \lambda_0^k(f_0) =  \phi_k\left(\nu_k^0 + \sum_l \int_{-A}^{0^-} h_{lk}(t - s)dN_s^l\right) \leq  L b (1 + N[-A,0)),  
\end{align*}
with $b = \max (\nu_k^0, \max_l \norm{h_{lk}^{0+}}_\infty)$ and for $t \in \R$, $S_t: \mathcal{N}(\R) \to S_tN = N(.+t)$ the shift operator by $t$ units of time. Applying Corollary 1.1 of \cite{costa18_supplement} with $f = Z$, $\pi_Af = \Exz{\lambda_0^k(f_0)} = \mu_k^0, \: \varepsilon = \delta_T/2$ and $\eta = \frac{1}{4} T^{-\beta}$, we obtain that for $\delta_0$ large enough,

\begin{align*}
\Probz{\left| \int_0^T \lambda_t^k(f_0) - \mu_k^0 T \right| \geq T\delta_T / 2} \leq \frac{1}{4} T^{-\beta}.
\end{align*}
For the first term on the RHS of \eqref{eq:decomp_ev_N2}, we use the computations of the proof Lemma 3 in the Supplementary Materials of \cite{donnet18:supplement} and obtain
\begin{align*}
\Probz{\left| N^k[0,T] - \int_0^T \lambda_t^k(f_0) \right| \geq T\delta_T / 2} \leq \frac{1}{4} T^{-\beta},
\end{align*}
for $\delta_0$ large enough.

For the third term on the RHS of \eqref{eq:decomp_event}, we denote $X_j = U_j^{(1)} - \tau_j$ for $1 \leq j \geq J_T-1$. We recall that the $X_j$'s are i.i.d. and follow an exponential law with rate $ \norm{r_0}_1$ under $\mathbb{P}_0$ and $\Exz{X_j} = \frac{1}{\norm{r_0}_1}$. We thus have
\begin{align*}
&\probz{\Omega_J \cap \Omega_U^c} \leq \Probz{\Omega_J \cap \left\{\sum_{j=1}^{J_T-1} X_j \leq 
     \frac{T}{\mathbb{E}_0[\Delta \tau_1] \|r_0\|_1} \left(1 - 2c_\beta \sqrt{\frac{\log T}{T}}\right) \right \}}\\
     &\quad \leq \Probz{\Omega_J \cap \left\{\sum_{j=1}^{J_T-1} X_j - \frac{J_T - 1}{\norm{r_0}_1}  \leq \frac{T}{\mathbb{E}_0[\Delta \tau_1] \|r_0\|_1} \left(1 - 2c_\beta \sqrt{\frac{\log T}{T}} - 1 + c_\beta \sqrt{\frac{\log T}{T}}\right) \right \}}\\
     &\quad = \Probz{\Omega_J \cap \left\{\sum_{j=1}^{J_T-1} X_j - \frac{J_T - 1}{\norm{r_0}_1}  \leq - \frac{c_\beta\sqrt{T \log T}}{\mathbb{E}_0[\Delta \tau_1] \|r_0\|_1} \right  \}}\leq \sum_{J \in \mathcal{J_T}} \Probz{\sum_{j=1}^{J-1} X_j - \frac{J - 1}{\norm{r_0}_1}  \leq - \frac{c_\beta\sqrt{T \log T}}{\mathbb{E}_0[\Delta \tau_1] \|r_0\|_1} },
\end{align*}
where in the first inequality we have used the fact that on $\Omega_J$,
\begin{align*}
J_T - 1 \geq \frac{T}{\mathbb{E}_0[\Delta \tau_1] }\left(1 - c_\beta \sqrt{\frac{\log T}{T}}\right).
\end{align*}
We apply the Bernstein's inequality using that for any $k \geq 1$,
$\Exz{X_1^k} \leq  k! (\norm{r_0}_1)^{-k+2} \Exz{X_1^2}/2$.
Therefore, since $\Exz{X_1^2}=\norm{r_0}_1^{-2}$, we obtain
\begin{align*}
 \Probz{\sum_{j=1}^{J-1} X_j - \frac{J - 1}{\norm{r_0}_1}  \leq -\frac{c_\beta\sqrt{T \log T}}{\mathbb{E}_0[\Delta \tau_1] \|r_0\|_1} } &\leq
\exp -\left \{ \frac{c_\beta^2 \log T}{\Exz{\Delta \tau_1}^2(1 + \frac{c_\beta \sqrt{\log T}}{\mathbb{E}_0[\Delta \tau_1]\sqrt{T}})} \right \} \\
&\leq \exp -\left \{\frac{c_\beta^2 \log T}{2\Exz{\Delta \tau_1}} \right \} \leq \frac{1}{4} T^{-\beta},
\end{align*}
for $c_\beta > 0$ large enough. Finally, reporting into \eqref{eq:decomp_event} we can conclude that for $C_\beta, c_\beta, \delta_0$ large enough,
\begin{align*}
\Probz{\eve^c} \leq T^{-\beta}.
\end{align*}
For the second part of the lemma, we can use the exact same arguments as in the proof of Lemma 2 in  \cite{donnet18:supplement} to obtain the result.

For the case of shifted ReLU link function with unknown shift, we similarly have that
\begin{align}
\probz{\Tilde{\Omega}_T'^c} \leq \probz{\Omega_N^c} + 3 \probz{\Omega_J^c} + \probz{\Omega_J \cap \Omega_U^c} + \Probz{\Omega_J \cap \Omega_A^c},
\end{align}
and therefore it only remains to bound the last term on the RHS of the previous inequality. Using Assumption~\ref{ass:identif_theta} (ii), let $0<x_1 < x_2$ and $c_\star$ such that $[x_1, x_2] \subset B_0 = \{x \in [0,A], h_{lk}^{0}(x) \leq - c_\ast \}$, $n_1 = \min \{n \in \N; nc_\ast > \nu_k^0\}$, $\delta' = (x_2 - x_1) / 3$. We denote $\mathcal{E}_0$ the set of indices satisfying
\begin{align*}
\mathcal{E}_0 = \{j \in [J_T]; \: N[\tau_j, \tau_j + \delta'] = N^l[\tau_j, \tau_j + \delta'] = n_1, \: N[\tau_j + \delta', \tau_{j+1}) = 0\}.
\end{align*}
Since $\forall t \in [\tau_j + x_1 + \delta', \tau_j + x_2]$, $\Tilde \lambda_t^k(f) < 0$, then $|A^k(f_0)| \geq \frac{2(x_2 - x_1)}{3} |\mathcal{E}_0|$ and, with $p_0 = \Probz{j \in \mathcal{E}_0}$,
\begin{align*}
    \Probz{|A^k(f_0)| < z_0 T} \leq \Probz{|\mathcal{E}_0| < \frac{3z_0}{2(x_2 - x_1)} T} \leq \Probz{|\mathcal{E}_0| < p_0 T /2},
\end{align*}
if $z_0 < 2 p_0(x_2 - x_1)/3$. Consequently, applying Hoeffding's inequality with $Y_j = \mathds{1}_{j \in \mathcal{E}_0} \overset{i.i.d.}{\sim} \mathcal{B}(p_0)$ for $j \in [J_T]$ with $J_T \geq 2T /3\Exz{\Delta \tau_1}$, we obtain
\begin{align*}
\Probz{ |\mathcal{E}_0| < \frac{p_0 T}{2}} &\leq \Probz{\sum_{j=1}^{2T/3\Exz{\Delta \tau_1}} Y_j < \frac{p_0 T}{2}} \lesssim e^{-\frac{ T p_0^2}{6\Exz{\Delta \tau_1}} } \leq \frac{1}{4} T^{-\beta}.
\end{align*}
Consequently, $\Probz{\Omega_J \cap \Omega_A^c} =  o(T^{-\beta})$, which terminates the proof of this lemma.

\end{proof}

\subsection{Proof of Lemma \ref{lem:ef}}

\begin{lemma}[Lemma \ref{lem:ef}]
For  $f \in \mathcal{F}_T$ and $l \in [K]$, let 
\begin{equation*}
    Z_{1l} = \int_{\tau_1}^{\xi_1} |\lambda^l_t(f) - \lambda^l_t(f_0)|dt, 
\end{equation*}
where $\xi_1$ is defined in \eqref{def:A2} in Section \ref{sec:proof_conc_g}. 
Under the assumptions of Theorem \ref{thm:conc_g} and Case 1 of Proposition \ref{prop:relu}, for $M_T \to \infty$ such that $M_T > M \sqrt{\kappa_T}$ with $M>0$ and for any $f \in \mathcal{F}_T$ such that $\norm{\nu-\nu_0}_1 \leq \max(\norm{\nu_0}_1, \Tilde{C})$ with $\Tilde{C}>0$,
there exists $l \in [K]$ such that on $\Tilde{\Omega}_{T}$,
    \begin{equation*}
        \Exf{Z_{1l}} \geq C(f_0)  \norm{f - f_0}_1, 
    \end{equation*}
with $C(f_0) > 0$ a constant that depends only on $f_0$ and $\phi=(\phi_k)_k$.

Similarly, under the assumptions of  Case 2 of Proposition \ref{prop:relu}, for $f \in \mathcal{F}_T$ and $\theta \in \Theta$, let $r_0 = (r_k^0)_k, \: r_f = (r_k^f)_k$ with $r_k^0 = \phi_k(\nu_k^0) = \theta_k^0 + \nu_k^0, \: r_k^f = \phi_k(\nu_k) = \theta_k + \nu_k, \:  \forall k$. If $\norm{r_f - r_0}_1 \leq \max(\norm{r_0}, \Tilde{C}')$ with $\Tilde{C}'>0$, then there exists $l \in [K]$ such that on $\Tilde{\Omega}_{T}$,
    \begin{equation}\label{eq:ef}
        \Exf{Z_{1l}} \geq C'(f_0) (\|r_f - r_0\|_1 + \norm{h - h_0}_1), \quad C'(f_0) > 0.
    \end{equation}
\end{lemma}

\begin{proof}
\ds{In this proof, we will show that \eqref{eq:ef} holds for all the models satisfying the assumptions of Theorem \ref{thm:conc_g} and  Proposition \ref{prop:relu}, with $r_k^0 = \phi_k(\nu_k^0)$ and $r_k^f = \phi_k(\nu_k)$ for all $k$. Then, excluding Case 2, we use the fact that for any $k$, $\phi_k^{-1}$ is fully known and $L'$-Lipshitz on $J_k = \phi_k(I_k)$ with $I_k$ defined in Assumption \ref{ass-psi} (which also holds for the ReLU link function by Assumption \ref{ass:identif_f}), to show that 
\begin{align*}
	\|r_f - r_0\|_1 + \norm{h - h_0}_1 &\geq 1/L'\|\nu - \nu_0\|_1 + \norm{h - h_0}_1 \\
	&\geq \min(1, 1/L')(\|\nu - \nu_0\|_1 + \norm{h - h_0}_1) =  \min(1, 1/L') \norm{f-f_0}_1.
\end{align*}}

\ds{The proof of \eqref{eq:ef}  is inspired by the proof of Lemma 4 in the supplementary material of \cite{donnet18:supplement}. The following computations are valid in all our estimation scenarios. We recall that for any $k$, $r_k^f = \nu_k$ for the ReLU link (Case 1 of Proposition \ref{prop:relu}) and  $r_k^f = \theta_k + \nu_k$ for the shifted ReLU link (Case 2 of Proposition \ref{prop:relu}).}

Let $A > x > 0$ and $\eta>0$ such that
\begin{equation}\label{eq:parameta}
    0 < \frac{(A + x)^2 \eta K^2}{1 - \eta K} < \frac 1 2 \qquad \text{ and } \qquad \eta \leq \frac{\min_{l} r_l^0}{2 C_0'},
\end{equation}
with $C_0'$ such that $ \|r_f - r_0\|_1 + \norm{h - h_0}_1 \leq C_0'$. Assume that for any $1 \leq l' \leq K$, $|r^f_{l'} - r_{l'}^0| \leq \eta (\|r_f - r_0\|_1 + \norm{h - h_0}_1) $ and let $l \in [K]$ such that
$
    \sum_k \|h_{kl} - h_{kl}^0\|_1 = \max_{l'} \sum_k \|h_{kl'} - h_{kl'}^0\|_1.
$

Then we have
\begin{equation}\label{eq:norm1}
  \|r_f - r_0\|_1 + \norm{h - h_0}_1\leq \left(\frac{\eta K^2}{1 - \eta K} + K \right)  \sum_k \|h_{kl} - h_{kl}^0\|_1.
\end{equation}
For each $k \in [K]$, we define the event $\Omega_k$ as
\begin{align*}
     \Omega_k &= \left\{ \max_{k' \neq k} N^{k'}[\tau_1, \tau_2) = 0, \: N^k[\tau_1, \tau_1 + x] = 0, \: N^k[\tau_1 + x, \tau_1 + x+A] = 1,  \: N^k[\tau_1 + x+A, \tau_2) = 0 \right\}.
\end{align*} 
On $\Omega_k$, we have $\xi_1 = U_1^{(1)}+ A$ and thus,
\begin{align*}
     \Exf{Z_{1l}} &\geq \sum_k\Exf{\mathds{1}_{\Omega_k} \int_{\tau_1}^{A+U_1^{(1)}} |\lambda_t^l(f)-\lambda_t^l(f_0)| dt}.
\end{align*}
Let $\mathbb{Q}$ be the point process measure of a homogeneous Poisson process with unit intensity on $\R^+$ and equal to the null measure on $[-A,0)$. Then
\begin{align*}
     \Exf{Z_{1l}} &\geq \sum_k \mathbb{E}_{\mathbb{Q}}\left[\int_{\tau_1}^{U_1^{(1)}+ A} \mathcal{L}_t(f) \mathds{1}_{\Omega_k}  |\lambda_t^l(f)-\lambda_t^l(f_0)|  \right]dt,
\end{align*}
with $\mathcal{L}_t(f)$ the likelihood process given by
\begin{align*}
    \mathcal{L}_t(f) = \exp \left(Kt - \sum_k \int_{\tau_1}^t \lambda_u^k(f) du + \sum_k \int_{\tau_1}^t \log (\lambda_u^k(f)) dN^k_u  \right).
\end{align*}
For $t \in [\tau_1, U_1^{(1)}+ A)$, since on $\Omega_k$, $\tau_1 + x \leq U_1^{(1)}\leq \tau_1 + A+x$, we have 
\begin{align*}
    \mathcal{L}_t(f)
     &\geq e^{Kt} \lambda^k_{U_1^{(1)}}(f) \exp \left \{ - \sum_{k'} \int_{\tau_1}^t \phi_{k'}(\Tilde \lambda_u^{k'}(f)) du \right \}.
\end{align*}
\ds{Under condition \textbf{(C2)}, since $\phi_{k'} \leq \Lambda_{k'}, \forall k'$, with $\Lambda = (\Lambda_1, \dots, \Lambda_K)$, we directly have that
\begin{align*}
 \mathcal{L}_t(f)
     &\geq e^{Kt} \lambda^k_{U_1^{(1)}}(f) e^{-\norm{\Lambda}_1} \geq r_k^f e^{-\norm{\Lambda}_1},
\end{align*}
since at $\lambda_{U_1^{(1)}}^k = r_k^f = \phi_k(\nu_k)$.}

\ds{Under condition \textbf{(C1bis)}, using that $\phi_k$ is $L$-Lipschitz, we have } 
\begin{align*}
\mathcal{L}_t(f)
      &\geq e^{- \sum_{k'}  \phi_{k'}(0)(A + U_1^{(1)} - \tau_1)}\lambda^k_{U_1^{(1)}}(f) \exp \left \{ -  \sum_{k'} \int_{\tau_1}^{A + U_1^{(1)}}( \phi_{k'}(\Tilde \lambda_u^{k'}(f)) - \phi_{k'}(0)) du \right \} \\
     &\geq e^{- \sum_{k'}  \phi_{k'}(0)(A + U_1^{(1)} - \tau_1)} \lambda^k_{U_1^{(1)}}(f) \exp \left \{ - L \sum_{k'} \left( (A + U_1^{(1)} - \tau_1) \nu_{k'} + \int_{U_1^{(1)}}^{A+U_1^{(1)}} h_{kk'}(u- U_1^{(1)}) du \right) \right \} \\
        &\geq e^{- \sum_{k'}  \phi_{k'}(0)(2A + x)} \lambda^k_{U_1^{(1)}}(f)\exp \left \{ - L \sum_{k'} \left( (2A + x) \nu_{k'}  + \int_{U_1^{(1)}}^{A+U_1^{(1)}} h^+_{kk'}(u- U_1^{(1)}) du \right) \right \} \\
    &\geq  e^{-  \sum_{k'}  \phi_{k'}(0)(2A + x)} r_k^f\exp \left \{- L \sum_{k'}\left((2A + x)\nu_{k'} + \|h^+_{kk'}\|_1 \right) \right \}.
\end{align*}
\normalsize
Moreover, since $\|S^+\|_1 < 1$, then $\forall (k,k') \in [K]^2, \: \|h^+_{kk'}\|_1 < 1$. Thus, we obtain 
\begin{align*}
    \mathcal{L}_t(f) &\geq  e^{-  \sum_{k'}  \phi_{k'}(0)(2A + x)}  r_k^fe^{- L K - L(2A + x) \sum_{k'}\nu_{k'} } \\
    &\geq \frac{  e^{-  \sum_{k'}  \phi_{k'}(0)(2A + x)} r_k^0}{2} e^{- L K - 6AL\max(\Tilde C,\norm{\nu_0}_1)} =: C.
\end{align*}
In the last inequality, we have used our assumption $\norm{\nu-\nu_0}_1 \leq \max(\norm{\nu_0}_1, \Tilde C)$ which implies that
\begin{align*}
\sum_{k'} \nu_{k'} \leq 2\max(\norm{\nu_0}_1, \Tilde{C}). 
\end{align*}


Moreover, we have that
\begin{align*}
     \Exf{Z_{1l}}  &\geq C \sum_k  \mathbb{E}_{\mathbb{Q}}\left[ \mathds{1}_{\Omega_k} \int_{U_1^{(1)}}^{U_1^{(1)} + A}\left| \phi_l(\Tilde \lambda_t^l(f)) - \phi_l( \Tilde \lambda_t^l(f_0))|    \right|dt\right] \\
     &\geq \frac{C}{L'}\sum_k  \mathbb{E}_{\mathbb{Q}}\left[ \mathds{1}_{\Omega_k} \int_{U_1^{(1)}}^{U_1^{(1)} + A}\left|(\nu_l - \nu_l^0) + (h_{kl}-h_{kl}^0)(t-U_1^{(1)})  \right|dt\right],
\end{align*}
\ds{in all models except Case 2. In fact, in the latter case,  we obtain}
\begin{align*}
     \Exf{Z_{1l}} &\geq C \sum_k  \mathbb{E}_{\mathbb{Q}}\left[ \mathds{1}_{\Omega_k} \int_{U_1^{(1)}}^{U_1^{(1)} + A}\left|(\theta_l + \nu_l - \theta_l^0 - \nu_l^0) + (h_{kl}-h_{kl}^0)(t-U_1^{(1)})  \right|dt\right]  \\
     &= C \sum_k  \mathbb{E}_{\mathbb{Q}}\left[ \mathds{1}_{\Omega_k} \int_{U_1^{(1)}}^{U_1^{(1)} + A}\left|(r_l^f - r_l^0) + (h_{kl}-h_{kl}^0)(t-U_1^{(1)})  \right|dt\right].
\end{align*}
On the one hand,
\begin{align*}
    \mathbb{E}_{\mathbb{Q}}\left[  \mathds{1}_{\Omega_k} \int_{U_1^{(1)}}^{U_1^{(1)} + A} |\nu_l - \nu_l^0| dt \right] &=  A |\nu_l - \nu_l^0|\mathbb{Q}(\Omega_k) \leq A L' |\phi_l(\nu_l) - \phi_l(\nu_l^0)|\mathbb{Q}(\Omega_k) = A L' |r_l^f - r_l^0| \mathbb{Q}(\Omega_k) \\
    &\leq A L' \frac{\eta K^2}{1 - \eta K} \sum_{k'} \|h_{k'l}-h_{k'l}^0\|_1,
\end{align*}
and in Case 2 we have
\begin{align*}
    \mathbb{E}_{\mathbb{Q}}\left[  \mathds{1}_{\Omega_k} \int_{U_1^{(1)}}^{U_1^{(1)} + A} |r_l^f - r_l^0| dt \right] &= A  |r_l - r_l^0| \mathbb{Q}(\Omega_k) \leq A \frac{\eta K^2}{1 - \eta K} \sum_{k'} \|h_{k'l}-h_{k'l}^0\|_1.
\end{align*}
On the other hand, by definition of $\mathbb{Q}$, $N^k[\tau_1, \tau_1 + x+A] \sim \text{Poisson}(x+A)$. Consequently, with $U$ a random variable with uniform distribution on $[\tau_1 + x, \tau_1 + x+A]$, we obtain
\begin{align*}
    \mathbb{E}_{\mathbb{Q}} &\left[ \mathds{1}_{\Omega_k} \int_{U_1^{(1)}}^{U_1^{(1)} + A} \left|(h_{kl}-h_{kl}^0)(t-U_1^{(1)}) \right| dt \right] = \mathbb{Q}(\Omega_k) \mathbb{E}\left[\int_{U}^{U+A} |(h_{kl}-h_{kl}^0)(t-U)| dt \right]  \\
    &= \frac{\mathbb{Q}(\Omega_k)}{A} \int_{\tau_1 + x}^{\tau_1 + A+x} \left[\int_s^{A+s} |h_{kl}-h_{kl}^0|(t-s) dt\right] ds   \geq \mathbb{Q}(\Omega_k) \|h_{kl}-h_{kl}^0\|_1.
   \end{align*}
Moreover, we have
\begin{align*}
\mathbb{Q}(\Omega_k) &\geq \mathbb{Q}(\max_{k' \neq k} N^{k'}[\tau_1, \tau_1 + x+2A] = 0, N^{k}[\tau_1, \tau_1 + x] = 0, N^{k}[\tau_1 + x, \tau_1 + x+A] = 1) \\
&= \mathbb{Q}(\max_{k' \neq k} N^{k'}[\tau_1, \tau_1 + x+2A] = 0) \mathbb{Q}(N^{k}[\tau_1, \tau_1 + x] = 0) \mathbb{Q}(N^{k}[ \tau_1 + x, \tau_1 + x+A] = 1)\\
&= e^{-(K-1)(x+2A)} \times e^{-x} \times A e^{-A} := C'.
\end{align*}
Using  \eqref{eq:parameta} together with \eqref{eq:norm1}, we obtain
\begin{align*}
     \Exf{Z_{1l}} &\geq \frac{C}{L'} \sum_k \frac{\mathbb{Q}(\Omega_k)}{A}  \left(\|h_{kl}-h_{kl}^0\|_1 - A^2 L' \frac{\eta K^2}{1 - \eta K} \|h_{kl}-h_{kl}^0\|_1\right) \geq  \frac{C}{L'} \frac{C'}{2} \sum_k  \|h_{kl} - h^0_{kl}\|_1 \\
    &\geq  C(f_0) (\|r - r_0\|_1 + \norm{h - h_0}_1), \quad C(f_0) =  \frac{C}{L'}\frac{C'}{2(K + \eta K^2 / (1 - \eta K))}.
\end{align*}

If there exists $ l \in [K]$ such that $|r_l^f - r_l^0| \geq \eta (\|r-f - r_0\|_1 + \norm{h - h_0}_1)$, we can use similar arguments as in the proof of Lemma 4 of \cite{donnet18:supplement}:
  $$  \Exf{Z_{1l}} \geq \Probf{\max_k N^k[\tau_1, \tau_1 + A] = 0} \times A |r_l^f - r_l^0|,$$
  and 
 \begin{align*}
  \Probf{\max_k N^k[\tau_1, \tau_1 + A] = 0}   &=  \mathbb{E}_{\mathbb{Q}}\left[\int_{\tau_1}^{\tau_1 + A} \mathcal{L}_t(f)  \mathds{1}_{\max \limits_k N^k[\tau_1, \tau_1 + A] = 0} dt \right]  = \mathbb{E}_{\mathbb{Q}}\left[\int_{\tau_1}^{\tau_1 + A} e^{A \|r\|_1}\mathds{1}_{\max \limits_k N^k[\tau_1, \tau_1 + A] = 0} dt \right] \\
        &\geq Ae^{A \|r_f\|_1} e^{-KA},
\end{align*}
so that 
 $$  \Exf{Z_{1l}} \geq C'(f_0)(\|r_f - r_0\|_1 + \norm{h - h_0}_1), \quad C'(f_0) = A^2 \eta e^{A \norm{r_0}_1 / 2} e^{-KA}.$$
We can conclude that in all cases,
\begin{equation*}
    \Exf{Z_{1l}} \geq \min(C(f_0), C'(f_0))(\|r_f - r_0\|_1 + \norm{h - h_0}_1),
\end{equation*}
\ds{and except in Case 2 of Proposition \ref{prop:relu},}
\begin{equation*}
    \Exf{Z_{1l}} \geq \min(C(f_0), C'(f_0), \frac{1}{L'},1) \norm{f - f_0}_1.
\end{equation*}
\end{proof}

\section{Additional results} \label{sec:results_costa}

\ds{In this section we recall some useful results on the regenerative properties of the nonlinear Hawkes model, which are mainly straightforward extensions of \cite{costa18_supplement} to our multivariate and general nonlinear setup. Besides, we recall the well-known Fuk-Nagaev's inequality.}

The first lemma is an extension of Theorem A.1 \cite{costa18_supplement} for a $M^K/G^K/\infty$ queue when the arrival process is the superposition of $K$ Poisson Point processes, corresponding to $K$ types of customers.

\begin{lemma}\label{lem:queue}
Consider a $M^K/G^K/\infty$ queue with $K$ types of customers that arrive according to a Poisson process with rate $r = (r_1,\dots, r_K)$. Assume that for each $k \in [K]$, the generic service time $H^k$ for a customer of type $k$ satisfies for some $\gamma > 0$ and for any $t \geq 0$:
\begin{equation*}
    \Prob{H^k \geq t} = o(e^{-\gamma t}).
\end{equation*}
Let $\mathcal{T}_1$ the first time of return of the queue to zero.
\begin{enumerate}
    \item If $\norm{r}_1 < \gamma$, then 
    $$\Prob{\mathcal{T}_1 \geq t} \leq \left[1 + \frac{\Ex{e^{\gamma B}}}{\gamma - \|r\|_1 }\right] e^{-\norm{r}_1 t},$$
where $B$ is the length of a busy period of the queue, i.e. $B = \mathcal{T}_1 - V_1$ with $V_1$ the arrival time of the first customer.    
    \item If $\gamma \leq \norm{r}_1$, then for any $0 < \alpha < \gamma$, $   \Prob{\mathcal{T}_1 \geq t} \leq c_1(\alpha) e^{-\alpha t}$,
with
\begin{align*}
    c_1(\alpha) &= \left[1 + \frac{\Ex{e^{\alpha B}}}{\|r\|_1 - \alpha}\right].
\end{align*}
\item $\forall \alpha \leq \norm{r}_1 \wedge \gamma, \: \Ex{e^{\alpha \mathcal{T}_1}} \leq \frac{\norm{r}_1}{\norm{r}_1 + s} \Ex{e^{\alpha B}}< + \infty$.
\end{enumerate}
\end{lemma}

\begin{proof}
In this situation, the arrival process of customers, \emph{regardless of their type}, is a superposition of $K$ Poisson processes with individual rate $r_k, \: k \in [K]$. Consequently, it is equivalent to a Poisson process with rate $\|r\|_1 = \sum_k r_k$. Moreover, the generic service time $H$ of a customer can be written as
   $H = \sum_k \delta_k H^k,$
with $\delta = (\delta_k)_{k \in [K]}$ a one-hot vector indicating the type of customer. We can easily see that
\begin{align*}
&\delta \sim \text{Mult}\left(1, \frac{r_1}{\|r\|_1},\dots, \frac{r_K}{\|r\|_1} \right),  \quad H | \delta \sim \delta \mathcal{P},
\end{align*}
with $\mathcal{P}$ the vector of service time distributions of the $K$ types of customers. We note that the service time $H$ is independent of the arrival process. Consequently, for $t\geq 0$,
    \begin{align*}
        \Prob{H \geq t} &= \sum_k \Prob{H^k \geq t, \: \delta_k = 1} \leq \sum_k \Prob{H^k \geq t}   = o(e^{-\gamma t}).
    \end{align*}
We can therefore conclude that this queue is equivalent to a $M/G/\infty$ queue with rate $\|r\|_1$ and generic service time satisfying $\Prob{H \geq t} = o(e^{-\gamma t})$. We can then apply Theorem A.1 in \cite{costa18_supplement} to obtain the results.
\end{proof}

\ds{The next lemma is a direct multivariate extension of the results in  Propositions 2.1 and 3.1 and Lemma 3.2 of \cite{costa18_supplement}. It introduces the mutually-exciting process dominating (in the sense of measure) a nonlinear Hawkes process.}

\begin{lemma}\label{lem:stoc_domination}
Let $Q$ be a $K$-dimensional Poisson point process on $(0,+\infty) \times (0,+\infty)^K$ with unit intensity. Let $N$ be the Hawkes process with immigration rate $\nu = (\nu_1, \dots, \nu_K), \: \nu_k > 0, \: k \in [K]$, interaction functions $h_{lk}: \mathbb{R}_+ \to \R, \: (l,k) \in [K]^2$ and initial measure $N_0$ on $[-A,0]$ driven by $(Q_t)_{t\geq 0}$ and satisfying one condition of Lemma \ref{lem:existence}. $N$ is the pathwise unique strong solution of the following system of stochastic equations
\begin{equation*}
\begin{cases}
N^k = N^k_0 + \int_{(0,+\infty) \times (0,+\infty)} \delta(u) \mathds{1}_{\theta \leq \lambda^k(u)} Q^k(du, d\theta), & \\
\lambda^k(u) = \phi_k \left(\nu_k + \sum_{l = 1}^K \int_{u-A}^u h_{l k} (u-s) dN^l_s \right), \: u>0, & k \in [K]
\end{cases}.
\end{equation*}
with $\delta(.)$ the Dirac delta function. Consider the similar equation for a point process $N^+$ in which $h_{lk}$ is replaced by $h^+_{lk}$ for any $l,k \in [K]^2$. Then 
\begin{enumerate}
\item there exists a pathwise unique strong solution $N$;
\item the same holds for $N^+$ and $N \leq N^+$ a.s. in the sense of measures. 
\end{enumerate}
This also implies that, with $\Delta \tau_1^+$ defined similarly to $\Delta \tau_1$ in \eqref{eq:def_tau} for the process $N^+$,
\begin{align*}
\Prob{\Delta \tau_1 \leq \Delta \tau_1^+} = 1.
\end{align*}
Moreover, with $\mathcal{T}_1$ defined as in Lemma \ref{lem:queue}, we also have $\Prob{\Delta \tau_1^+ = \mathcal{T}_1} = 1$.
\end{lemma}

Finally, the last proposition is the Fuk-Nagaev's inequality.

\begin{proposition}\label{fuk-nagaev}
Let $(X_i)_{i\geq 1}$ a sequence of independent and centered random variables with finite variance and $S_n = \sum_{i=1}^n X_i$. With $v = \sum_{i=1}^n \mathbb{V}(X_i)$, for any $x \geq 0$ and $\lambda \geq 0$, it holds that
\begin{align*}
\Prob{S_n \geq \lambda} \leq \sum_{i=1}^n \Prob{X_i > x} + \exp \left \{- \frac{v}{x^2} h\left(\frac{x \lambda}{v}\right) \right \},
\end{align*}
where $h(u) = (1+u) \log (1 + u) - u, \: u \geq 0$.
\end{proposition}

\bibliographystyle{plainnat} 
\bibliography{bib}       

\makeatletter\@input{m_aux.tex}\makeatother